\documentclass[aop,MSNbibl,citesort,dvips]{arximspdf}
\usepackage{mathrsfs}
\usepackage{graphicx}
\doi{10.1214/11-AOP725} 
\volume{41}
\issue{3A}
\pubyear{2013}
\firstpage{1243}
\lastpage{1314}

\makeatletter
\newtheorem{theorem}{Theorem}
\newtheorem{lemma}[theorem]{Lemma}
\newtheorem{proposition}[theorem]{Proposition}
\newtheorem{corollary}[theorem]{Corollary}
\newtheorem{cor}[theorem]{Corollary}
\newtheorem{prop}{Proposition}
\newtheorem{conj}[theorem]{Conjecture}

\newproclaim{definition}[theorem]{Definition}
\newproclaim{remark}[theorem]{Remark}

\newcommand{\re}{\operatorname{Re}}
\newcommand{\im}{\operatorname{Im}}
\newcommand{\tr}{\operatorname{tr}}
\def\C{\mathbb{C}}
\def\Z{\mathbb{Z}}
\def\R{\mathbb{R}}
\def\H{\mathcal{H}}
\def\U{\mathcal{U}}
\def\Ai{\operatorname{Ai}}
\newcommand{\Hi}[0]{\mathcal{H}}
\newcommand{\kapi}{\omega}
\def\v{v}
\newcommand{\e}{\varepsilon}
\newcommand{\ep}{\varepsilon}
\newcommand{\eps}{\varepsilon}
\newcommand{\dimple}{\varsigma}

\makeatother

\begin{document}
\begin{frontmatter}

\title{Crossover distributions at the edge of~the~rarefaction fan}
\runtitle{Distributions at the edge of the rarefaction fan}

\begin{aug}
\author[A]{\fnms{Ivan} \snm{Corwin}\corref{}\ead[label=e1]{corwin@cims.nyu.edu}\thanksref{t1}}
\and
\author[B]{\fnms{Jeremy} \snm{Quastel}\ead[label=e2]{quastel@math.utoronto.ca}\thanksref{t2}}
\thankstext{t1}{Supported by an NSF graduate research fellowship and by PIRE Grant OISE-07-30136.}
\thankstext{t2}{Supported by the Natural Science and Engineering Research Council of
Canada.}
\runauthor{I. Corwin and J. Quastel}
\affiliation{New York University and University of Toronto}
\address[A]{Courant Institute of Mathematical Sciences\\
New York University\\
251 Mercer Street\\
New York, New York 10012\\
USA\\
\printead{e1}} 
\address[B]{Department of Mathematics\\
University of Toronto\\
40 St. George street\\
Toronto, Ontario\\
Canada M5S 2E4\\
\printead{e2}}
\end{aug}

\received{\smonth{11} \syear{2010}}
\revised{\smonth{10} \syear{2011}}

%
\begin{abstract}
We consider the weakly asymmetric limit of simple exclusion process
with drift to the left, starting from step Bernoulli initial data with
$\rho_-<\rho_+$ so that macroscopically one has a rarefaction fan. We
study the fluctuations of the process observed along slopes in the fan,
which are given by the Hopf--Cole solution of the Kardar--Parisi--Zhang
(KPZ) equation, with appropriate initial data. For slopes strictly
inside the fan, the initial data is a Dirac delta function and the one
point distribution functions have been computed in [\textit{Comm. Pure Appl. Math.} \textbf{64} (2011)
466--537] and
[\textit{Nuclear Phys. B} \textbf{834} (2010) 523--542].
At the edge of the rarefaction fan, the initial data is \textit{one-sided
Brownian}. We obtain a new family of crossover distributions giving the
exact one-point distributions of this process, which converge, as
$T\nearrow\infty$ to those of the Airy $\mathcal{A}_{2\to \mathrm{BM}}$ process.
As an application, we prove moment and large deviation estimates for
the equilibrium Hopf--Cole solution of KPZ. These bounds rely on the
apparently new observation that the FKG inequality holds for the
stochastic heat equation. Finally, via a Feynman--Kac path integral,
the KPZ equation also governs the free energy of the continuum directed
polymer, and thus our formula may also be interpreted in those terms.
\end{abstract}

%

\end{frontmatter}

\section{Introduction}\label{sec1}
It is expected that a large class of one-dimensional, asymmetric,
stochastic, conservative interacting particle systems/growth models
fall into the Kardar--Parisi--Zhang (KPZ) universality class. A
manifestation of this is that the KPZ equation should appear as the
limit of such systems in the weakly asymmetric limit. The weakly
asymmetric limit means to observe the process on space scales of order
$\e^{-1}$ and time scales of order $\e^{-2}$, while simultaneously
rescaling the asymmetry of the model so that it is of order $\e^{1/2}$.
This sort of weak asymmetry zooms in on the critical transition point
between the two universality classes associated with growth models---the
KPZ class (positive asymmetry) and the Edwards Wilkinson (EW) class
(symmetry)---and thus further confirms a~mantra of statistical physics
that at critical points one expects universal scaling limits.

Bertini and Giacomin~\cite{BG} obtained the first result for the weakly
asymmetric simple exclusion process near equilibrium. This is extended
to some situations farther from equilibrium in~\cite{ACQ}, directed
random polymers in~\cite{AKQ} and partial results are now available
\cite{GJ} for speed changed asymmetric exclusion. In this article we
study the situation where asymmetry is to the left and the initial data
has an increasing step, so that in the hydrodynamic limit one sees a
rarefaction fan. We observe the process along a line $x=\v t$ within
the fan and study the fluctuations. These converge to the KPZ equation
with initial data depending on $\v$. For $\v$ strictly inside the fan,
the initial data is an appropriate scaling of a delta function, and the
distribution of the fluctuations is known exactly~\cite{ACQ,SaSp1}.
Our main interest in this article is the fluctuations at the
edge of the rarefaction fan. The scaling turns out to be a little
different, but the fluctuations are still given by KPZ. Note that since
the work of~\cite{BG} it is understood that KPZ is only a formal
equation for the fluctuation field and is rigorously defined as the
logarithm of the stochastic heat equation. The edge fluctuations
correspond to starting the stochastic heat equation with
\[
\exp\{ -B(X)\}\mathbf{1}_{X> 0},
\]
where $B(X)$ is a standard Brownian motion in $X$, with $B(0)=0$. We
will obtain an exact expression for the one-point probability
distribution of the resulting law---the \textit{edge crossover
distribution}---at any positive time. The main tool is the
Tracy--Widom determinantal formula for one-sided Bernoulli data, and
therefore we are restricted to asymmetric exclusion. The resulting law
is expected to be universal for fluctuations at the edge of the
rarefaction fan for models in the KPZ class.

\subsection{Height function fluctuations at the edge of the rarefaction
fan for ASEP}
The asymmetric simple exclusion process (ASEP) with parameters $p,q\geq
0$ (such that $p+q=1$) is a continuous time Markov process on the
discrete lattice $\Z$ with state space $\{0,1\}^{\Z}$ (the 1s are
thought of as particles and the 0s as holes). The dynamics for this
process are given as follows: Each particle has an independent
exponential alarmclock which rings at rate one. When the alarm goes
off, the particle flips a coin, and with probability $p$ attempts to
jump one site to the right, and with probability $q$ attempts to jump
one site to the left. If there is a particle at the destination, the
jump is suppressed, and the alarm is reset (see~\cite{Liggett} for a
rigorous construction of this process). If $q=1,p=0$ this process is
the totally asymmetric simple exclusion process (TASEP); if $q>p$ it is
the asymmetric simple exclusion process (ASEP); if $q=p$ it is the
symmetric simple exclusion process (SSEP). Finally, if we introduce a
parameter into the model, we can let $q-p$ go to zero with that
parameter, and then this class of processes is known as the weakly
asymmetric simple exclusion process (WASEP). It is the WASEP, that is, of
central interest in this paper since it interpolates between the SSEP
and ASEP, and is intimately connected with a stochastic partial
differential equation known as the KPZ equation. We denote the asymmetry
\[
\gamma=q-p.
\]
We consider the family of initial conditions for these exclusion
processes which are known of as \textit{two-sided Bernoulli} and which are
parametrized by densities $\rho_-,\rho_+\in[0,1]$. At time zero, each
site $x>0$ is occupied with probability $\rho_+$, and each site $x\leq
0$ is occupied with probability $\rho_-$ (all occupation random
variables are independent). These initial conditions interpolate
between the \textit{step} initial condition (where $\rho_-=0$ and
$\rho
_+=1$) and the \textit{equilibrium} or \textit{stationary} initial condition
(where $\rho_-=\rho_+=\rho$). We will focus on anti-shock initial
conditions where $\rho_-\leq\rho_+$.

Associated to an exclusion process are occupation variables $\eta(t,x)$
which equal~1 if there is a particle at position $x$ at time $t$ and 0
otherwise. From these we define spin variables $\hat{\eta}=2\eta-1$
which take values $\pm1$ and define the height function for WASEP with
asymmetry $\gamma=q-p$ by
\[
h_{\gamma}(t,x) =
\cases{
2N(t) + \displaystyle\sum_{0<y\leq x}\hat{\eta}(t,y), & \quad $x>0,$\vspace*{2pt}\cr
2N(t), &\quad $x=0,$\vspace*{4pt}\cr
2N(t)- \displaystyle\sum_{x<y\leq0}\hat{\eta}(t,y), & \quad $x<0,$}
\]
where $N(t)$ is equal to the net number of particles which crossed from
the site 1 to the site 0 in time $t$. Note that at time $t=0$,
$h_{\gamma}(0,x)$ is a two-sided simple random walk, with drift $2\rho
_+-1$ in the positive direction from the origin and drift $2\rho_--1$
in the negative direction.

\begin{proposition}[(Hydrodynamic limit)]\label{hydrolimitprop}
Let $\rho_-\leq\rho_+$, $\gamma=\e^{1/2}$ and $t=\e^{3/2}$. Then, in
probability,
%
%
\begin{eqnarray}\label{hydrodynamic_eqn}
&&\lim_{\e\rightarrow0} \frac{h_{\gamma}({t}/{\gamma},vt)}{t}
\nonumber
\\[-8pt]
\\[-8pt]
\nonumber
&&\qquad=
\cases{
2\rho_-(1-\rho_-)+(2\rho_- -1)v, & \quad $\mbox{for } v\leq2\rho_-
-1,$\vspace*{2pt}\cr
(1+v^2)/2, & \quad $\mbox{for } v\in[2\rho_- -1,2\rho_+ -1],$\vspace*{2pt}\cr
2\rho_+(1-\rho_+)+(2\rho_+ -1)v, & \quad $\mbox{for } v\geq2\rho_+ -1.$}
\end{eqnarray}
\end{proposition}

For $\gamma$ positive and not going to zero with $\e$, this result is
well known~\cite{Rez91,TS:1998h}. We were not able to find a reference
in the weakly asymmetric case. It is an easy consequence of the
fluctuation results (i.e., Theorem~\ref{WASEP_two_sided_thm}) which
make up the main contribution of this paper; see, however, Remark~\ref{cavaet}.

The region $v\in(2\rho_- -1,2\rho_+ -1)$ is the rarefaction fan, while
$v=2\rho_{\pm}-1$ is the edge of the fan. See Figure~\ref{hydro_fig}
for an illustration of this limit shape.

\begin{figure}

\includegraphics{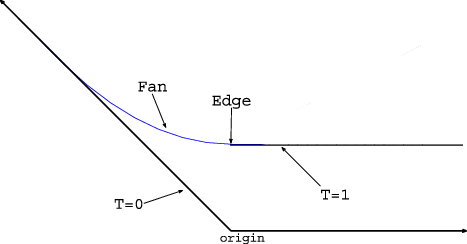}

\caption{Limit profile for the height function of ASEP with two-sided
Bernoulli initial conditions with $\rho_-=0$ and $\rho_+=1/2$. The
rarefaction fan corresponds to velocities $v<0$, and the edge to
velocity $v=0$.}\label{hydro_fig}
\end{figure}

For the purposes of this \hyperref[sec1]{Introduction} let us set $\rho_-=0$ and $\rho
_+=1/2$ so that the right edge of the rarefaction fan is at velocity
$v=0$. Around this velocity one sees a transition from a curved limit
shape for the height function to a flat limit shape. According to (\ref
{hydrodynamic_eqn}), the limit height is $t/2$.
\begin{definition}\label{defoffluc}
For $m\geq1$, $\e>0$, $T>0$ and $X_1,\ldots, X_m \in\R$ set
%
%
\begin{equation}\label{scales}
t=\e^{-3/2}T,\qquad x_k=2^{1/3}t^{2/3}X_k\quad \mbox{and}\quad \gamma=\e^{1/2}.
\end{equation}
Define the height fluctuation field $h^{\mathrm{fluc}}_{\gamma}(\frac
{t}{\gamma},x)$ by\setcounter{footnote}{2}\footnote{We attempt to use capital letters for all
variables (such as $X$, $T$) on the macroscopic level of the stochastic
PDEs and polymers. Lower case letters (such as $x$, $t$) will denote
WASEP variables, the microscopic discretization of these SPDEs.}
%
%
\begin{equation}\label{height_fluc_eqn}
h^{\mathrm{fluc}}_{\gamma}\biggl(\frac{t}{\gamma},x\biggr):=\frac{h_{\gamma
}({t}/{\gamma},x) - {t}/{2}}{t^{1/3}}.
\end{equation}
\end{definition}

Our first main result is a fluctuation theorem for the WASEP at the
edge of the rarefaction fan.

\begin{theorem}\label{hfluc_thm}
Let $\rho_-=0,\rho_+=1/2$ and $h^{\mathrm{fluc}}_{\gamma}(\frac
{t}{\gamma},x)$ be the edge fluctuation field defined in Definition~\ref{defoffluc}.
Then for $t,\gamma$ and $x$ as in (\ref{scales}), and $h
^{\mathrm{fluc}}$ as
in~(\ref{height_fluc_eqn}),
\[
\lim_{\e\rightarrow0} P\biggl(h^{\mathrm{fluc}}_{\gamma}\biggl(\frac
{t}{\gamma},x\biggr) \geq2^{-1/3}(X^2-s)\biggr) =F^{\mathrm{edge}}_{T,X}(s),\vadjust{\goodbreak}
\]
where the \textit{edge crossover distribution} $F^{\mathrm
{edge}}_{T,X}(s)$ is given by
%
%
\begin{equation}\label{eight}
F^{\mathrm{edge}}_{T,X}(s) = \int_{\tilde\mathcal{C}}e^{-\tilde\mu
}\frac{d\tilde\mu}{\tilde\mu}\det(I-K^{\mathrm
{edge}}_{s})_{L^2(\tilde
\Gamma_{\eta})}.
\end{equation}
The operator $K^{\mathrm{edge}}_{s}$, which depends on $T$ and $X$, and
the contour $\tilde\Gamma_{\eta}$, $\tilde\mathcal{C}$ are defined in
Definition~\ref{main_theorem_definition}.
Alternative formulas for the distribution function $F^{\mathrm
{edge}}_{T,X}(s)$ are given in Section~\ref{kernel_manip_sec}.
\end{theorem}

This theorem is proved in Section~\ref{TW_sec}. The proof uses the same
method as the proof of the main theorem of~\cite{ACQ} and relies upon a
recently discovered exact formula for the probability distribution for
the location of a fixed particle in ASEP with step Bernoulli initial
conditions (in our case $\rho_-=0$ and $\rho_+=1/2$). The main
technical modification is due to a new infinite product, which we call
$g(\zeta)$. Relating this to $q$-Gamma functions we are able to extract
the new asymptotic kernel which now also contains Gamma functions.

As we will see in Section~\ref{kpz_sec}, the $F^{\mathrm
{edge}}_{T,X}(s)$ distribution is also the one-point distribution for
the KPZ equation (\ref{KPZ0}) with specific initial data (\ref
{half_brownian}). It is clear from this result that time and space
scale differently. Specifically, the ratio of the scaling exponents for
time${}\dvtx{}$space${}\dvtx{}$fluctuations is $3\dvtx 2\dvtx 1$. This scaling ratio
was shown in~\cite{CFP10b} to hold for a wide class of $1+1$ dimensional growth
models. Finally, the $X^2$ shift with respect to the $s$ variable
reflects the parabolic curvature of the rarefaction fan nearby the edge.

The above result should be compared (see Section \ref
{height_func_two_sided_sec}) with the existing fluctuation theory for
TASEP and ASEP. In those cases, using the same centering and scaling as
in (\ref{height_fluc_eqn}),~\cite{BBP,BAC} and~\cite{TW4} (resp., for
TASEP and ASEP) obtained formulas for the one-point probability
distribution function. These formulas actually first arose in the study
of the largest eigenvalue of rank one perturbations of complex Wishart
random matrix ensembles~\cite{BBP}. Remarkably, the limiting
distributions are the same regardless of the asymmetry~$\gamma$, as
long as it is held positive as the other variables scale to infinity.
By scaling $\gamma$ as above, we focus in on the crossover between the
ASEP and the SSEP and the new family of edge crossover distribution
functions represent this transition.

For TASEP, Corwin, Ferrari and P\'{e}ch\'{e}~\cite{CFP} gave a formula
for the asymptotic equal time height function fluctuation process (in
terms of finite dimensional distributions). We paraphrase this as
Theorem~\ref{thm21}. For our present case, it says that if we fix
$\gamma=1$, $m\geq1$ (the case $m=1$ is just the~\cite{BBP,BAC} result
mentioned above), then for any choices of $T>0$, $X_1,\ldots, X_m\in
\R
$ and $s_1,\ldots, s_m\in\R$,
\[
\lim_{\e\rightarrow0} P\Biggl(\bigcap_{k=1}^{m}\biggl\{h_{\gamma}^{\mathrm
{fluc}}\biggl(\frac{t}{\gamma},x_k\biggr)\geq2^{-1/3}(X_k^2-s_k)\biggr\}\Biggr) = P
\Biggl(\bigcap_{k=1}^{m} \{\mathcal{A}_{2\to\mathrm{BM}}(X_k)\leq s_k\}\Biggr),
\]
where $\mathcal{A}_{2\to\mathrm{BM}}$ is a spatial process (defined below
in Definition~\ref{cor_and_conj_definition}) which interpolates between
the $\mathrm{Airy}_2$ process and Brownian motion.

Since in WASEP we scale the asymmetry with time and space in a critical
way, the fluctuation distributions are not the same as for TASEP or
ASEP. Rather than having $\gamma=\e^{1/2}$, one could perform
asymptotics with $\gamma=\alpha\e^{1/2}$. Doing this, it becomes
apparent that increasing $T$ is like increasing $\alpha$; hence, one
expects to recover the TASEP distributions from the WASEP edge
crossover distributions as $T\nearrow\infty$.

\begin{conj}\label{conj}
Let $\rho_-=0,\rho_+=1/2$, as well as $t,\gamma, x$ and $h^{\mathrm
{fluc}}_{\gamma}(\frac{t}{\gamma},x)$ be as in Definition \ref
{defoffluc}. Then for any $m\geq1$,
\begin{eqnarray*}
&&
\lim_{T\rightarrow\infty} \lim_{\e\rightarrow0} P\Biggl(\bigcap
_{k=1}^{m}\biggl\{h_{\gamma}^{\mathrm{fluc}}\biggl(\frac{t}{\gamma},x_k\biggr)\geq
2^{-1/3}(X_k^2-s_k)\biggr\}\Biggr)\\
&&\qquad = P\Biggl(\bigcap_{k=1}^{m} \{\mathcal
{A}_{2\to\mathrm{BM}}(X_k)\leq s_k\}\Biggr),
\end{eqnarray*}
where the joint distribution for the process $\mathcal{A}_{2\to
\mathrm{BM}}$ is given in Definition~\ref{cor_and_conj_definition} in
terms of
a Fredholm determinant.
\end{conj}

Extracting asymptotics from the result of Theorem~\ref{hfluc_thm}, we
are able to confirm this conjecture in the case of $m=1$ (see Section
\ref{BBP_proof_sec} for the proof):
\begin{cor}\label{cor1}
The $F_{T,X}^{\mathrm{edge}}$ distribution has a long time limit which
is given by
\[
\lim_{T\rightarrow\infty} F^{\mathrm{edge}}_{T,X}(s) = P\bigl(\mathcal
{A}_{2\to\mathrm{BM}}(X)\leq s\bigr),
\]
where the above one-point function for the $\mathcal{A}_{2\to\mathrm
{BM}}$ process coincides with the so-called BBP-transition~\cite{BBP} in
the study of perturbed Wishart random matrices and is given in
Definition~\ref{cor_and_conj_definition}.
\end{cor}

Thus by rescaling the edge crossover distribution for WASEP, we recover
the universal distribution at the edge of the rarefaction fan for TASEP
and ASEP. In the other direction we can also extract the small $T$
asymptotics, which are Gaussian. This is best stated in terms of the
stochastic heat equation so we delay it to Proposition~\ref{small_T_prop}.

%

\subsection{KPZ equation as the limit for WASEP height function
fluctuations at the edge}\label{kpz_sec}
Following the approach of~\cite{BG}, we prove that as $\e$ goes to
zero, a slight variant on the fluctuation field $h^{\mathrm
{fluc}}_{\gamma}(\frac{t}{\gamma},x)$ converges to the KPZ equation
with appropriate initial data.

The KPZ equation was introduced by Kardar, Parisi and Zhang in 1986 as
arguably the simplest stochastic PDE which contained\vadjust{\goodbreak} terms to account
for the desired behavior of one-dimensional interface growth~\cite{KPZ}.
%
%
\begin{equation}\label{KPZ0}
\partial_T \H= -\tfrac12(\partial_X \H)^2 + \tfrac12 \partial_X^2
\H+
\dot{\mathscr{W}},\vspace*{-1pt}
\end{equation}
where $\dot{\mathscr{W}}(T,X)$ is space--time white noise (see \cite
{ACQ}, Section 1.4, for a rigorous definition of white noise)
\[
E[\dot{\mathscr{W}}(T,X) \dot{\mathscr{W}}(S,Y)] = \delta
(T-S)\delta(Y-X).\vspace*{-1pt}
\]
Despite its simplicity, the KPZ equation has resisted analysis for
quite some time. The reason is that, even for nice initial data, the
solution at a later time $T>0$ will look locally like a Brownian motion
in $X$. Hence the nonlinear term is ill-defined.

In order to make sense of this KPZ equation, we follow~\cite{BG} and
simply define the \textit{Hopf--Cole solution} to the KPZ equation as
%
%
\begin{equation}\label{hc}
\H(T,X) = -\log\mathcal{Z} (T,X),\vspace*{-1pt}
\end{equation}
where $\mathcal{Z} (T,X)$ is the well-defined~\cite{W} solution of the
stochastic heat equation,
%
%
\begin{equation}\label{she}
\partial_T \mathcal{Z} = \tfrac12 \partial_X^2 \mathcal{Z} -
\mathcal{Z}
\dot{\mathscr{W}}.\vspace*{-1pt}
\end{equation}
Starting (\ref{KPZ0}) with initial data $\H(0,X)$ means starting
(\ref
{she}) with initial data $\mathcal{Z} (0,X)=\exp\{-\H(0,X)\}$. However,
one is best advised not to think in terms of $\H$ for the initial data
since here we will deal with initial data for $\mathcal{Z}$ (such as
Dirac-delta functions) which do not have a well-defined logarithm.

The stochastic partial differential equation (\ref{she}) is shorthand
for its integral version,
%
%
\begin{eqnarray}\label{16}
\mathcal{Z} (T,X) &=& \int_{-\infty}^\infty p(T,X-X_0) \mathcal{Z} (0,X)
\,dX_0
\nonumber
\\[-9pt]
\\[-9pt]
\nonumber
&&{}- \int_0^T \int_{-\infty}^\infty p(T-T_1,X-X_1) \mathcal{Z}
(T_1,X_1) \mathscr{W}(dX_1\,dT_1),\vspace*{-1pt}
\end{eqnarray}
where $p(T,X)=(2\pi T)^{-1/2}\exp\{-X^2/2T\}$ is the heat kernel.
Iterating, one obtains the chaos expansion (convergent in $\mathscr
{L}^2$ of the white noise $\mathscr{W}$)
%
%
\begin{equation}\label{soln}
\mathcal{Z}(T,X)= \sum_{n=0}^\infty(-1)^n I_n(T,X),\vspace*{-1pt}
\end{equation}
where
%
%
\begin{eqnarray}\label{17}I_n(T,X)&:=&
\int_{\Delta'_n(T)} \int_{\mathbb R^{n+1}}\prod_{i=0}^n
p(T_{i+1}-T_{i}, X_{i+1}-X_{i}) \mathcal{Z}(0,X_0)\,dX_0
\nonumber
\\[-9pt]
\\[-9pt]
\nonumber
&&\hspace*{54pt}{}\times\prod
_{i=1}^{n}\mathscr{W} (dT_i \,dX_i),
\\[-2pt]
\Delta'_n(T)&: =& \{(T_1,\ldots,T_n) \dvtx  0=T_0\le T_1\le\cdots\le
T_n\le
T_{n+1}=T\},\nonumber\vspace*{-1pt}
\end{eqnarray}
and $X_{n+1}=X$.\vadjust{\goodbreak}

\subsubsection{Microscopic Hopf--Cole transform}

We can now show that the WASEP height fluctuation field converges to
KPZ, in the sense that its Hopf--Cole transform converges to the
solution of the stochastic heat equation. This idea was first
implemented for equilibrium initial conditions in~\cite{BG} and is
facilitated by the fact that the Hopf--Cole transform of the
fluctuations (with lower order changes to the scalings) actually
satisfies a discrete space, continuous time stochastic heat equation
itself~\cite{G}. Specifically let
%
%
\begin{eqnarray}
\nu_\e&=& p+q-2\sqrt{qp} = \tfrac{1}{2} \e+ \tfrac{1}{8}\e^2 +
\mathcal{O}(\e^3),\\
\lambda_\e&=& \tfrac{1}{2}\log(q/p) = \e^{1/2} + \tfrac{1}{3}\e^{3/2}
+ \mathcal{O}(\e^{5/2}),
\end{eqnarray}
where we recall that with asymmetry $\gamma=\e^{1/2}$ we must have
$q=\frac{1}{2}+\frac{1}{2}\e^{1/2}$ and $p=\frac{1}{2}-\frac
{1}{2}\e^{1/2}$.

Define the random functions $Z_{\e}(T,X)$ by setting
%
%
\begin{equation}\label{scaledhgt}
Z_{\e}(T,X) = \exp\biggl\{-\lambda_{\e}h_{\gamma}\biggl(\frac{\e
^{-3/2}T}{\gamma
},\e^{-1}X\biggr) +\nu_{\e}\frac{\e^{-3/2}T}{\gamma}\biggr\}.
\end{equation}
Since $\rho_-=0,\rho_+=1/2$, $h_{\gamma}(0,x)$ is $|x|$ for $x\leq0$
and a simple symmetric random walk for $x>0$. Using this fact and the
Taylor approximation for $\lambda_{\e}\approx\e^{1/2}$, we find that
for $X<0$, $\lambda_{\e}h_{\gamma}(0,\e^{-1}X)$ is like $\e^{-1/2}X$,
and for $X\geq0$ it is converging to a standard Brownian motion
$B(X)$. Thus negating and exponentiating, we see that $Z_{\e}(0,X)$
converges to initial data $\mathbf{1}_{X\geq0}\exp\{-B(X)\}$.

\begin{definition} The solution of KPZ with half-Brownian initial data
is defined as
%
%
\begin{equation}\label{22}
\H^{\mathrm{edge}}(T,X):= -\log\mathcal{Z}^{\mathrm{edge}}(T,X),
\end{equation}
where $\mathcal{Z}^{\mathrm{edge}}(T,X)$ is the unique solution of the
stochastic heat equation
%
%
\begin{equation}
\partial_T\mathcal{Z}^{\mathrm{edge}}= \tfrac12\partial_X^2\mathcal
{Z}^{\mathrm{edge}}-\mathcal{Z}^{\mathrm{edge}}\dot{\mathscr{W}},\qquad
\label{half_brownian}
\mathcal{Z}^{\mathrm{edge}}(0,X)= \mathbf{1}_{X\geq0}\exp\{-B(X)\}.\hspace*{-35pt}
\end{equation}
The formal initial conditions for the (equally formal) KPZ equation
would be $\H^{\mathrm{edge}}(0,X)= B(X)$ for $X\geq0$ and
$\H^{\mathrm{edge}}(0,X)= -\infty$ for $X< 0$.
\end{definition}

Now observe that via the Taylor expansions of $\nu_\e$ and $\lambda
_\e
$, we have
\[
-\log Z_{\e}(T,X) = T^{1/3}h^{\mathrm{fluc}}_{\gamma}\biggl(\frac{\e
^{-3/2}T}{\gamma},\e^{-1}X\biggr) +\frac{T}{4!} + o(1).
\]
This suggests that
\[
\lim_{\e\rightarrow0} h^{\mathrm{fluc}}_{\gamma}\biggl(\frac{\e
^{-3/2}T}{\gamma},\e^{-1}X\biggr)= \frac{\H^{\mathrm{edge}}(T,X)-
{T}/{(4!)}}{T^{1/3}}.
\]
To state this precisely, observe that the random functions $Z_{\e
}(T,X)$ above have discontinuities both in
space and in time. If desired, one can linearly interpolate in space so\vadjust{\goodbreak}
that they become
a jump process taking values in the space of continuous functions. But
it does not really make things easier. The key point is that the jumps
are small, so we use instead the space $D_u([0,\infty); D_u(\mathbb
R))$, where $D$ refers to right continuous paths with left limits, and
$D_u(\R)$ indicates that in space these functions are equipped with the
topology of uniform convergence on compact sets. Let $\mathscr{P}_\e$
denote the probability measure on $D_u([0,\infty); D_u(\mathbb R))$
corresponding to the process $Z_{\e}(T,X)$.

\begin{theorem}\label{BG_thm}
$\mathscr{P}_\e$, $\ep\in(0,1/4)$, are a tight family of measures and
the unique limit point is supported on $C([0,\infty); C(\mathbb R))$
and corresponds to the solution of (\ref{she}) with initial conditions
(\ref{half_brownian}).
\end{theorem}

The proof of this theorem is a variation on that of~\cite{BG} and
\cite{ACQ} and is given in Section~\ref{BG_sec}.

Our second main result is a corollary of Theorems~\ref{hfluc_thm} and
\ref{BG_thm} and provides an exact formula for the one-point
distributions of KPZ with half-Brownian initial data.

\begin{cor}\label{KPZ_edge_cor}
For each fixed $T>0$, $X\in\R$ and $s\in\R$,
\[
P\biggl(\frac{\H^{\mathrm{edge}}(T,2^{1/3}T^{2/3}X)-
{T}/{(4!)}}{T^{1/3}}\geq2^{-1/3}(X^2-s)\biggr) = F^{\mathrm{edge}}_{T,X}(s),
\]
where $F^{\mathrm{edge}}_{T,X}(s)$ is given in Definition \ref
{main_theorem_definition}. As $T\nearrow\infty$ the above converges to
$P(\mathcal{A}_{2\to\mathrm{BM}}(X)\leq s)$; see Definition
\ref{cor_and_conj_definition}.
\end{cor}

Translating Conjecture~\ref{conj} into the KPZ language gives
\begin{conj}
For $m\geq1$, $T>0$, $X_1,\ldots,X_m\in\R$ and $s_1,\ldots,s_m\in
\R$,
\begin{eqnarray*}
&&\lim_{T\rightarrow\infty} P\Biggl(\bigcap_{k=1}^{m} \biggl\{\frac{\H
^{\mathrm{edge}}(T,2^{1/3}T^{2/3}X_k)-{T}/{(4!)}}{T^{1/3}} \geq
2^{-1/3}(X_k^2-s_k)\biggr\}\Biggr)\\
&&\qquad = P\Biggl(\bigcap_{k=1}^{m} \{\mathcal
{A}_{2\to\mathrm{BM}}(X_k)\leq s_k\}\Biggr).
\end{eqnarray*}
\end{conj}

We prove $m=1$ as the second statement of Corollary~\ref{KPZ_edge_cor}.

One expects~\cite{KK} that for large $y$,
%
%
\begin{equation}\label{tail}
1- F_{T,0}^{\mathrm{edge}} (y) \sim c e^{-(2/3) y^{3/2}},\qquad
F_{T,0}^{\mathrm{edge}} (-y) \sim c e^{-c y^3}.
\end{equation}
So far, we have not been able to obtain (\ref{tail}) from the
determinantal formula (\ref{eight}) for $F_{T,0}^{\mathrm{edge}}$. In fact,
using only the determinantal formulas, it is an open problem to show
that $F_{T,0}^{\mathrm{edge}}(y)\to0$ as $y\to-\infty$. We only
know it is
true because of Corollary~\ref{KPZ_edge_cor} together with (\ref{16}),
(\ref{17}),\vadjust{\goodbreak} (\ref{22}), which show that $\H^{\mathrm{edge}}(T,X)$
is a
nondegenerate random variable. The problem is that unlike the Airy
kernel used to define the Tracy--Widom distributions, the eigenvalues
of our crossover kernels are not in $[0,1]$.
We can obtain some asymptotics at the other end.

\begin{proposition}\label{LDPthm2}
There exists $c_1,c_2,c_3<\infty$ such that for $T\ge1$,
%
%
\begin{equation}\label{tail0}
1- F_{T,0}^{\mathrm{edge}} (2^{1/3} y -2^{1/3}T^{-1/3}\log2) \le
c_1T^{1/2}( e^{ - c_2 y^{3/2} } + e^{-c_3 T^{1/3} y}).
\end{equation}
\end{proposition}

The proof is in Section~\ref{upper_tail_sec}. The constants and
dependence on $T$ given here are not expected to be optimal. We remark
that using the same methods as in our proof, one can compute the upper
tail for $F_{T,X}^{\mathrm{edge}}$ and, with better constants, the upper
tail for the distribution $F^{\mathrm{fan}}_{T,X}$ (which on recalls
does not depend on $X$).

In terms of moments of
\[
\frac{\H^{\mathrm{edge}}(T,2^{1/3}T^{2/3}X)-{T}/{(4!)}}{T^{1/3}},
\]
from the convergence in law of the one point distribution, and the
general lower semicontinuity, one obtains lower bounds
\[
\liminf_{T\nearrow\infty}E\biggl[\biggl( \frac{\H^{\mathrm
{edge}}(T,2^{1/3}T^{2/3}X)-{T}/{(4!)}}{T^{1/3}}\biggr)^p\biggr] \ge C_p(X)>0,
\]
where $C_p(X)= E[(\mathcal{A}_{2\to\mathrm{BM}}(X))^p]$. The
corresponding upper bounds do not come as easily, though presumably
they could be derived by an appropriate asymptotic
analysis of the Tracy--Widom formulas for ASEP.

In the other direction we can also extract the small $T$ asymptotics.
Solving the regular heat equation $\partial_T \mathcal{Z} = \frac12
\partial_X^2 \mathcal{Z} $ with initial data (\ref{half_brownian}) gives
$\int_0^\infty\frac{
e^{ -{ (X-X_0)^2 }/{(2T) }}
}
{\sqrt{2\pi T} } \times e^{-B(X_0)} \,dX_0$.

\begin{proposition}\label{small_T_prop}
As $T\searrow0$,
\begin{eqnarray*}
\mathcal{Z}^{\mathrm{edge}}(T,X) &=& \int_0^\infty\frac{ e^{ -{
(X-X_0)^2 }/{(2T) }}}{\sqrt{2\pi T} } e^{-B(X_0)} \,dX_0 + T^{1/4} \mathcal
{Z}^{\mathrm{init}}(T,T^{-1/2}X) \\
&&{}+ o( T^{1/4}),
\end{eqnarray*}
where, given $B(X), X\ge0$, $\mathcal{Z}^{\mathrm{init}}(T,X)$ is a
Gaussian process with mean zero and covariance
\[
\operatorname{Cov} (\mathcal{Z}^{\mathrm{init}}(T,X), \mathcal
{Z}^{\mathrm{init}}(T,Y)) = \int_0^\infty\int_0^\infty\Psi(X,Y,X_0, X_0')
\,dX_0\,dX_0',
\]
where
\[
\Psi(x,y,x_0,x_0') =\frac{e^{- (1/4) (x+y-x_0-x_0')^2}}{4\pi^{3/2}}
\int_0^1
\frac{ e^{ -{ (x-y)^2 }/{(4(1-s))} - { (x_0-x_0')^2 }/{(4s)} }
}{\sqrt{s(1-s)} }\,ds.\vadjust{\goodbreak}
\]
\end{proposition}

The above proposition (proved in Section~\ref{small_t}) indicates that
the fluctuations scale and behave differently at the two extremes of
$T\nearrow\infty$ and $T\searrow0$. One could have started with an
asymmetry of $\gamma=b\e^{1/2}$. It turns out that the effect on the
limiting statistic of modulating this $b$ term is the same as
modulating $T$. Thus, as~$T$ goes to infinity, it is effectively like
(up to an interchange of limits) increasing the asymmetry away from the
weakly asymmetric range to the realm of positive asymmetry. This
explains the $T^{1/3}$ and $T^{2/3}$ scaling of fluctuations and space
in the large $T$ limit. On the other hand, taking $T$ to zero is like
moving toward symmetry, and this explains the $T^{1/4}$ and $T^{1/2}$
scalings of fluctuations and space in this limit. These two classes are
called the KPZ and EW universality classes. Thus we see that the KPZ
equation is, in fact, the universal mechanism for crossing between
these classes.

These results along with similar results of~\cite{BG,BQS,ACQ} and those contained in Section~\ref{full_fan_sec} below, provide
overwhelming evidence that the Hopf--Cole solution (\ref{hc}) to the
KPZ equation is the correct solution for modeling growth processes.
There do exist other interpretations of the KPZ equation, but they all
suffer from the fact that they lead to answers which do not yield the
desired scaling properties and limit distributions~\cite{TC}.

\begin{remark}
There is a Feynman--Kac formula for the solution of the stochastic heat equation
\[
\mathcal{Z}(T,X) = E_{T,X}\biggl[\mathcal{Z}(0,X) \dvtx  \exp\dvtx  \biggl\{-\int
_0^T \dot{\mathscr{W}}(t,b(t)) \,dt \biggr\}\biggr],
\]
where we make use of the Wick ordered exponential and where $E_{T,X}$
is the standard Wiener measure on $b(t)$ ending at position $X$ at time
$T$. This partition function is rigorously defined in terms of the
chaos expansion (\ref{soln}), or alternatively, as a~limit of mollified
versions of the white noise~\cite{BC}. Hence $\mathcal{Z}(T,X)$ has an
interpretation as
a partition function, and $-\H(T,X)$ an interpretation as a free
energy, of a continuum directed random polymer. These can be
shown to be universal limits of discrete polymers with rescaled
temperature~\cite{AKQ}. All of the results in this article have
alternative and immediate interpretations in terms of these polymer
models. This polymer perspective is more along the line of the approach
taken in~\cite{ACQ}.
\end{remark}

\subsection{Applications to KPZ in equilibrium}\label{stoch_dom_sec}
We now consider the Hopf--Cole solution of KPZ corresponding with
growth models with \textit{equilibrium} or \textit{stationary} initial
conditions. The initial data for KPZ is given~\cite{BG} by a two-sided
Brownian motion $B(X)$, $X\in\mathbb{R}$, or in other words, the
stochastic heat equation is given initial data
\[
\mathcal{Z}^{\mathrm{eq}}(0,X) = \exp\{-B(X)\}.
\]
As always, $B(X)$ is assumed independent of the space--time white noise
$\dot{\mathscr{W}}$. Strictly speaking, this is not an equilibrium
solution for KPZ because of global height shifts, but it is a genuine
equilibrium~\cite{BG} for the stochastic Burgers equation
\[
\partial_T \U= -\tfrac12 \partial_X \U^2 + \tfrac12 \partial_X^2
\U+
\partial_X\dot{\mathscr{W}},
\]
formally satisfied by its derivative $\U(T,X)=\partial_X \H^{\mathrm
{eq}}(T,X)$ where
\[
\H^{\mathrm{eq}}(T,X) = -\log\mathcal{Z}^{\mathrm{eq}}(T,X).
\]
In~\cite{BQS} it was shown that the variance of $\H^{\mathrm
{eq}}(T,0)$ is
of the correct order. In particular, there are constants $0<C_1\le
C_2<\infty$ such that
\[
C_1 T^{2/3}\le\operatorname{Var}( \H^{\mathrm{eq}}(T,0) ) \le C_2 T^{2/3}.
\]
At this time we are not able to obtain the distribution of $\H
^{\mathrm{eq}}(T,X) $ because the corresponding formulas of Tracy and
Widom \cite
{TW4} are not in the form of Fredholm determinants.
However, we will obtain some large deviation estimates and moment
bounds. The idea is to represent $\mathcal{Z}^{\mathrm{eq}}$ in terms of
solutions with half-Brownian initial data
\[
\mathcal{Z}^{\mathrm{eq}} = \mathcal{Z_+}+\mathcal{Z_-},
\]
where $\mathcal{Z_+}$ and $\mathcal{Z_-}$ solve the stochastic heat
equation with the same white noise and initial data
\[
\mathcal{Z_{\pm}}(0,X) = \mathbf{1}_{x\in\R_{\pm}}\exp\{-B(X)\}.
\]

Note that the two initial data are independent, and we have as an
additional tool the following
correlation inequality which is novel to our knowledge. At a heuristic
level it is clear that any two increasing functions of white noise
should be positively correlated. Using the Feynman--Kac (continuum
polymer) interpretation of the stochastic heat equation it is
physically clear that the solution is increasing in the white noise.

\begin{prop}[(FKG inequality for KPZ)]\label{FKGcor}
Let
$\mathcal
{Z}_1$, $\mathcal{Z}_2$ be two solutions of the stochastic heat
equation (\ref{she}) with the same white noise $\mathscr{\dot{W}}$, but
independent random initial data $\mathcal{Z}_1(0,X)$, $\mathcal
{Z}_2(0,X)$. We make the technical assumption that the solution to the
stochastic heat equation with initial data $\mathcal{Z}_1(0,X)$ and
$\mathcal{Z}_2(0,X)$ can be approximated, in the sense of process-level
convergence, by the rescaled exponential height functions (\ref
{scaledhgt}) for WASEP. Let $\H_i(T,X) = -\log\mathcal{Z}_i(0,X)$
denote the corresponding Hopf--Cole solutions of KPZ.
Then for any $T_1,T_2>0$, $X_1,X_2\in\R$ and $s_1,s_2\in\R$,
%
%
\begin{eqnarray}\label{a}
&& P\bigl(\mathcal{Z}_1(T_1,X_1)\leq s_1 \mbox{ and } \mathcal
{Z}_2(T_2,X_2)\leq s_2\bigr)
\nonumber
\\[-8pt]
\\[-8pt]
\nonumber
&&\qquad\geq P\bigl(\mathcal{Z}_1(T_1,X_1)\leq
s_1\bigr)P\bigl(\mathcal
{Z}_2(T_2,X_2)\leq s_2\bigr),
\end{eqnarray}
where $P$ denotes the probability with respect to the white noise as
well as the initial data. In particular, we have
%
%
\begin{eqnarray}\label{b}
&&P\bigl(\H_1(T_1,X_1)\geq s_1 \mbox{ and } \H_2(T_2,X_2)\geq s_2\bigr)
\nonumber
\\[-8pt]
\\[-8pt]
\nonumber
&&\qquad\geq
P\bigl(\H
_1(T_1,X_1)\geq s_1\bigr)P\bigl(\H_2(T_2,X_2)\geq s_2\bigr).
\end{eqnarray}
%
\end{prop}

This proof uses the FKG inequality at the level of a discrete system
which converges to the stochastic heat equation. We choose to use the
WASEP approximation for the stochastic heat equation explained in this
paper, though it would also be possible to prove this result via a
discrete polymer approximation.
The WASEP approximation assumption is not very restrictive. The work of
Bertini and Giacomin~\cite{BG}, Amir, Corwin and Quastel~\cite{ACQ} and
this paper show that a wide range of initial data fall into this class,
and one should be able to expand this even more. We also remark that
stronger forms of the above FKG inequality may be formulated and
similarly proved, though we do not pursue this further here.

\begin{pf*}{Proof of Proposition \protect\ref{FKGcor}}
By assumption we can approximate the relevant solutions to the
stochastic heat equation in terms of the WASEP as $Z_{1,\e}$ and
$Z_{2,\e}$. The graphical construction of ASEP can be thought of as a
priori setting an environment of attempted left and right jumps.
However, for our purposes we think of first throwing a Poisson point
process of attempted jumps and then assigning the jumps a direction
(left or right) independently with probability $q$ and $p$. There is a
natural monotonicity in this construction which says that changing a
right jump to a left jump will only increase the associated height
function. Taking the approximations for the initial data to be
independent of each other, this implies that the events $A_{i,\e}=\{
Z_{i,\e}\leq s_i\}$ are increasing events if one thinks of the Poisson
process of attempted jumps as giving a (random) lattice and the jump
directions as being~$1$ (left) or $-1$ (right). This jump lattice is
infinite; however, with probability one only a finite portion of it
affects the value of the two $Z_{i,\e}(T_i,X_i)$. Therefore, with
probability one the FKG inequality applies to this setting because of
the product structure of jump assignments on the attempted jump
lattice. Since the $A_{i,\e}$ are increasing events, they are
positively correlated. Taking the limit as $\varepsilon\to0$ gives the
desired result (\ref{a}) at the continuum level. Since $-\log$ is a
decreasing function, (\ref{b}) follows immediately from (\ref{a}).
\end{pf*}

\begin{proposition}\label{LDPthm}
For all $y\in\R$ and $T>0$,
%
%
\begin{eqnarray}\label{tail1}
&&\bigl(1- F_{T,0}^{\mathrm{edge}} (2^{1/3} y -2^{1/3}T^{-1/3}\log2)\bigr)^2
\nonumber\\
&&\qquad \le P\biggl( \H^{\mathrm{eq}}(T,0) -{\frac{T}{4!}} \le-
T^{1/3}y \biggr)
\\
&&\qquad \le 2\bigl(1- F_{T,0}^{\mathrm{edge}} (2^{1/3} y
-2^{1/3}T^{-1/3}\log2)\bigr),\nonumber
\end{eqnarray}
and
%
%
\begin{eqnarray}\label{tail2}
\qquad\hspace*{6pt}\bigl(F_{T,0}^{\mathrm{edge}} (-2^{1/3} y -2^{1/3}T^{-1/3}\log2)\bigr)^2 &
\le&P\biggl( \H^{\mathrm{eq}}(T,0) -{\frac{T}{4!}} \ge
T^{1/3} y
\biggr)
\nonumber
\\[-8pt]
\\[-8pt]
\nonumber
& \le& 2 F_{T,0}^{\mathrm{edge}}(-2^{1/3} y
-2^{1/3}T^{-1/3}\log2).
\end{eqnarray}
\end{proposition}

One can derive similar expressions to those above for other values of
$X\neq0$, though presently we do not state such results.

\begin{corollary}\label{LDPthm3}
There exist $c_1,c_2,c_3<\infty$ such that for $T>1$,
%
%
\begin{equation}\label{ldbd}
P\biggl( \H^{\mathrm{eq}}(T,0) -{\frac{T}{4!}} \le-T^{1/3} y
\biggr) \le c_1T^{1/2}(e^{ - c_2 y^{3/2} } + e^{-c_3 T^{1/3} y}).
\end{equation}
Furthermore, for each $p>0$, there exists $C_p>0$ such that for
sufficiently large~$T$,
%
%
\begin{equation}\label{varlb}
E\biggl[\biggl( \H^{\mathrm{eq}}(T,0) -{\frac{T}{4!}}
\biggr)^p\biggr] \ge C_pT^{p/3}.
\end{equation}
\end{corollary}

%
\begin{pf*}{Proof of Proposition \protect\ref{LDPthm}}
If
\[
\H^{\mathrm{eq}}(T,X)=-\log\mathcal{Z}^{\mathrm{eq}}(T,X), \qquad\H_\pm
(T,X)=-\log\mathcal{Z_\pm}(T,X),
\]
then using the increasing nature of the logarithm and the fact that
\[
2\min(\mathcal{Z_+},\mathcal{Z_-}) \leq\mathcal{Z_+}+\mathcal
{Z_-}\leq
2\max(\mathcal{Z_+},\mathcal{Z_-}),
\]
we have the following simple, yet significant inequality which
expresses the equilibrium solution to KPZ in terms of two coupled
half-Brownian solutions,\looseness=1
%
%
\begin{equation}\label{h_ineq}
-\log2 +\min(\H_+,\H_-) \leq\H^{\mathrm{eq}} \leq-\log2 +\max
(\H_+,\H_-).
\end{equation}\looseness=0
Thus
%
%
\begin{eqnarray}
&&P\biggl( \H^{\mathrm{eq}}(T,0) -{\frac{T}{4!}} \ge T^{1/3} y
\biggr)
\nonumber
\\[-8pt]
\\[-8pt]
\nonumber
&&\qquad \le P\biggl( \max(\H_+(T,0),\H_-(T,0)) -{\frac
{T}{4!}}\ge T^{1/3} y +\log2\biggr)
\\
\label{rt8}&&\qquad \le 2 P\biggl( \H_+(T,0) -{\frac{T}{4!}}\ge T^{1/3}
y+\log2 \biggr).
\end{eqnarray}
In (\ref{rt8}) we used that by symmetry, $\H_+(T,0)$ and $\H_-(T,0)$
have the same distribution, and the upper bound of (\ref{tail1})
follows by Corollary~\ref{KPZ_edge_cor}. For the lower bound in (\ref
{tail1}) we have, by (\ref{b}),
%
%
\begin{eqnarray}
&&P\biggl(\H^{\mathrm{eq}}(T,0) -{\frac{T}{4!}} \ge T^{1/3} y
\biggr)
\nonumber
\\[-8pt]
\\[-8pt]
\nonumber
&&\qquad \ge P\biggl( \min(\H_+(T,0),\H_-(T,0)) -{\frac
{T}{4!}}\ge T^{1/3} y +\log2\biggr) \\
&&\qquad \ge \biggl[P\biggl( \H_+(T,0) -{\frac{T}{4!}}\ge T^{1/3}
y+\log2 \biggr)\biggr]^2.
\end{eqnarray}
Equation (\ref{tail2}) is obtained from the lower bound of (\ref{h_ineq}) in
exactly the same way.
\end{pf*}

\begin{pf*}{Proof of Corollary \protect\ref{LDPthm3}}
The large deviation bound (\ref{ldbd}) follows from (\ref{tail1}) and
Proposition~\ref{LDPthm2}.
To prove (\ref{varlb}), suppose that $G$ and $F$ are probability
distribution functions satisfying
\[
1-G(x) \ge\bigl(1-F(c_2 x +c_3)\bigr)^2 \quad\mbox{and}\quad G(-x) \ge
F^2(-c_2 x +c_3)
\]
for all $x\in\mathbb{R}$ for some $c_1,c_2,c_3$.
We have the bound that
\begin{eqnarray*}
\int x^p \,dG(x)&=& p\int_0^\infty x^{p-1}\bigl(1-G(x) + G(-x)\bigr)\,dx \\
&\geq&2\int
_0^\infty x \bigl(1-F(c_2 x +c_3)\bigr)^2 +F^2(-c_2 x +c_3) \,dx.
\end{eqnarray*}
Hence the right-hand side of (\ref{varlb}) is bounded below by
\begin{eqnarray*}
&&p\int_0^\infty x^{p-1} \bigl\{\bigl(1- F_{T,0}^{\mathrm{edge}} (2^{1/3} x
-2^{1/3}T^{-1/3}\log2)\bigr)^2\\
&&\hspace*{34pt}\qquad{}+\bigl(F_{T,0}^{\mathrm{edge}} (-2^{1/3} x
-2^{1/3}T^{-1/3}\log2)\bigr)^2\bigr\} \,dx.
\end{eqnarray*}
By Fatou's lemma, the limit inferior as $T\nearrow\infty$ is greater
than the same integral with the distribution function
$F_{T,0}^{\mathrm{edge}}$ replaced by the distribution function of
$\mathcal
{A}_{2\to\mathrm{BM}}(0)$. Since the latter is strictly positive, this
gives (\ref{varlb}).~%
\end{pf*}
%
\subsection{Outline}
In the \hyperref[sec1]{Introduction} we have focused on the WASEP with $\rho_-=0$, $\rho
_+=1/2$ two-sided Bernoulli initial conditions and velocity $v=0$ so as
to be at the edge of the rarefaction fan. For those parameters we
described the height function fluctuations for WASEP, the link to the
solution of the KPZ equation with specific initial data, and then the
fluctuation theory for that solution. In Section~\ref{full_fan_sec} we
explain the situation for general values of $\rho_-\leq\rho_+$ and $v$
either inside the rarefaction fan or at the edge; see Remark~\ref
{cavaet}. In Section~\ref{full_fan_sec} we explain how the connection
between WASEP and KPZ generalizes as well. Section~\ref{definition_sec}
contains all of the important definitions of kernels, contours and
Airy-like processes used in the paper. Section~\ref{TW_sec} contains a
heuristic and then rigorous\vadjust{\goodbreak} proof of Theorem~\ref{hfluc_thm}. Section
\ref{BG_sec} contains a proof that the WASEP converges to the KPZ
equation (the claimed results of Section~\ref{kpz_sec}). Finally,
Section~\ref{manipSec} contains a proof of the large $T$ asymptotics of
the KPZ equation, as well as other tail and short time asymptotics of
the edge crossover distributions.

\section{The fluctuation characterization for two-sided Bernoulli WASEP
and KPZ}\label{full_fan_sec}
In the \hyperref[sec1]{Introduction} we focused on a particular choice of two-sided
Bernoulli initial conditions where $\rho_-=0$ and $\rho_+=1/2$. By
looking at velocity $v=0$ (which corresponds to the edge of the
rarefaction fan) we uncovered a~new family of edge crossover
distributions and showed that the fluctuation process near the edge
converges to the Hopf--Cole solution to the KPZ equation with
half-Brownian initial data.

In this section we consider what happens for other choices of $\rho
_-\leq\rho_+$ and~$v$. Theorem~\ref{WASEP_two_sided_thm}, the main
result of this section, shows that under the same sort of scaling as
present for TASEP (\cite{CFP} or Theorem~\ref{thm21} below), the WASEP
height function fluctuations converge to three different crossover
distributions:
\begin{enumerate}[(1)]
\item[(1)] the \textit{fan crossover distributions} for fluctuations around
velocities $v\in(2\rho_--1, 2\rho_+-1)$;
\item[(2)] the \textit{edge crossover distributions} for fluctuations around
velocities $v= 2\rho_\pm-1$;
\item[(3)] the \textit{equilibrium crossover distributions} for $\rho
_-=\rho
_+=\rho$ and fluctuations around the characteristic $v=2\rho-1$.
\end{enumerate}
The three cases above also correspond to the three different possible
KPZ limits of the WASEP height function fluctuations. The stochastic
heat equation initial data in the three cases are:
\begin{enumerate}[(1)]
\item[(1)]$\mathcal{Z}^{\mathrm{fan}}(0,X)=\delta_{X=0}$;
\item[(2)]$\mathcal{Z}^{\mathrm{edge}}(0,X)=\mathbf{1}_{X \geq0}\exp\{
-B(X)\}$;
\item[(3)]$\mathcal{Z}^{\mathrm{eq}}(0,X)=\exp\{-B(X)\}.$
\end{enumerate}
We similarly label the associated continuum height function $\H
(T,X)= -\log\mathcal{Z}(T,X)$ as $\H^{\mathrm{fan}}(T,X)$, $\H
^{\mathrm
{edge}}(T,X)$ and $\H^{\mathrm{eq}}(T,X)$.

\subsection{Height function fluctuations for two-sided Bernoulli
initial conditions}\label{height_func_two_sided_sec}
Before considering the WASEP height function fluctuations for general
values of $\rho_-<\rho_+$ and $v$, it is worth reviewing the analogous
theory developed in~\cite{J,PS02,FS,BAC,CFP} for the TASEP. For the
TASEP, Pr\"{a}hofer and Spohn~\cite{PS02} conjectured (based on
existing results~\cite{BaikRains} for the related PNG model) a
characterization of the fluctuations of the height function for TASEP
started with two-sided Bernoulli initial conditions. They identified
three different limiting one-point distribution functions which,
depending on the region's location, that is, whether the velocity $v$
is chosen so as to be: (1)~within the rarefaction fan; (2) at its edge;
(3) in\vadjust{\goodbreak} equilibrium ($\rho_-=\rho_+=\rho$), at the characteristic speed
$v=2\rho-1$. This conjecture was proved in~\cite{FS,BAC}. An analogous
theory for the multiple-point limit distribution structure was proved
in~\cite{BFP,CFP} and again only depended on regions
(1)--(3).\looseness=-1

Let us now paraphrase\vspace*{1pt} Theorem 2.1 of~\cite{CFP} (which also includes
the main result of~\cite{BFP}).
Define a positive velocity version of $h_{\gamma}^{\mathrm
{fluc}}(\frac{t}{\gamma},x)$ as
\[
h_{\gamma,v}^{\mathrm{fluc}}\biggl(\frac{t}{\gamma},x\biggr):= \frac{
h_{\gamma
}({t}/{\gamma},vt+ x(1-v^2)^{1/3}) - (({(1+v^2)}/{2}) t +
x(1-v^2)^{1/3}v)}{(1-v^2)^{2/3}t^{1/3}}.
\]
The $(1-v^2)$ scaling term is the only new element and simply reflects
the change of reference frame due to the nonzero velocity. When $v=0$
we recover $h_{\gamma}^{\mathrm{fluc}}$ from~(\ref{height_fluc_eqn}).

\begin{theorem}[(\cite{CFP}, Theorem 2.1(a), paraphrased)] \label{thm21}
Fix $\gamma=1$ (TASEP), $m>0$ and $\rho_-,\rho_+\in[0,1]$ such that
$\rho_-\leq\rho_+$. Then for any choices of $T>0$, $X_1,\ldots
,X_m\in
\R$ and $s_1,\ldots,s_m\in\R$, if we set $t=\e^{-3/2}T$ and
$x_k=2^{1/3}t^{2/3}X_k$:
\begin{longlist}
\item[(1)] for $v\in(2\rho_--1,2\rho_+-1)$,
%
%
\begin{equation}\label{above1}
\lim_{\e\rightarrow0}P\Biggl(\bigcap_{k=1}^{m}\{h_{\gamma,v}^{\mathrm
{fluc}}({t}/{\gamma},x) \geq2^{-1/3}(X_k^2-s_k)\}\Biggr) = P
\Biggl(\bigcap_{k=1}^{m} \{\mathcal{A}_{2}(X_k)\leq s_k\}\Biggr);\hspace*{-35pt}
\end{equation}
\item[(2)] for $\rho_-<\rho_+$ and $v=2\rho_\pm-1$,
%
%
\begin{eqnarray}\label{above2}
&&\lim_{\e\rightarrow0}P\Biggl(\bigcap_{k=1}^{m}\{h_{\gamma,v}^{\mathrm
{fluc}}({t}/{\gamma},x) \geq2^{-1/3}(X_k^2-s_k)\}\Biggr)
\nonumber
\\[-8pt]
\\[-8pt]
\nonumber
&&\qquad = P
\Biggl(\bigcap_{k=1}^{m} \{\mathcal{A}_{2\rightarrow\mathrm{BM}}(\pm
X_k)\leq
s_k\}\Biggr);
\end{eqnarray}
\item[(3)] for $\rho_-=\rho_+=\rho$ and $v=2\rho-1$,
%
%
\begin{eqnarray}\label{above3}
&&\lim_{\e\rightarrow0}P\Biggl(\bigcap_{k=1}^{m}\{h_{\gamma,v}^{\mathrm
{fluc}}({t}/{\gamma},x) \geq2^{-1/3}(X_k^2-s_k)\}\Biggr)
\nonumber
\\[-8pt]
\\[-8pt]
\nonumber
&&\qquad= P
\Biggl(\bigcap_{k=1}^{m} \{\mathcal{A}_{\mathrm{eq}}(X_k)\leq s_k+X_k^2\}\Biggr).
\end{eqnarray}
\end{longlist}
\end{theorem}

The definitions of the three limit processes are given in Definition
\ref{cor_and_conj_definition} of Section~\ref{definition_sec}. The
$X_k^2$ in the $\mathcal{A}_{\mathrm{eq}}$ probability reflects a
shift of
that amount which was already present in the definition of that process
originally given in~\cite{BFP} (in that paper this process was called
$\mathcal{A}_{\mathrm{stat}}$, though for our purposes we find it more
informative to call it $\mathcal{A}_{\mathrm{eq}}$).

For ASEP, where $\gamma$ is less than one but still strictly positive,
significantly less is known rigorously.\vadjust{\goodbreak} The only case where results
analogous to Theorem~\ref{thm21} have been proved is $\rho_-=0, m=1$;
though, given those results of~\cite{TW3,TW4}, it is certainly
reasonable to conjecture that Theorem~\ref{thm21} holds for all
$\gamma>0$.

As we saw in the \hyperref[sec1]{Introduction}, there exists a critical scaling for
$\gamma$ going to zero at which there arises new limits for the height
function fluctuations which correspond to the Hopf--Cole solution to
the KPZ equation. In that WASEP scaling we then have the following
theorem, very much analogous to Theorem~\ref{thm21}.

%
%
\begin{theorem}\label{WASEP_two_sided_thm}
Fix\vspace*{1pt} $\gamma=\e^{1/2}$ and $0\leq\rho_-\leq\rho_+\leq1$. Then for any
choices of $T>0$, $X\in\R$ and $s\in\R$ if we set $t=\e^{-3/2}T$ and
$x=2^{1/3}t^{2/3}X$:
\begin{longlist}
\item[(1)] For $v\in(2\rho_--1,2\rho_+-1)$,
\[
\lim_{\e\rightarrow0} P\bigl(h_{\gamma,v}^{\mathrm{fluc}}(
{t}/{\gamma},x) \geq2^{-1/3}(L_\e+X^2-s)\bigr) =F^{\mathrm
{fan}}_{(1-v^2)^2 T,X}(s)
\]
(see Definition~\ref{ACQ_formula_def}) where $L_\e=
2^{-1/3}T^{-1/3}\log(\e^{-1/2}/c)$ with $c=(2\rho
_+-1-v)^{-1}+(v-2\rho
_-+1)^{-1}$.
As $T\nearrow\infty$ we recover the right-hand side of (\ref{above1}),
$P(\mathcal{A}_{2}(X_k)\leq s)$.

\item[(2)] For $v=2\rho_\pm-1$,
\[
\lim_{\e\rightarrow0} P\bigl(h_{\gamma,v}^{\mathrm{fluc}}(
{t}/{\gamma},x) \geq2^{-1/3}(X^2-s)\bigr) =F^{\mathrm
{edge}}_{(1-v^2)^2 T,X}(s)
\]
(see Definition~\ref{main_theorem_definition}). As $T\nearrow\infty$ we
recover the right-hand side of (\ref{above2}), $P(\mathcal
{A}_{2\rightarrow\mathrm{BM}}(\pm X)\leq s)$.

\item[(3)] For $\rho_-=\rho_+=\rho$ and $v=2\rho-1$,
\[
\lim_{\e\rightarrow0} P\bigl(h_{\gamma,v}^{\mathrm{fluc}}(
{t}/{\gamma},x) \geq2^{-1/3}(X^2-s)\bigr) =F^{\mathrm{eq}}_{(1-v^2)^2 T,X}(s)
\]
for which presently no formula exists; see, however, Section \ref
{stoch_dom_sec} for some bounds. As $T\nearrow\infty$ we should recover
the right-hand side of (\ref{above3}).
\end{longlist}
\end{theorem}

The case $\rho_-=0$, $\rho_+=1$ and $v=0$ above was previously solved
in~\cite{ACQ} and~\cite{SaSp1,SaSp2,SaSp3}. It should be noted that
$F^{\mathrm{fan}}_{T,X}$ does not, in fact, depend on $X$. The
logarithmic correction $L_\e$ in the case of the fan is unique to that
case and does not have a~parallel in the analogous TASEP result.

\begin{remark}\label{cavaet}
The above theorem can be proved in two ways. For $\rho_-=0$ one can use
the formulas of~\cite{TW4} directly and extract asymptotics (as we do
for the case $\rho_+=1/2$ in Section~\ref{TW_sec}). Aside from some
small technical modifications to the proof, the only change is that for
$v\neq0$ one must center the asymptotic analysis around the point
\[
\xi= -\biggl(\frac{1+v}{1-v}\biggr) - \frac{2X}{T(1-v)^2}\e^{1/2}.
\]
In order to arrive at the full (i.e., also $\rho_->0$) theorem above,
one cannot appeal to exact formulas\vadjust{\goodbreak} because those that exist~\cite{TW5}
are not in a form for which it is known how to do asymptotics. Instead
we will prove that the height function fluctuation process converges to
the solution to the KPZ equation with initial data depending only on
the region (fan, edge or equilibrium along the characteristic). In the
case of the fan and the edge, we have exact formulas for the one-point
distributions for these KPZ solutions, in the case of the fan, obtained
from the special case $\rho_-=0,\rho_+=1$, and in the case of the edge,
obtained from the special case $\rho_-=0, \rho_+=1/2$. The equilibrium
result also follows in the same way, despite not having a formula for
$F^{\mathrm{eq}}$. We will give the general velocity version of
Theorem~\ref{BG_thm} in a forthcoming paper. Thus, strictly speaking, in this
paper we only prove the above theorem (and likewise Proposition \ref
{hydrolimitprop}) for (a) $\rho_-=0$, and general $\rho_+$ and $v$
(using asymptotic analysis of formulas of~\cite{TW4}); (b) general
$\rho
_-$ and $\rho_+$, yet $v=0$ (combining Theorem~\ref{BG_thm}\ or the
analogous theorems of~\cite{ACQ} and~\cite{BG}, along with the exact
statistics results determined herein or in~\cite{ACQ} for the KPZ
equation itself).
\end{remark}

%


%



\subsection{Kernel and contour definitions}\label{definition_sec}
Here we collect the definitions of the kernels and contours used in the
statement of the main results of this paper.\looseness=-1

\begin{definition}\label{main_theorem_definition}
The \textit{edge crossover distribution} is defined as
\[
F^{\mathrm{edge}}_{T,X}(s) = \int_{\tilde\mathcal{C}}e^{-\tilde\mu
}\frac{d\tilde\mu}{\tilde\mu}\det(I-K^{\mathrm
{edge}}_{s})_{L^2(\tilde
\Gamma_{\eta})}.
\]

The contour $\tilde\mathcal{C}$ is given as
\[
\tilde\mathcal{C}=\{e^{i\theta}\}_{\pi/2\leq\theta\leq3\pi/2}
\cup\{
x\pm i\}_{x>0}.
\]
The contours $\tilde\Gamma_{\eta}$, $\tilde\Gamma_{\zeta}$ are
given as
%
%
\begin{eqnarray}
\tilde\Gamma_{\eta}&=&\biggl\{\frac{c_3}{2}+ir\dvtx  r\in(-\infty,-1)\cup
(1,\infty
)\biggr\}\cup\tilde\Gamma_{\eta}^d,\\
\tilde\Gamma_{\zeta}&=&\biggl\{-\frac{c_3}{2}+ir\dvtx  r\in(-\infty,-1)\cup
(1,\infty)\biggr\}\cup\tilde\Gamma_{\eta}^d,
\end{eqnarray}
where $\tilde\Gamma_{\zeta}^d$ is a dimple which goes to the\vspace*{-1pt} right of
$XT^{-1/3}$ and joins with the rest of the contour, and where $\tilde
\Gamma_{\eta}^d$ is the same contour just shifted to the right by
distance~$c_3$; see Figure~\ref{new_gamma_contours}. The constant $c_3$
is defined henceforth as
\[
c_3=2^{-4/3}.
\]
The kernel $K_s^{\mathrm{edge}}$ acts on the function space
$L^2(\tilde
\Gamma_{\eta})$ through%
\begin{figure}

\includegraphics{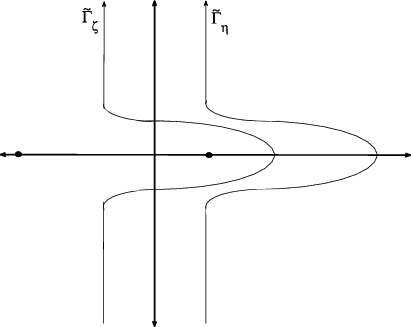}

\caption{The contours $\tilde\Gamma_{\zeta}$ and $\tilde\Gamma
_{\eta}$
extend vertically along the lines $c_3/2$ and $-c_{3}/2$ except for
a~dimple to the right of $XT^{-1/3}$ so as to avoid the poles of the
Gamma function (marked as solid black circles).}\label{new_gamma_contours}
\end{figure}
its kernel,
%
%
\begin{eqnarray}
 K_s^{\mathrm{edge}}(\tilde\eta,\tilde\eta') &=&\int_{\tilde\Gamma_{\zeta}} \exp\biggl\{-\frac
{T}{3}(\tilde
\zeta^3-\tilde\eta'^3)+sT^{1/3}(\tilde\zeta-\tilde\eta')\biggr\}
2^{1/3}
\nonumber
\\[-8pt]
\\[-8pt]
\nonumber
&&\hspace*{4pt}\quad{}\times\int_{-\infty}^{\infty} \frac{\tilde\mu e^{-2^{1/3}t(\tilde\zeta
-\tilde
\eta')}}{e^{t}-\tilde\mu}\,dt\frac{\Gamma(2^{1/3}\tilde\zeta-
2^{1/3}XT^{-1/3})}{\Gamma(2^{1/3}\tilde\eta' - 2^{1/3}XT^{-1/3})}
\frac
{d\tilde\zeta}{\tilde\zeta-\tilde\eta},
\end{eqnarray}
or, equivalently, evaluating the $t$ integral,
%
%
\begin{eqnarray}
 K_s^{\mathrm{edge}}(\tilde\eta,\tilde\eta') &=&\int
_{\tilde\Gamma_{\zeta}} \exp\biggl\{-\frac{T}{3}(\tilde\zeta^3-\tilde
\eta
'^3)+sT^{1/3}(\tilde\zeta-\tilde\eta')\biggr\}
\nonumber
\\[-8pt]
\\[-8pt]
\nonumber
 &&\hspace*{4pt}\quad{}\times \frac{\pi2^{1/3}(-\tilde
\mu
)^{-2^{1/3}(\tilde\zeta-\tilde\eta')}}{\sin(\pi2^{1/3}(\tilde
\zeta
-\tilde\eta'))}\frac{\Gamma(2^{1/3}\tilde\zeta-
2^{1/3}XT^{-1/3})}{\Gamma(2^{1/3}\tilde\eta' - 2^{1/3}XT^{-1/3})}
\frac
{d\tilde\zeta}{\tilde\zeta-\tilde\eta}.
\end{eqnarray}
$\Gamma(z)$ is the standard Gamma function, defined for $\re(z)>0$ by
$\Gamma(z) = \int_0^\infty s^{z-1}\times e^{-s}\,ds$ and
extended by analytic continuation to $\mathbb{C} -\{0,-1,-2,\ldots\}$.
To follow previous works we continue to use the letter gamma for
contours, but always
with a subscript to differentiate them from the Gamma function.
\end{definition}

\begin{definition}\label{ACQ_formula_def}
The \textit{fan crossover distribution} is defined as
\[
F^{\mathrm{fan}}_{T,X}(s) = \int_{\tilde\mathcal{C}}e^{-\tilde\mu
}\frac{d\tilde\mu}{\tilde\mu}\det(I-K^{\mathrm
{fan}}_{s})_{L^2(\tilde
\Gamma_{\eta})}.
\]
The kernel $K_s^{\mathrm{fan}}$ acts on the function space $L^2(\tilde
\Gamma_{\eta})$ through its kernel,
\begin{eqnarray*}
K_s^{\mathrm{fan}}(\tilde\eta,\tilde\eta') &=& \int_{\tilde\Gamma
_{\zeta
}} \exp\biggl\{-\frac{T}{3}(\tilde\zeta^3-\tilde\eta
'^3)+sT^{1/3}(\tilde\zeta
-\tilde\eta')\biggr\} 2^{1/3}\\
&&\hspace*{4pt}\quad{}\times \biggl (\int_{-\infty}^{\infty} \frac{\tilde\mu
e^{-2^{1/3}t(\tilde\zeta-\tilde\eta')}}{e^{t}-\tilde\mu}\,dt\biggr) \frac
{d\tilde\zeta}{\tilde\zeta-\tilde\eta},
\end{eqnarray*}
or, evaluating the inner integral, equivalently,
\begin{eqnarray*}
K_s^{\mathrm{fan}}(\tilde\eta,\tilde\eta') &=& \int_{\tilde\Gamma
_{\zeta
}} \exp\biggl\{-\frac{T}{3}(\tilde\zeta^3-\tilde\eta
'^3)+sT^{1/3}(\tilde\zeta
-\tilde\eta')\biggr\} \\
&&\quad\hspace*{4pt}{}\times\frac{\pi2^{1/3}(-\tilde\mu)^{-2^{1/3}(\tilde
\zeta
-\tilde\eta')}}{\sin(\pi2^{1/3}(\tilde\zeta-\tilde\eta'))}
\frac
{d\tilde\zeta}{\tilde\zeta-\tilde\eta}.
\end{eqnarray*}
As far as the choice of contours $\tilde\Gamma_{\zeta}$ and $\tilde
\Gamma_{\eta}$, one can use the same as above without the extra dimple
(since there are no poles to avoid now). This distribution is closely
related to the \textit{crossover distribution} of~\cite{ACQ}, though one
observes that the scalings are slightly different. As we now have a
whole class of crossover distributions we find it useful to name them
more descriptively.\looseness=1
\end{definition}

\begin{definition}\label{cor_and_conj_definition}
The Airy$_2$ process is defined in terms of finite dimensional
distributions as
\[
P\Biggl(\bigcap_{k=1}^{m}\{\mathcal{A}_2(X_k)\leq s_k\}\Biggr)=\det
(I-\chi_s K_{\mathcal{A}_2} \chi_s)_{L^2(\{X_1,\ldots, X_m\}\times
\R)},
\]
where $\chi_s(X_k,x) = \mathbf{1}_{[x>s_k]}$, and $K_{\mathcal{A}_2}$
is the extended Airy kernel
\[
K_{\mathcal{A}_2}(X,x;X',y)=
\cases{
{ \displaystyle\int_{0}^{\infty} \,dt\, e^{(X'-X)t}\Ai(t+x)\Ai(t+y)}, &
\quad $X\geq X',$\vspace*{2pt}\cr
{ -\displaystyle\int_{-\infty}^0 \,dt\, e^{(X'-X)t}\Ai(t+x)\Ai(t+y)}, & \quad $X<X'.$}
\]
The Airy$_2$ process was discovered in the PNG model~\cite{PS}. It is a
stationary process with one-point distribution given by the GUE
Tracy--Widom distribution $F_{\mathrm{GUE}}$~\cite{TW0}. An integral
representation of $K_{\mathcal{A}_2}$ can be found in Proposition~2.3
of~\cite{JDTASEP}; another form is in Definition 21 of~\cite{BFS09} in
the $M=0$ case.

We denote by $\mathcal{A}_{2\to\mathrm{BM}}$ the transition process from
Airy$_2$ to Brownian motion. It is defined in terms of finite
dimensional distributions as
\[
P\Biggl(\bigcap_{k=1}^{m}\{\mathcal{A}_{2\to\mathrm{BM}}(X_k)\leq s_k\}
\Biggr)=\det(I-\chi_s K_{\mathcal{A}_{\mathrm{BM}\to2}} \chi_s)_{L^2(\{
X_1,\ldots, X_m\}\times\R)},
\]
where $K_{\mathcal{A}_{2\to\mathrm{BM}}}$ is the rank-one perturbation
$K_{\mathcal{A}_2}$,
\begin{eqnarray*}
&&K_{\mathcal{A}_{2\to\mathrm{BM}}}(X,x;X',y)\\
&&\qquad=K_{\mathcal
{A}_2}(X,x;X',y)+\Ai(x)\biggl(-e^{(1/3)X'^3+X'y}-\int_{0}^\infty \,dt\, \Ai
(t+y)e^{-X't}\biggr).
\end{eqnarray*}
This transition process was derived in~\cite{SI04} and was shown to
arise in TASEP at the edge of the rarefaction fan in~\cite{CFP}. An
integral representation of the kernel can be found in~\cite{BFS09},
Definition 21, in the $m=1$ case. For $m=1$ this formula corresponds
with the distribution $F_1$ of~\cite{BBP}.

The definition of the process for equilibrium TASEP, $\mathcal
{A}_{\mathrm{eq}}$, is quite intricate and is given in~\cite{BFP}
where it is
actually called the $\mathcal{A}_{\mathrm{stat}}$ process. Its joint
distributions is the right-hand side of equation (1.9) in~\cite{BFP}.
\end{definition}

\section{Weakly asymmetric limit of the Tracy--Widom step Bernoulli
ASEP formula}\label{TW_sec}

In this section we will prove our main result, Theorem~\ref{hfluc_thm}.
The proof of that theorem follows by combining the proof of the main
theorem of~\cite{ACQ} (for step initial condition WASEP) with a few
lemmas to cover a new element [the $g(\zeta)$ term stated below] which
shows up for step Bernoulli initial conditions. As noted in the
\hyperref[sec1]{Introduction}, the key technical tool behind this proof is the
\textit{exact} formula for the transition probability of a single particle in
ASEP with step Bernoulli initial data~\cite{TW4}.

\begin{theorem}[(Main results of~\cite{TW4})]\label{TWmainthm}
Let $q>p$ with $q+p=1$, $\gamma=q-p$ and $\tau=p/q$. Fix $\rho_-=0$,
and for $\rho_+ \in(0,1]$, set
\[
\alpha= (1-\rho_+)/{\rho_+}.
\]
Since $\rho_-=0$ we can initially label our particles $1,2,3,\ldots$ by
setting the leftmost to be particle 1 and the second left most to be
particle 2, and so on. Let $\mathbf{x}(t,m)$ denote the location of
particle $m$ at time $t$. Then for $m>0$, $t\geq0$ and $x\in\Z$,
\cite
{TW4} gives the following exact formula:
%
%
\begin{equation}\label{TW_prob_equation}
P\bigl(\mathbf{x}(\gamma^{-1}t,m)\leq x\bigr) = \int_{S_{\tau^+}}\frac{d\mu
}{\mu}
\prod_{k=0}^{\infty} (1-\mu\tau^k)\det(I+\mu J^{\Gamma}_{\mu
})_{L^2(\Gamma_{\eta})},
\end{equation}
where $S_{\tau^+}$ is a circle centered at zero of radius strictly
between $\tau$ and 1, and where the kernel of the determinant is given by
%
%
\begin{equation}\label{J_eqn_def}
J^{\Gamma}_{\mu}(\eta,\eta')=\int_{\Gamma_{\zeta}} \exp\{\Psi
(\zeta
)-\Psi(\eta')\}\frac{f(\mu,\zeta/\eta')}{\eta'(\zeta-\eta
)}\frac{g(\eta
')}{g(\zeta)}\,d\zeta,
\end{equation}
where
%
%
\begin{eqnarray}
 f(\mu,z)&=&\sum_{k=-\infty}^{\infty} \frac{\tau
^k}{1-\tau^k\mu
}z^k,\nonumber\\
\Psi(\zeta) &=& \Lambda(\zeta)-\Lambda(\xi),\\
\Lambda(\zeta) &=& -x\log(1-\zeta) + \frac{t\zeta
}{1-\zeta
}+m\log\zeta,\nonumber\\
g(\zeta) &=& \prod_{n=0}^{\infty}( 1+\tau^n \alpha\zeta).
\end{eqnarray}
The contours are a little tricky: $\eta$ and $\eta'$ are on $\Gamma
_{\eta}$, a circle of diameter\footnote{Following~\cite{TW5}, this
means the circle is symmetric about the real axis and intersects it at
$-\alpha^{-1}+2\delta$ and $1-\delta$.} $[-\alpha^{-1}+2\delta
,1-\delta
]$ for $\delta$ small. And the $\zeta$ integral is on $\Gamma_{\zeta}$,
a circle of diameter $[-\alpha^{-1} +\dimple, 1+\dimple]$. One should
choose $\dimple$ so as to ensure that $|\zeta/\eta|\in(1,\tau^{-1})$.
This choice of contour avoids the poles of the new infinite product
which are at $-\alpha^{-1}\tau^{-n}$ for $n\geq0$. Of course we can
take $\delta$ to depend on $\e$.
\end{theorem}

\subsection{Heuristic explanation of the asymptotics of the
Tracy--Widom formula}\label{heuristic_sec}
We start by restating the result and give a heuristic explanation of
the proof. In Section~\ref{thm_proof_sec} we will give a complete proof
of these asymptotics, roughly following the method of proof of Theorem
1 of~\cite{ACQ}.

The following theorem uses different scalings so as to conform to the
notation of~\cite{ACQ}. Therefore, the resulting formula differs and we
introduce the kernel $K_a^{\csc,\Gamma}$. From the following result, by
careful scaling, one arrives at Theorem~\ref{hfluc_thm}.

\begin{theorem}[(Equivalent to Theorem~\ref{hfluc_thm} after
rescaling)]\label{thm_proof_thm}
Consider $\e>0$, $T>0$ and $X\in\R$ and set $\rho_-=0$, $\rho_+=1/2$,
$t=\e^{-3/2}T$, $x=\e^{-1}X$ and $\gamma=\e^{1/2}$. Then
\begin{eqnarray*}
\lim_{\e\rightarrow0} P\biggl(\e^{1/2}\biggl[h_{\gamma}\biggl(\frac{t}{\gamma},x\biggr)
- \frac{t}{2}\biggr] \geq-s\biggr) &=& \lim_{\e\rightarrow0} P\biggl(\mathbf
{x}_{\gamma}\biggl(\frac{t}{\gamma},m\biggr)\leq x\biggr)\\
 &=& \int_{\tilde C}e^{-\tilde
\mu
} \det(I-K_{a}^{\csc,\Gamma})_{L^2(\tilde\Gamma_{\eta})}\frac
{d\tilde\mu
}{\tilde\mu},
\end{eqnarray*}
where $a=a(s)=s+\frac{X^2}{2T}$,
\[
m=\frac{1}{2}\biggl[\e^{-1/2}\biggl(-a+\frac{X^2}{2T}\biggr) + \frac
{t}{2}+x\biggr]
\]
and the operator $K_a^{\csc,\Gamma}$ acts on the function space
$L^2(\tilde\Gamma_{\eta})$ through its kernel,
%
%
\begin{eqnarray}\label{KGammaDef}
K_a^{\csc,\Gamma}(\tilde\eta,\tilde\eta') &= &\int_{\tilde\Gamma
_{\zeta
}} \exp\biggl\{-\frac{T}{3}(\tilde\zeta^3-\tilde\eta
'^3)+2^{1/3}a(\tilde\zeta
-\tilde\eta')\biggr\} 2^{1/3}
\nonumber
\\[-8pt]
\\[-8pt]
\nonumber
&&\quad\hspace*{4pt}{}\times \biggl(\int_{-\infty}^{\infty} \frac{\tilde\mu
e^{-2^{1/3}t(\tilde\zeta-\tilde\eta')}}{e^{t}-\tilde\mu}\,dt\biggr)\frac
{\Gamma(2^{1/3}\tilde\zeta- {X}/{T})}{\Gamma(2^{1/3}\tilde
\eta' -
{X}/{T})} \frac{d\tilde\zeta}{\tilde\zeta-\tilde\eta}.
\end{eqnarray}
The contours $\tilde\mathcal{C}$, $\tilde\Gamma_{\zeta}$ and
$\tilde
\Gamma_{\eta}$ are defined in Definition~\ref{main_theorem_definition},
though for the last two contours, the dimples are modified to go to the
right of the poles of the Gamma function above (the rightmost of which
lies at $2^{-1/3}X/T$).
\end{theorem}

We now proceed with the heuristic proof of the above result. Note that
given the values of $\rho_-$ and $\rho_+$, the parameter $\alpha$
defined above in Theorem~\ref{TWmainthm} is equal to 1. We will,
however, keep $\alpha$ in the calculations since one can then see
readily how to generalize to $\alpha\neq1$.

The first term in the integrand of (\ref{TW_prob_equation}) is the
infinite product $\prod_{k=0}^{\infty}(1-\mu\tau^k)$. Observe that
$\tau=q/p\approx1-2\e^{1/2}$ and that $S_{\tau^+}$, the contour on
which $\mu$ lies, is a circle centered at zero of radius between $\tau$
and 1. The infinite product is not well behaved along most of this
contour; however, we can deform the contour to one along which the
product is not highly oscillatory. Care must be taken, however, since
the Fredholm determinant has poles at every $\mu=\tau^k$. The\vadjust{\goodbreak}
deformation must avoid passing through them. As in~\cite{ACQ} observe
that if
%
%
\begin{equation}\label{mutilde_formula}
\mu= \e^{1/2}\tilde\mu,
\end{equation}
then
\[
\prod_{k=0}^{\infty}(1-\mu\tau^k) \approx e^{-\sum_{k=0}^{\infty}
\mu
\tau^k} = e^{-\mu/(1-\tau)} \approx e^{-\tilde\mu/2}.
\]

We make the $\mu\mapsto\e^{-1/2}\tilde\mu$ change of variables and
find that if we consider a $\tilde\mu$ contour
\[
\tilde\mathcal{C}_{\e} = \{e^{i\theta}\}_{\pi/2\leq\theta\leq
3\pi
/2}\cup\{x\pm i\}_{0<x\leq\e^{-1/2}-1}\cup\{\e^{-1/2}-1+iy\}_{-1<y<1},
\]
then the above approximations are reasonable. Thus the infinite product
goes to $\exp\{-\tilde\mu/2\}$.

Now we turn to the Fredholm determinant and determine a candidate for
the pointwise limit of the kernel.
The kernel $J^{\Gamma}_{\mu}(\eta,\eta')$ is given by an integral whose
integrand has four main components: an exponential
\[
\exp\{\Lambda(\zeta)-\Lambda(\eta')\};
\]
a rational function (we include the differential with this term for
scaling purposes)
\[
{d\zeta}/{\eta'(\zeta-\eta)};
\]
a doubly infinite sum
\[
\mu f(\mu,\zeta/\eta');
\]
an infinite product
\[
{g(\eta')}/{g(\zeta)}.
\]

We proceed by the method of steepest descent, so in order to determine
the region along the $\zeta$ and $\eta$ contours which affects the
asymptotics, we must consider the exponential term first. The argument
of the exponential is given by $\Lambda(\zeta)-\Lambda(\eta')$ where
\[
\Lambda(\zeta)=-x\log(1-\zeta) + \frac{t\zeta}{1-\zeta}+m\log
(\zeta),
\]
where $x$, $t$ and $m$ are as in Theorem~\ref{thm_proof_thm}. For small
$\e$, $\Lambda(\zeta)$ has a critical point in an $\e^{1/2}$
neighborhood of $-1$. For purposes of having a nice ultimate answer, we
choose to center in on the point
\[
\xi=-1-2\e^{1/2}\frac{X}{T}.
\]
We can rewrite the argument of the exponential as $(\Lambda(\zeta
)-\Lambda(\xi))-(\Lambda(\eta')-\Lambda(\xi))=\Psi(\zeta)-\Psi
(\eta')$.
The idea of extracting asymptotics for this term (which starts like
those done in~\cite{TW3} but quickly becomes more involved due to the
fact that $\tau$ tends to 1 as $\e$ goes to zero) is then to deform the
$\zeta$ and $\eta$ contours to lie along curves such that outside the
scale $\e^{1/2}$ around~$\xi$, $\re\Psi(\zeta)$ is very negative, and
$\re\Psi(\eta')$ is very positive, and hence the\vadjust{\goodbreak} contribution from
those parts of the contours is negligible. Rescaling around $\xi$ to
blow up this $\e^{1/2}$ scale, gives us the asymptotic exponential
term. This change of variables sets the scale at which we should
analyze the other three terms in the integrand for the $J$
kernel.

Returning to $\Psi(\zeta)$, we make a Taylor expansion around $\xi$ and
find that in a~neighborhood of $\xi$,
\[
\Psi(\zeta) \approx-\frac{T}{48} \e^{-3/2}(\zeta-\xi)^3 + \frac
{a}{2}\e^{-1/2}(\zeta-\xi).
\]
This suggests the following change of variables:
%
%
\begin{eqnarray}\label{change_of_var_eqn}
\tilde\zeta&=& 2^{-4/3}\e^{-1/2}(\zeta-\xi),\qquad \tilde\eta=
2^{-4/3}\e^{-1/2}(\eta-\xi),
\nonumber
\\[-8pt]
\\[-8pt]
\nonumber
 \tilde\eta'& =& 2^{-4/3}\e^{-1/2}(\eta
'-\xi),
\end{eqnarray}
after which our Taylor expansion takes the form
\[
\Psi(\tilde\zeta) \approx-\frac{T}{3} \tilde\zeta^3
+2^{1/3}a\tilde
\zeta.
\]
In the spirit of steepest descent analysis we would like the $\zeta$
contour to leave $\xi$ in a direction where this Taylor expansion is
decreasing rapidly. This is accomplished by leaving at an angle $\pm
2\pi/3$. Likewise, since $\Psi(\eta)$ should increase rapidly, $\eta$
should leave $\xi$ at angle $\pm\pi/3$. Since $\rho_+=1/2$, $\alpha=1$
which means that the $\zeta$ contour is originally on a circle of
diameter $[-1+\delta,1+\delta]$ and the $\eta$ contour on a circle of
diameter $[-1+2\delta,1-\delta]$ for some positive $\delta$ [which can
and should depend on $\e$ so as to ensure that $|\zeta/\eta|\in
(1,\tau
^{-1})$]. In order to deform these contours to their steepest descent
contours without changing the value of the determinant, great care must
be taken to avoid the poles of $f$, which occur whenever $\zeta/\eta
'=\tau^k$, $k\in\Z$, and the poles of $1/g$, which occur whenever
$\zeta=-\tau^{-n}$, $n\geq0$. We will ignore these considerations in
the formal calculation but will take them up more carefully in Section
\ref{thm_proof_sec}. The one very important consideration in this
deformation, even formally, is that we must end up with contours which
lie to the right of the poles of the $1/g$ function.

Let us now assume that we can deform our contours to curves along which
$\Psi$ rapidly decays in $\zeta$ and increases in $\eta$, as we move
along them away from~$\xi$. If we apply the change of variables in
(\ref
{change_of_var_eqn}) the straight part of our contours become infinite
rays at angles $\pm2\pi/3$ and $\pm\pi/3$ which we call $\tilde
\Gamma
_{\zeta}$ and $\tilde\Gamma_{\eta}$. Note that this is \textit
{not} the
actual definition of these contours which we use in the statement and
proof of Theorem~\ref{hfluc_thm} because of the singularity problem
mentioned above.

Applying this change of variables to the kernel of the Fredholm
determinant changes the $L^2$ space, and hence we must multiply the
kernel by the Jacobian term $2^{4/3}\e^{1/2}$. We will include this
term with the $\mu f(\mu,z)$ term and take the $\e\to0$ limit of
that product.

Before we consider that term, however, it is worth looking at the new
infinite product term $g(\eta')/g(\zeta')$. In order to do that let us
consider the following. Set
\[
q=1-r,\qquad a= \frac{\log\alpha(c-xr)}{\log q},\qquad b= \frac{\log
\alpha(c-yr)}{\log q}.
\]
Then observe that
%
%
\begin{eqnarray}\label{qgamma_eqn}
&&\prod_{n=0}^\infty\frac{1+(1-r)^n \alpha(-c+xr)}{1+(1-r)^n
\alpha(-c+yr)}\nonumber\\
&&\qquad= \frac{(q^a;q)_\infty}{(q^b;q)_\infty}
=\frac{\Gamma_q(b)}{\Gamma_q(a)} (1-q)^{b-a} = \frac{\Gamma_q(b)}{\Gamma_q(a)} e^{(b-a)\log r}
\\
&&\qquad= \frac{\Gamma_{1-r}(-r^{-1} \log(\alpha c) + c^{-1} y +
o(r))}{\Gamma
_{1-r}(-r^{-1} \log(\alpha c) + c^{-1} x+o(r))} e^{ c^{-1} (y-x) \log
r +o(r\log r) },\nonumber
\end{eqnarray}
where the $q$-Gamma function and the $q$-Pochhammer symbols are given by
\[
\Gamma_{q}(x):=\frac{(q;q)_{\infty}}{(q^x;q)_{\infty}}(1-q)^{1-x}
\]
when $|q|<1$ and
\[
(a;q)_{\infty} = (1-a)(1-aq)(1-aq^2)\cdots.
\]
The notation $o(f(r))$ above refers to a function $f'(r)$ such that
$f'(r)/f(r)\to0$ as $r\to0$. The $q$-Gamma function converges to the
usual Gamma function as $q\to1$, uniformly on compact sets; see \cite
{AAR} for more details and a statement of this result.

Now consider the $g$ terms and observe that in the rescaled variables
this corresponds with (\ref{qgamma_eqn}) with $r=2\e^{1/2}$, $c=1$
(recall $\alpha=1$ as well) and
\[
y=2^{1/3}\tilde\zeta- \frac{X}{T}, \qquad x=2^{1/3}\tilde\eta' - \frac{X}{T}.
\]
Since $\alpha c =1$ and since we are away from the poles and zeros of
the Gamma functions, we find that
%
%
\begin{equation}\label{gfinaleqn}
\frac{g(\eta')}{g(\zeta)} \rightarrow\frac{\Gamma(2^{1/3}\tilde
\zeta-{X}/{T})}{\Gamma(2^{1/3}\tilde\eta'-
{X}/{T})} \exp\{2^{1/3}(\tilde\zeta-\tilde\eta')\log(2\e
^{1/2})\}.
\end{equation}
This exponential can be rewritten as
%
%
\begin{equation}\label{exponential_term}
\exp\biggl\{\frac{\tilde z}{4} \log\e\biggr\} \exp\{2^{1/3}\log
(2)(\tilde\zeta-\tilde\eta')\},
\end{equation}
where
%
%
\begin{equation}\tilde z=2^{4/3}(\tilde\zeta-\tilde\eta').
\end{equation}
It appears that there is a problem in these asymptotics as $\e$ goes to
zero; however, we will find that this apparent divergence exactly
cancels with a similar term in the doubly infinite summation term
asymptotics. We will now show how that $\log\e$ in the exponent can be
absorbed into the $2^{4/3}\e^{1/2}\mu f(\mu,\zeta/\eta')$ term.
Recall
\[
\mu f(\mu,z) = \sum_{k=-\infty}^{\infty} \frac{\mu\tau^k}{1-\tau
^k \mu}z^k.
\]
If we let $n_0=\lfloor\log(\e^{-1/2}) /\log(\tau)\rfloor$, then
observe that
\[
\mu f(\mu,z) = \sum_{k=-\infty}^{\infty} \frac{ \mu\tau
^{k+n_0}}{1-\tau^{k+n_0}\mu}z^{k+n_0} =z^{n_0} \tau^{n_0}\mu\sum
_{k=-\infty}^{\infty} \frac{ \tau^{k}}{1-\tau^{k}\tau^{n_0}\mu}z^{k}.
\]
By the choice of $n_0$, $\tau^{n_0}\approx\e^{-1/2}$, so
\[
\mu f(\mu,z) \approx z^{n_0} \tilde\mu f(\tilde\mu,z).
\]
The discussion on the exponential term indicates that it suffices to
understand the behavior of this function only in the region where
$\zeta
$ and $\eta'$ are within a neighborhood of $\xi$ of order $\e^{1/2}$.
Equivalently, letting $z=\zeta/\eta'$, it suffices to understand
$\mu f(\mu,z) \approx z^{n_0} \tilde\mu f(\tilde\mu,z)$ for
\[
z= \frac{\zeta}{\eta'}=\frac{\xi+ 2^{4/3}\e^{1/2}\tilde\zeta
}{\xi+
2^{4/3}\e^{1/2}\tilde\eta'}\approx1-\e^{1/2}\tilde z.
\]
Let us now consider $z^{n_0}$ using the fact that $\log\tau\approx
-2\e^{1/2}$.
\[
z^{n_0} \approx(1-\e^{1/2}\tilde z)^{\e^{-1/2}({1}/{4})\log\e}
\approx e^{-({1}/{4})\tilde z \log\e}.
\]
Plugging back in the value of $\tilde z$ in terms of $\tilde\zeta$ and
$\tilde\eta'$ we see that this prefactor of~$z^{n_0}$ exactly cancels
the $\log\e$ term which came from the $g$ infinite product term.

What remains is to determine the limit of $2^{4/3}\e^{1/2}\tilde\mu
f(\tilde\mu, z)$ as $\e$ goes to zero and for $z\approx1-\e^{1/2}
\tilde z$. This limit can be found by interpreting the infinite sum as
a Riemann sum approximation for an appropriate integral. Define $t=k\e
^{1/2}$, then observe that
\[
\e^{1/2}\tilde\mu f(\tilde\mu,z) = \sum_{k=-\infty}^{\infty}
\frac{
\tilde\mu\tau^{t\e^{-1/2}}z^{t\e^{-1/2}}}{1-\tilde\mu\tau^{t\e
^{-1/2}}}\e^{1/2} \rightarrow\int_{-\infty}^{\infty} \frac{\tilde
\mu
e^{-2t}e^{-\tilde z t}}{1-\tilde\mu e^{-2t}}\,dt.
\]
This used the fact that $\tau^{t\e^{-1/2}}\rightarrow e^{-2t}$ and that
$z^{t\e^{-1/2}}\rightarrow e^{-\tilde z t}$, which hold at least
pointwise in $t$. If we change variables of $t$ to $t/2$ and multiply
the top and bottom by $e^{-t}$, then we find that
\[
2^{4/3}\e^{1/2}\mu f(\mu,\zeta/\eta') \rightarrow2^{1/3} \int
_{-\infty
}^{\infty} \frac{\tilde\mu e^{-\tilde zt/2}}{e^{t}-\tilde\mu}\,dt.
\]
As far as the final term, the rational expression, under the change of
variables and zooming in on $\xi$, the factor of $1/\eta'$ goes to $-1$
and the $\frac{d\zeta}{\zeta-\eta'}$ goes to $\frac{d\tilde\zeta
}{\tilde
\zeta-\tilde\eta'}$.

Therefore we formally find the following kernel: $-K_{a'}^{\csc,\Gamma
}(\tilde\eta,\tilde\eta')$ acting on $L^2(\tilde\Gamma_{\eta})$, where
\begin{eqnarray*}
K_{a'}^{\csc,\Gamma}(\tilde\eta,\tilde\eta')& = &\int_{\tilde
\Gamma
_{\zeta}} \exp\biggl\{-\frac{T}{3}(\tilde\zeta^3-\tilde\eta
'^3)+2^{1/3}a'(\tilde\zeta-\tilde\eta')\biggr\} 2^{1/3} \\
&&\quad\hspace*{4pt}{}\times\biggl(\int_{-\infty
}^{\infty} \frac{\tilde\mu e^{-2^{1/3}t(\tilde\zeta-\tilde\eta
')}}{e^{t}-\tilde\mu}\,dt\biggr)\frac{\Gamma(2^{1/3}\tilde\zeta-
{X}/{T})}{\Gamma(2^{1/3}\tilde\eta'-{X}/{T})} \frac
{d\tilde\zeta}{\tilde\zeta-\tilde\eta},
\end{eqnarray*}
where $a'=a+\log2$ [recall that this $\log2$ came from (\ref
{exponential_term})].

We have the identity
%
%
\begin{equation}\label{cscIntegral}
\int_{-\infty}^{\infty} \frac{\tilde\mu e^{-\tilde
zt/2}}{e^{t}-\tilde
\mu}\,dt =(-\tilde\mu)^{-\tilde z/2}\pi\csc(\pi\tilde z/2),
\end{equation}
where the branch cut in $\tilde\mu$ is along the positive real axis,
hence $(-\tilde\mu)^{-\tilde z/2} =e^{-\log(-\tilde\mu)\tilde z/2}$
where $\log$ is taken with the standard branch cut along the negative
real axis.
We may use the identity to rewrite the kernel as
%
%
\begin{eqnarray}\label{kcscgamma}
K_{a'}^{\csc,\Gamma}(\tilde\eta,\tilde\eta') &=& \int_{\tilde
\Gamma
_{\zeta}} \exp\biggl\{-\frac{T}{3}(\tilde\zeta^3-\tilde\eta
'^3)+2^{1/3}a'(\tilde\zeta-\tilde\eta')\biggr\} 2^{1/3}
\nonumber
\\[-8pt]
\\[-8pt]
\nonumber
&&\hspace*{4pt}\quad{}\times\frac{\pi
(-\tilde\mu
)^{-2^{1/3}(\tilde\zeta-\tilde\eta')}}{\sin(\pi2^{1/3}(\tilde
\zeta
-\tilde\eta'))} \frac{\Gamma(2^{1/3}\tilde\zeta-{X}/{T}
)}{\Gamma(2^{1/3}\tilde\eta'-{X}/{T})} \frac{d\tilde\zeta
}{\tilde\zeta-\tilde\eta}.
\end{eqnarray}
To make this cleaner we replace $\tilde\mu/2$ with $\tilde\mu$. Taking
into account this change of variables (it also changes the $\exp\{
-\tilde\mu/2\}$ in front of the determinant to $\exp\{-\tilde\mu\}$),
we find that our final answer is
\[
\int_{\tilde\mathcal{C}}e^{-\tilde\mu}\frac{d\tilde\mu}{\tilde
\mu}\det
(I-K_{a}^{\csc,\Gamma})_{L^2(\tilde\Gamma_{\eta})},
\]
which, up to the definitions of the contours $\tilde\Gamma_{\eta}$ and
$\tilde\Gamma_{\zeta}$, is the desired limiting formula.

It is important to note the many possible pitfalls of such a heuristic
computation:
(1) Pointwise convergence of both the prefactor infinite product and
the Fredholm determinant is not enough to prove convergence of the
$\tilde\mu$ integral; (2) the deformations of the $\eta$ and $\zeta$
contours to the steepest descent curves are invalid, as they pass
through multiple poles of the kernel, coming both from the $f$ term and
the $g$ term; (3) one has to show that the kernels converge in the
sense of trace norm as opposed to just pointwise. The Riemann sum
approximation argument can in fact be made rigorous though; in \cite
{ACQ} an alternative proof of the validity of that limit is given via
analysis of singularities and residues.\vadjust{\goodbreak}

These possible pitfalls are addressed below in Section~\ref{thm_proof_sec}.

\subsection{\texorpdfstring{Proof of Theorem \protect\ref{hfluc_thm}}
{Proof of Theorem 3}}\label{thm_proof_sec}
In this section we provide a complete proof of Theorem \ref
{thm_proof_thm}, from which one recovers Theorem~\ref{hfluc_thm} via
scaling. The proof follows the same argument to that of~\cite{ACQ}. As
a convention, $c$ (or capitalized, primed, etc., versions) will represent
a finite constant which can vary line to line, unless explicitly noted.

In Theorem~\ref{thm_proof_thm}, we have reformulated the claim of
Theorem~\ref{hfluc_thm} in terms of the weakly asymmetric simple
exclusion process with half step Bernoulli initial data. Our proof,
therefore, reduces to a rigorous asymptotic analysis of Tracy and
Widom's formula (\ref{TW_prob_equation}). That formula contains an
integral over a $\mu$ contour of a product of a prefactor infinite
product and a Fredholm determinant. The first step toward taking the
limit of this as $\e$ goes to zero is to control the prefactor, $\prod
_{k=0}^{\infty} (1-\mu\tau^k)$. Initially $\mu$ lies on a contour
$S_{\tau^+}$ which is centered at zero and of radius between $\tau$ and
1. Along this contour the partial products (i.e., product up to $N$)
form a highly oscillatory sequence, and hence it is hard to control the
convergence of the sequence.

\begin{figure}

\includegraphics{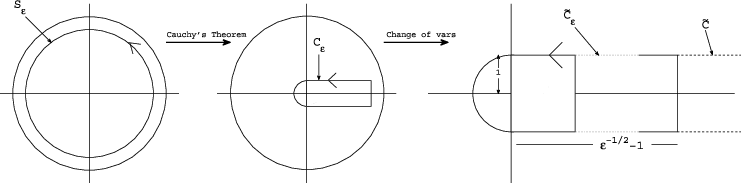}

\caption{The $S_{\e}$ contour is deformed to the $C_{\e}$ contour via
Cauchy's theorem, and then a change of variables leads to $\tilde
{C}_{\e
}$, with its infinite extension $\tilde{C}$.}\label{deform_to_c}
\end{figure}

The first step in our proof is to deform the $\mu$ contour $S_{\tau^+}$
to the long, skinny cigar-shaped contour
\[
\mathcal{C}_{\e} = \{\e^{1/2}e^{i\theta}\}_{\pi/2\leq\theta\leq
3\pi
/2}\cup\{x\pm i \e^{1/2}\}_{0<x\leq1-\e^{1/2}}\cup\{1-\e^{1/2}+\e
^{1/2}iy\}_{-1<y<1};
\]
see Figure~\ref{deform_to_c}. We orient $ \mathcal{C}_{\e}$
counter-clockwise. Notice that this new contour still includes all of
the poles at $\mu=\tau^{k}$ associated with the $f$ function in the
$J$ kernel.

In order to justify replacing $S_{\tau^+}$ by $\mathcal{C}_{\e}$ we
need the following:

\begin{lemma}\label{deform_mu_to_C}
In equation (\ref{TW_prob_equation}) we can replace the contour $S_\e$
with $\mathcal{C}_\e$ as the contour of integration for $\mu$ without
affecting the value of the integral.
\end{lemma}

We thank the referee for pointing out a mistake in the proof of this
result in~\cite{ACQ}, and suggesting an alternative proof which we
detail in Section~\ref{proofs_sec}.

Having made this deformation of the $\mu$ contour, we now observe that
the natural scale for $\mu$ is on order $\e^{1/2}$.\vadjust{\goodbreak} With this in mind
we make the change of variables
\[
\mu= \e^{1/2}\tilde\mu.
\]

\begin{remark}
Throughout the proof of this theorem and its lemmas and propositions,
we will use the tilde to denote variables which are $\e^{1/2}$ rescaled
versions of the original, untilded variables.
\end{remark}

The $\tilde\mu$ variable now lives on the contour
\[
\tilde\mathcal{C}_{\e} = \{e^{i\theta}\}_{\pi/2\leq\theta\leq
3\pi
/2}\cup\{x\pm i\}_{0<x\leq\e^{-1/2}-1}\cup\{\e^{-1/2}-1+iy\}_{-1<y<1},
\]
which grow and ultimately approach
\[
\tilde\mathcal{C} = \{e^{i\theta}\}_{\pi/2\leq\theta\leq3\pi
/2}\cup
\{x\pm i\}_{x>0}.
\]
In order to show convergence of the integral as $\e$ goes to zero, we
must consider two things, the convergence of the integrand for $\tilde
\mu$ in some compact region around the origin on $\mathcal{\tilde{C}}$,
and the controlled decay of the integrand on $\tilde\mathcal{C}_{\e}$
outside of that compact region. This second consideration will allow us
to approximate the integral by a finite integral in $\tilde\mu$, while
the first consideration will tell us what the limit of that integral
is. When all is said and done, we will paste back in the remaining part
of the $\tilde\mu$ integral and have our answer.
With this in mind we give the following bound which is taken word for
word from Lemma 2.3 of~\cite{ACQ} and whose proof (given therein)
relies only on elementary inequalities for the logarithm.

\begin{lemma}\label{mu_inequalities_lemma}
Define two regions, depending on a fixed parameter $r\geq1$,
\begin{eqnarray*}
R_1 &=& \biggl\{\tilde\mu\dvtx  |\tilde\mu|\leq\frac{r}{\sin(\pi/10)}\biggr\},\\
R_2 &=& \biggl\{\tilde\mu\dvtx  \re(\tilde\mu)\in\biggl[\frac{r}{\tan(\pi
/10)},\e
^{-1/2}\biggr] \mbox{ and } \im(\tilde\mu)\in[-2,2]\biggr\}.
\end{eqnarray*}
$R_1$ is compact, and $R_1 \cup R_2$ contains all of the contour
$\mathcal{\tilde{C}}_{\e}$. Furthermore define the function (the
infinite product after the change of variables)
\[
p_{\e}(\tilde\mu) = \prod_{k=0}^{\infty} (1-\e^{1/2}\tilde\mu
\tau^k).
\]
Then uniformly in $\tilde\mu\in R_1$,
%
%
\begin{equation}\label{g_e_ineq1}
p_{\e}(\mu)\to e^{-\tilde\mu/2}.
\end{equation}
Also, for all $\e<\e_0$ (some positive constant) there exists a
constant $c$, such that for all $\tilde\mu\in R_2$, we have the
following tail bound:
%
%
\begin{equation}\label{g_e_ineq2}
|p_{\e}(\tilde\mu)| \leq|e^{-\tilde\mu/2}| |e^{-c\e^{1/2}\tilde
\mu^2}|.
\end{equation}
[By the choice of $R_2$, for all $\tilde\mu\in R_2$, $\re(\tilde\mu
^2)>\delta>0$ for some fixed $\delta$. The constant~$c$ can be taken to
be $1/8$.]\vadjust{\goodbreak}
\end{lemma}

We now turn our attention to the Fredholm determinant term in the
integrand. Just as we did for the prefactor infinite product in Lemma
\ref{mu_inequalities_lemma} we must establish uniform convergence of
the determinant for $\tilde\mu$ in a fixed compact region around the
origin, and a suitable tail estimate valid outside that compact region.
The tail estimate must be such that for each finite~$\e$, we can
combine the two tail estimates (from the prefactor and from the
determinant) and show that their integral over the tail part of
$\tilde\mathcal{C}_{\e}$ is small and goes to zero as we enlarge the
original compact region. For this we have the following two
propositions (the first is the most substantial and is proved in
Section~\ref{J_to_K_sec}, while the second is proved in Section \ref
{proofs_sec}).

\begin{proposition}\label{uniform_limit_det_J_to_Kcsc_proposition}
Fix $s\in\R$, $T>0$ and $X\in\R$. Then for any compact subset of
$\tilde\mathcal{C}$, we have that for all $\delta>0$, there exists an
$\e_0>0$ such that for all $\e<\e_0$ and all $\tilde\mu$ in the
compact subset,
\[
\bigl|\det(I+\e^{1/2}\tilde\mu J^{\Gamma}_{\e^{1/2}\tilde\mu
})_{L^2(\Gamma_{\eta})} - \det(I-K^{\csc,\Gamma}_{a'})_{L^2(\tilde
\Gamma
_{\eta})}\bigr|<\delta.
\]
Here $a'=a+\log2$ and $K_{a',\Gamma}^{\csc}$ is defined in equation
(\ref{KGammaDef}) and depends implicitly on $\tilde\mu$.
\end{proposition}
\begin{proposition}\label{originally_cut_mu_lemma}
There exist $c,c'>0$ and $\e_0>0$ such that for all $\e<\e_0$ and all
$\tilde\mu\in\tilde\mathcal{C}_{\e}$,
\[
\bigl|p_{\e}(\tilde\mu)\det(I+\e^{1/2}\tilde\mu J^{\Gamma}_{\e
^{1/2}\tilde\mu})_{L^2(\Gamma_{\eta})}\bigr| \leq c'e^{-c|\tilde\mu|}.
\]
\end{proposition}

This exponential decay bound on the integrand shows that, by choosing a
suitably large (fixed) compact region around zero along the contour
$\tilde\mathcal{C}_{\e}$, it is possible to make the $\tilde\mu$
integral outside of this region arbitrarily small, uniformly in $\e\in
(0,\e_0)$. This means that we may assume henceforth that $\tilde\mu$
lies in a compact subset of $\tilde\mathcal{C}$.

Now that we are on a fixed compact set of $\tilde\mu$, the first part
of Lemma~\ref{mu_inequalities_lemma} and Proposition \ref
{uniform_limit_det_J_to_Kcsc_proposition} combine to show that the
integrand converges uniformly to
\[
\frac{e^{-\tilde\mu/2}}{\tilde\mu} \det(I-K^{\csc,\Gamma
}_{a'})_{L^2(\tilde\Gamma_{\eta}),}
\]
and hence the integral converges to the integral with this integrand.

To finish the proof of the limit in Theorem~\ref{thm_proof_thm}, it is
necessary that for any $\delta$ we can find a suitably small $\e_0$
such that the difference between the two sides of the limit differ by
less than $\delta$ for all $\e<\e_0$. Technically we are in the
position of a $\delta/3$ argument. One portion of $\delta/3$ goes to
the cost of cutting off the $\tilde\mu$ contour outside of some compact
set. Another $\delta/3$ goes to the uniform convergence of the
integrand. The final portion goes to repairing the $\tilde\mu$ contour.
As $\delta$ gets smaller, the cut for the $\tilde\mu$ contour must
occur further out. Therefore the limiting integral will be over the
limit of the $\tilde\mu$ contours,\vadjust{\goodbreak} which we called $\tilde\mathcal{C}$. The final $\delta/3$ is spent on the following proposition, whose
proof is given in Section~\ref{proofs_sec}.

\begin{proposition}\label{reinclude_mu_lemma}
There exists $c,c'>0$ such that for all $\tilde\mu\in\tilde\mathcal{C}$ with \mbox{$|\tilde\mu|\geq1$},
\[
\biggl|\frac{e^{-\tilde\mu/2}}{\tilde\mu} \det(I-K^{\csc,\Gamma
}_{a})_{L^2(\tilde\Gamma_{\eta})}\biggr| \leq|c'e^{-c\tilde\mu}|.
\]
\end{proposition}

Recall that the kernel $K^{\csc,\Gamma}_{a}$ is a function of $\tilde
\mu$.
The argument used to prove this proposition immediately shows that
$K_a^{\csc,\Gamma}$ is a trace class operator on $L^2(\tilde\Gamma
_{\eta})$.

It is an immediate corollary of this exponential tail bound that for
sufficiently large compact sets of $\tilde\mu$, the cost to include the
rest of the $\tilde\mu$ contour is less than $\delta/3$. This, along
with the change of variables in $\tilde\mu$ described at the end of
Section~\ref{heuristic_sec} finishes the proof of Theorem~\ref{thm_proof_thm}.

\subsection{\texorpdfstring{Proof of Proposition \protect\ref{uniform_limit_det_J_to_Kcsc_proposition}}
{Proof of Proposition 26}}\label{J_to_K_sec}

In this section we provide all of the steps necessary to prove
Proposition~\ref{uniform_limit_det_J_to_Kcsc_proposition}. To ease
understanding of the argument we relegate more technical points to
lemmas whose proof we delay to Section~\ref{JK_proofs_sec}.

During the proof of this proposition, it is important to keep in mind
that we are assuming that $\tilde\mu$ lies in a fixed compact subset of
$\tilde\mathcal{C}$. Recall that $\tilde\mu= \e^{-1/2}\mu$. We
proceed via the following strategy to find the limit of the Fredholm
determinant as~$\e$ goes to zero. The first step is to deform the
contours $\Gamma_{\eta}$ and $\Gamma_{\zeta}$ to suitable curves along
which there exists a small region outside of which the kernel of our
operator is exponentially small. This justifies cutting the contours
off outside of this small region. We may then rescale everything so
this small region becomes order one in size. Then we show uniform
convergence of the kernel to the limiting kernel on the compact subset.
Finally we need to show that we can complete the finite contour on
which this limiting object is defined to an infinite contour without
significantly changing the value of the determinant.

Recall from Theorem~\ref{TWmainthm} that $\Gamma_{\zeta}$ is defined to
be a circle of diameter $[-1+\dimple, 1+\dimple]$, while $\Gamma
_{\eta}$
is a circle of diameter $[-1+2\dimple,1-\dimple]$. The condition
imposed on $\dimple$ is that for $\zeta$ and $\eta$ on the above
contours, $|\zeta/\eta|\leq(1,\tau^{-1})$. We take $\dimple=\e
^{1/2}/2$, and since $\tau^{-1}\approx1+2\e^{1/2}$, it is clear that
for this choice of $\dimple$, $|\zeta/\eta|\leq(1,\tau^{-1})$. The
choice of contours is also such that the poles of the infinite product
$1/g(\zeta)$, which occur at $-\tau^{-n}$ for $n\geq0$, lie to the
left of the contours. Also recall
\[
\xi= -1 - 2\e^{1/2}\frac{X}{T}.
\]
The function $f(\mu,\zeta/\eta')$ which shows up in the definition of
the kernel for $J$ has poles as every point $\zeta/\eta'=z=\tau^k$ for
$k\in\Z$.\vadjust{\goodbreak}

As long as we simultaneously deform the $\Gamma_{\zeta}$ contour as we
deform $\Gamma_{\eta}$ so as to keep $\zeta/\eta'$ away from these
poles, we may use Proposition~\ref{TWprop1} (Proposition 1 of~\cite
{TW3}), to justify the fact that the determinant does not change under
this deformation. In this way we may deform our contours to the
following modified contours $\Gamma_{\eta,l},\Gamma_{\zeta,l}$:


\begin{definition}\label{kappacontours}
Let $\Gamma_{\eta,l}$ and $\Gamma_{\zeta,l}$ be two families (indexed
by $l>0$) of simple closed contours in $\C$ defined as follows. Let
%
%
\begin{equation}\label{kappa_eqn}
\kappa(\theta) = \frac{2X}{T} \tan^2\biggl(\frac{\theta}{2}\biggr)\log
\biggl(\frac{2}{1-\cos\theta}\biggr).
\end{equation}
Both $\Gamma_{\eta,l}$ and $\Gamma_{\zeta,l}$ will be symmetric across
the real axis, so we need only define them on the top half. $\Gamma
_{\eta,l}$ begins on the real axis $-1+\e^{1/2}$ and follows a smooth,
northwesterly pointing curve and joins the vertical line with real part
$\xi+\e^{1/2}/2$ (see Figure~\ref{new_gamma_contours} for an
illustration of such a curve). It then follows the straight vertical
line for a distance $l\e^{1/2}$ and then joins the curve
%
%
\begin{equation}\label{kappa_param_eqn}
\bigl[1+\e^{1/2}\bigl(\kappa(\theta)+\alpha\bigr)\bigr]e^{i\theta}
\end{equation}
parametrized by $\theta$ from $\pi-l\e^{1/2} + O(\e)$ to $0$, and where
$\alpha= -1/2 + O(\e^{1/2})$. The small errors are necessary to make
sure that the curves join up at the end of the vertical section of the
curve. We extend this to a closed contour by reflection through the
real axis and orient it clockwise. We denote the first two parts (the
northwesterly pointing curve and vertical line), of the contour by
$\Gamma_{\eta,l}^{\mathrm{vert}}$ and the remaining, roughly circular part by
$\Gamma_{\eta,l}^{\mathrm{circ}}$. This means that $\Gamma_{\eta,l}=\Gamma
_{\eta
,l}^{\mathrm{vert}}\cup\Gamma_{\eta,l}^{\mathrm{circ}}$,
and along this contour we can think of parametrizing $\eta$ by $\theta
\in[0,\pi]$.

We define $\Gamma_{\zeta,l}$ similarly,\vspace*{1pt} except that it starts out at
$-1+\e^{1/2}/2$, joins the vertical line with real part $\xi-\e
^{1/2}/2$ and finally joins the curve given by equation (\ref
{kappa_param_eqn}) where the value of $\theta$ ranges from $\theta
=\pi
-l\e^{1/2} + O(\e)$ to $\theta=0$ and where $\alpha= 1/2 + O(\e
^{1/2})$. We similarly denote this contour by the union of $\Gamma
_{\zeta,l}^{\mathrm{vert}}$ and $\Gamma_{\zeta,l}^{\mathrm{circ}}$.
\end{definition}

By virtue of these definitions, it is clear that $\e^{-1/2}|\zeta
/\eta
'-\tau^k|$ stays bounded away from zero for all $k$, that $|\zeta
/\eta
'|$ is bounded in an closed set contained in $(1,\tau^{-1})$ for all
$\zeta\in\Gamma_{\zeta,l}$ and $\eta\in\Gamma_{\eta,l}$ and
that $\e
^{1/2}(\zeta+1)$ is bounded from zero. Therefore, for any $l>0$ we may,
by deforming both the $\eta$ and $\zeta$ contours simultaneously,
assume that our operator acts on $L^2(\Gamma_{\eta,l})$ and that its
kernel is defined via an integral along $\Gamma_{\zeta,l}$. It is
critical that we now show that, due to our choice of contours, we are
able to forget about everything except for the northwesterly pointing
curve and vertical part of the contours. To formulate this we have the
following:

\begin{definition}
Let $\chi_l^{\mathrm{vert}}$ and $\chi_l^{\mathrm{circ}}$ be projection operators acting
on $L^2(\Gamma_{\eta,l})$ which project onto $L^2(\Gamma_{\eta
,l}^{\mathrm{vert}})$ and $L^2(\Gamma_{\eta,l}^{\mathrm{circ}})$,\vadjust{\goodbreak} respectively.
Also
define two operators $J_l^{\mathrm{vert},\Gamma}$ and $J_l^{\mathrm{circ},\Gamma}$ which
act on $L^2(\Gamma_{\eta,l})$ and have kernels identical to
$J^{\Gamma
}$ [see equation (\ref{J_eqn_def})], except the $\zeta$ integral is\vspace*{1pt}
over $\Gamma_{\zeta,l}^{\mathrm{vert}}$ and $\Gamma_{\zeta,l}^{\mathrm{circ}}$, respectively.
Thus we have a family (indexed by $l>0$) of decompositions of our
operator $J$ as follows:
\[
J^{\Gamma} = J_l^{\mathrm{vert},\Gamma}\chi_{l}^{\mathrm{vert}} +J_l^{\mathrm{vert},\Gamma
}\chi
_{l}^{\mathrm{circ}}+J_l^{\mathrm{circ},\Gamma}\chi_{l}^{\mathrm{vert}}+J_l^{\mathrm{circ},\Gamma}\chi
_{l}^{\mathrm{circ}}.
\]
\end{definition}

We now show that it suffices to just consider the first part of this
decomposition ($J_l^{\mathrm{vert},\Gamma}\chi_{l}^{\mathrm{vert}}$) for sufficiently
large $l$.

\begin{proposition}\label{det_1_1_prop} Assume that $\tilde\mu$ is
restricted to a bounded subset of the contour $\tilde\mathcal{C}$.
For all $\delta>0$ there exist $\e_0>0$ and $l_0>0$ such that for all
$\e<\e_0$ and all $l>l_0$,
\[
\bigl|\det(I+\mu J^{\Gamma}_{\mu})_{L^2(\Gamma_{\eta,l})} - \det
(I+J_{l}^{\mathrm{vert},\Gamma})_{L^2(\Gamma_{\eta,l}^{\mathrm{vert}})}\bigr|<\delta.
\]
\end{proposition}

\begin{pf}
As was explained in the \hyperref[sec1]{Introduction}, if we let
%
%
\begin{equation}\label{n_0_eqn}
n_0=\lfloor\log(\e^{-1/2})/\log(\tau)\rfloor,
\end{equation}
then it follows from the invariance of the doubly infinite sum for
$f(\mu,z)$ that
%
%
\begin{equation}\label{nshift}
\mu f(\mu,z) = z^{n_0} \bigl(\tilde\mu f(\tilde\mu,z) + O(\e^{1/2})\bigr).
\end{equation}
Note that the $O(\e^{1/2})$ does not play a significant role in what
follows so we drop it.

Using the above argument and the following three lemmas (which are
proved in Section~\ref{JK_proofs_sec}), we will be able to complete the
proof of Proposition~\ref{det_1_1_prop}.

\begin{lemma}\label{kill_gamma_2_lemma}
For all $c>0$, there exist $0<l_0<\infty$ and $\e_0>0$ such that for
all $l>l_0$, $\e<\e_0$ and $\eta\in\Gamma_{\eta,l}^{\mathrm{circ}}$,
$\zeta\in
\Gamma_{\zeta,l}^{\mathrm{circ}}$,
\[
\re(\Psi(\eta))\geq c|\xi-\eta|\e^{-1/2},\qquad \re(\Psi(\zeta))\leq
-c|\xi-\zeta|\e^{-1/2}.
\]
Additionally, there exists a $C>0$, $0<l_0<\infty$ and $\e_0>0$ such
that for all $l>l_0$, $\e<\e_0$ and $\eta\in\Gamma_{\eta,l}^{\mathrm{circ}}$,
$\zeta\in\Gamma_{\zeta,l}^{\mathrm{circ}}$,
\[
\re(\Psi(\eta))\geq C|\xi-\eta|\e^{-1},\qquad \re(\Psi(\zeta))\leq
-C|\xi-\zeta|\e^{-1}.
\]
\end{lemma}

\begin{lemma}\label{mu_f_polynomial_bound_lemma}
For all $l>0$ there exist $\e_0>0$ and $c>0$ such that for all $\e<\e
_0$, $\eta'\in\Gamma_{\eta,l}$ and $\zeta\in\Gamma_{\zeta,l}$,
\[
|\tilde\mu f(\tilde\mu,\zeta/\eta')|\leq\frac{c}{|\zeta-\eta'|}.
\]
\end{lemma}

\begin{lemma}\label{new_lemma_1}
For all $c>0$ there exists $l_0>0$ and $\e_0>0$ such that for all
$l>l_0$, $\e<\e_0$ and all $\eta\in\Gamma^{\mathrm{circ}}_{\eta,l}$ and
$\zeta
\in\Gamma^{\mathrm{circ}}_{\zeta,l}$
\[
\re\Biggl(n_0\log(\zeta/\eta) + \sum_{n=0}^{\infty} \log\biggl(\frac
{1+\tau^n \eta}{1+\tau^n \zeta}\biggr)\Biggr) \leq c(|\xi-\zeta|+|\xi
-\eta|)\e^{-1}.\vadjust{\goodbreak}
\]
\end{lemma}

The $n_0$ above accounts for the $z^{n_0}$ from equation (\ref
{nshift}). Comparing the exponential (of order $\e^{-1}$) decay of the
second part of Lemma~\ref{kill_gamma_2_lemma} with the upper bound of
Lemma~\ref{new_lemma_1}, we find that since the constant in Lemma \ref
{new_lemma_1} is arbitrary, the decay of $\Psi(\eta)$ overwhelms the
possible growth of $(\zeta/\eta)^{n_0}g(\eta)/g(\eta)$. Additionally
taking into account the polynomial control of Lemma \ref
{mu_f_polynomial_bound_lemma} and the remaining term $1/(\zeta-\eta)$,
we find that for any $\delta>0$, we can find $l_0$ large enough that
$\Vert J_l^{\mathrm{vert},\Gamma}\chi_{l}^{\mathrm{circ}}\Vert _1$, $\Vert J_l^{\mathrm{circ},\Gamma}\chi
_{l}^{\mathrm{vert}}\Vert _1$ and $\Vert J_l^{\mathrm{circ},\Gamma}\chi_{l}^{\mathrm{circ}}\Vert _1$ are all
bounded by $\delta/3$. Technically, in order to show this we can factor
these various operators into a product of Hilbert--Schmidt operators
and then use the decay explained above to prove that each of the
Hilbert--Schmidt norms goes to zero (for a similar argument, see the
bottom of page 27 of~\cite{TW3}). This completes the proof of
Proposition~\ref{det_1_1_prop}.
\end{pf}

We now return to the proof of Proposition \ref
{uniform_limit_det_J_to_Kcsc_proposition}. We have successfully
restricted ourselves to considering $J_{l}^{\mathrm{vert},\Gamma}$ acting on
$L^2(\Gamma_{\eta,l}^{\mathrm{vert}})$. Having focused on the region of
asymptotically nontrivial behavior, we can now rescale and show that
the kernel converges to its limit, uniformly on the compact contour.

\begin{definition}\label{change_of_var_tilde_definitions}
Recall $c_3=2^{-4/3}$, and let
\[
\eta= \xi+ c_3^{-1}\e^{1/2}\tilde\eta, \qquad\eta' = \xi+ c_3^{-1}\e
^{1/2}\tilde\eta', \qquad\zeta= \xi+ c_3^{-1}\e^{1/2}\tilde\zeta.
\]
Under these change of variables the contours $\Gamma_{\eta,l}^{\mathrm{vert}}$
and $\Gamma_{\zeta,l}^{\mathrm{vert}}$ become
\begin{eqnarray*}
\tilde\Gamma_{\eta,l} &=& \{c_3/2+ir\dvtx r\in(-c_3l,1)\cup(1,c_3l)\}\cup
\tilde\Gamma_{\eta}^d,\\
\tilde\Gamma_{\zeta,l} &=& \{-c_3/2+ir\dvtx r\in(-c_3l,1)\cup(1,c_3l)\}
\cup
\tilde\Gamma_{\eta}^d,
\end{eqnarray*}
where $\tilde\Gamma_{\zeta}^d$ is a dimple which goes to the right of
$XT^{-1/3}$ and joins\vspace*{-1pt} with the rest of the contour, and where $\tilde
\Gamma_{\eta}^d$ is the same contour just shifted to the right by
distance~$c_3$; see Figure~\ref{new_gamma_contours}.

As $l$ increases to infinity, these contours approach their infinite versions,
\begin{eqnarray*}
\tilde\Gamma_{\eta} &=& \{c_3/2+ir\dvtx r\in(-\infty,-1)\cup(1,\infty)\}
\cup
\tilde\Gamma_{\eta}^d,\\
\tilde\Gamma_{\zeta} &=& \{-c_3/2+ir\dvtx r\in(-\infty,-1)\cup(1,\infty
)\}
\cup\tilde\Gamma_{\eta}^d.
\end{eqnarray*}
With respect to the change of variables define an operator $\tilde
J^{\Gamma}_{l}$ acting on $L^2(\tilde\Gamma_{\eta})$ via the kernel
\begin{eqnarray*}
\mu\tilde J^{\Gamma}_l (\tilde\eta,\tilde\eta') &=& c_{3}^{-1}\e^{1/2}\int_{\tilde\Gamma_{\zeta,l}} e^{\Psi(\xi+c_3^{-1} \e
^{1/2}\tilde\zeta
)-\Psi(\xi+c_3^{-1} \e^{1/2}\tilde\eta')} \\
&&\phantom{c_{3}^{-1}\e^{1/2}\int_{\tilde\Gamma_{\zeta,l}}}{}\times\frac{\mu f(\mu,
{(\xi
+c_3^{-1}\e^{1/2}\tilde\zeta)}/{(\xi+c_3^{-1}\e^{1/2}\tilde\eta
')})}{(\xi
+c_3^{-1}\e^{1/2}\tilde\eta')(\tilde\zeta-\tilde\eta)}\,d\tilde
\zeta.
\end{eqnarray*}
Finally, define the operator $\tilde\chi_l$ which projects
$L^2(\tilde
\Gamma_{\eta})$ onto $L^2(\tilde\Gamma_{\eta,l})$.
\end{definition}

It is clear that applying the change of variables, the Fredholm
determinant $\det(I+J_{l}^{\mathrm{vert},\Gamma})_{L^2(\Gamma_{\eta
,l}^{\mathrm{vert}})}$ becomes
$\det(I+\tilde\chi_l \mu\tilde J_l^{\Gamma} \tilde\chi
_l)_{L^2(\tilde
\Gamma_{\eta,l})}$.

We now state a proposition which gives, with respect to these fixed
contours $\tilde\Gamma_{\eta,l}$ and $\tilde\Gamma_{\zeta,l}$, the
limit of the determinant in terms of the uniform limit of the kernel.
Since all contours in question are finite, uniform convergence of the
kernel suffices to show trace class convergence of the operators and
hence convergence of the determinant.

Recall the definition of the operator $K_a^{\csc,\Gamma}$ given in
equation (\ref{KGammaDef}). For the purposes of this proposition,
modify the kernel so that the integration in $\zeta$ occurs now only
over $\tilde\Gamma_{\zeta,l}$ and not all of $\tilde\Gamma_{\zeta}$.
Call this modified operator~$K_{a',l}^{\csc,\Gamma}$.

\begin{proposition}\label{converges_to_kcsc_proposition}
For all $\delta>0$ there exist $\e_0>0$ and $l_0>0$ such that for all
$\e<\e_0$, $l>l_0$ and $\tilde\mu$ in our fixed compact subset of
$\tilde\mathcal{C}$,
\[
\bigl|\det(I+\tilde\chi_l \mu\tilde J_l^{\Gamma} \tilde\chi
_l)_{L^2(\tilde\Gamma_{\eta,l})} - \det(I-\tilde\chi_l
K_{a',l}^{\csc
,\Gamma}\tilde\chi_l)_{L^2(\tilde\Gamma_{\eta,l})}\bigr|< \delta,
\]
where $a'=a+\log2$.
\end{proposition}
\begin{pf}
The proof of this proposition relies on showing the uniform convergence
of the kernel of $\mu\tilde J_{l}^{\Gamma}$ to the kernel of
$K_{a',l}^{\csc,\Gamma}$, which suffices because of the compact
contour. Furthermore, since the $\zeta$ integration is itself over a
compact set, it suffices to show uniform convergence of this integrand.
The two lemmas stated below will imply such uniform convergence and
hence complete this proof.

First, however, recall that $\mu f(\mu,z) = z^{n_0} (\tilde\mu
f(\tilde
\mu,z) + O(\e^{1/2}))$ where $n_0$ is defined in equation (\ref
{n_0_eqn}). We are interested in having $z=\zeta/\eta'$, which, under
the change of variables can be written as
\[
z=1-\e^{1/2}\tilde z +O(\e), \qquad\tilde z = c_3^{-1}(\tilde\zeta
-\tilde\eta')=2^{4/3}(\tilde\zeta-\tilde\eta').
\]
Therefore, since $n_0= -\frac{1}{2}\log(\e^{-1/2})\e^{-1/2}+ O(1)$ it
follows that
\[
z^{n_0} = \exp\{-2^{1/3}(\tilde\zeta-\tilde\eta')\log(\e
^{-1/2})\}\bigl(1+o(1)\bigr).
\]
This expansion still contains an $\e$, and hence the argument blows up
as $\e$ goes to zero. This exactly counteracts the asymptotics of the
ratio of $g(\eta')/g(\zeta)$ due to the following:
\begin{lemma}\label{gratio_limit}
For all $l>0$ and all $\delta>0$ there exists $\e_0>0$, such that for
all $\tilde\eta'\in\tilde\Gamma_{\eta,l}$ and $\tilde\zeta\in
\tilde
\Gamma_{\zeta,l}$, we have for $0<\e\leq\e_0$,
\[
\biggl|\exp\{n_0\log(\zeta/\eta')\}\frac{g(\eta')}{g(\zeta)} - \frac
{\Gamma(2^{1/3}\tilde\zeta-{X}/{T})}{\Gamma(2^{1/3}\tilde\eta
'-{X}/{T})} e^{2^{1/3}\log(2) (\tilde\zeta-\tilde\eta')}
\biggr|<\delta.
\]
\end{lemma}

This is combined with the following two lemmas, and all three are
proved in Section~\ref{JK_proofs_sec}.\vadjust{\goodbreak}
\begin{lemma}\label{compact_eta_zeta_taylor_lemma}
For all $l>0$ and all $\delta>0$ there exists $\e_0>0$ such that for
all $\tilde\eta'\in\tilde\Gamma_{\eta,l}$ and $\tilde\zeta\in
\tilde
\Gamma_{\zeta,l}$ we have for $0<\e\le\e_0$,
\[
\biggl|\bigl(\Psi(\tilde\zeta)-\Psi(\tilde\eta')\bigr) - \biggl(-\frac
{T}{3}(\tilde\zeta^3-\tilde\eta'^3) + 2^{1/3}a(\tilde\zeta-\tilde
\eta
)\biggr)\biggr|<\delta.
\]
Similarly we have
\[
\bigl|e^{\Psi(\tilde\zeta)-\Psi(\tilde\eta')} - e^{-
({T}/{3})(\tilde
\zeta^3-\tilde\eta'^3) + 2^{1/3}a(\tilde\zeta-\tilde\eta')}
\bigr|<\delta.
\]
\end{lemma}

\begin{lemma}\label{muf_compact_sets_csc_limit_lemma}
For all $l>0$ and all $\delta>0$, there exists $\e_0>0$ such that for
all $\tilde\eta'\in\tilde\Gamma_{\eta,l}$ and $\tilde\zeta\in
\tilde
\Gamma_{\zeta,l}$ we have for $0<\e\le\e_0$,
\[
\biggl|\e^{1/2}\tilde\mu f\biggl(\tilde\mu, \frac{\xi+c_3^{-1}\e
^{1/2}\tilde\zeta}{\xi+c_3^{-1}\e^{1/2}\tilde\eta'}\biggr) - \int
_{-\infty}^{\infty} \frac{\tilde\mu e^{-2^{1/3}t(\tilde\zeta
-\tilde\eta
')}}{e^{t}-\tilde\mu}\,dt\biggr|<\delta.
\]
\end{lemma}

The integral above in Lemma~\ref{muf_compact_sets_csc_limit_lemma}
converges since our choices of contours for $\tilde\zeta$ and $\tilde
\eta'$ ensure that $\re(-2^{1/3}(\tilde\zeta-\tilde\eta'))=1/2$.
Note that the above integral also has a representation (\ref
{cscIntegral}) in terms of the cosecant function. This provides the
analytic extension of the integral to all $\tilde z\notin2\Z$ where
$\tilde{z}=2^{4/3}(\tilde\zeta-\tilde\eta')$ (note, however, that
we do
not require use of this analytic extension due to our choice of contours).

Finally, the sign change in front of the kernel of the Fredholm
determinant comes from the $1/\eta'$ term which, under the change of
variables converges uniformly to $-1$. The reason why $a'=a+\log2$
arises here is due to the term $e^{2^{1/3}\log(2) (\tilde\zeta
-\tilde
\eta')}$ from Lemma~\ref{gratio_limit}. This proves the desired result.
\end{pf}

Having successfully taken the $\e$ to zero limit, all that now remains
is to paste the rest of the contours, $\tilde\Gamma_{\eta}$ and
$\tilde
\Gamma_{\zeta}$, to their abbreviated versions, $\tilde\Gamma_{\eta,l}$
and $\tilde\Gamma_{\zeta,l}$. To justify this we must show that the
inclusion of the rest of these contours does not significantly affect
the Fredholm determinant.
Just as in the proof of Proposition~\ref{det_1_1_prop} we have three
operators which we must re-include at provably small cost. Each of
these operators, however, can be factored into the product of Hilbert--Schmidt
operators and then an analysis similar to that done following
Lemma~\ref{mu_f_polynomial_bound_lemma} (see, in particular, pages
27--28 of~\cite{TW3}) shows that because $\re(\tilde\zeta^3)$ grows
like $|\tilde\zeta|^2$ along $\tilde\Gamma_{\zeta}$ (and likewise but
opposite for $\eta'$), there is sufficiently strong exponential decay
to show that the trace norms of these three additional kernels can be
made arbitrarily small by taking $l$ large enough.

This last estimate completes the proof of Proposition \ref
{uniform_limit_det_J_to_Kcsc_proposition}.

\subsection{Technical lemmas, propositions and proofs}
\subsubsection{Properties of Fredholm determinants}\label{pre_lem_ineq_sec}

Before beginning the proofs of the propositions and lemmas, we give the
definitions and some important properties for Fredholm determinants,
trace class operators and\vadjust{\goodbreak} Hilbert--Schmidt operators. For a more
complete treatment of this theory see, for example,~\cite{BS:book}.

Consider a (separable) Hilbert space $\Hi$ with bounded linear
operators $\mathcal{L}(\Hi)$. If $A\in\mathcal{L}(\Hi)$, let
$|A|=\sqrt
{A^*A}$ be the unique positive square-root. We say that $A\in\mathcal
{B}_1(\Hi)$, the trace class operators, if the trace norm
$\Vert A\Vert _1<\infty$. Recall that this norm is defined relative to an
orthonormal basis of $\Hi$ as $\Vert A\Vert _1:= \sum_{n=1}^{\infty} (e_n,|A|e_n)$.
This norm is well defined and does not depend on the choice of
orthonormal basis $\{e_n\}_{n\geq1}$. For $A\in\mathcal{B}_1(\Hi)$,
one can then define the trace $\tr A:=\sum_{n=1}^{\infty} (e_n,A
e_n)$. We say that $A\in\mathcal{B}_{2}(\Hi)$, the Hilbert--Schmidt
operators, if the Hilbert--Schmidt norm $\Vert A\Vert _2:= \sqrt{\tr
(|A|^2)}<\infty$.


For $A\in~\mathcal{B}_1(\Hi)$ we can also define a Fredholm determinant
$\det(I+~A)_{\Hi}$. Consider $u_i\in\Hi$, and define the tensor
product $u_1\otimes\cdots\otimes u_n$ by its action on $v_1,\ldots,
v_n \in\Hi$ as
\[
u_1\otimes\cdots\otimes u_n (v_1,\ldots, v_n) = \prod_{i=1}^{n} (u_i,v_i).
\]
Then $\bigotimes_{i=1}^{n}\Hi$ is the span of all such tensor products.
There is a vector subspace of this space which is known as the
alternating product
\[
\bigwedge^n(\Hi) = \Biggl\{h\in\bigotimes_{i=1}^{n} \Hi\dvtx  \forall\sigma
\in
S_n, \sigma h =-h\Biggr\},
\]
where $\sigma u_1\otimes\cdots\otimes u_n = u_{\sigma(1)}\otimes
\cdots\otimes u_{\sigma(n)}$.
If $e_1,\ldots,e_n$ is a basis for $\Hi$, then $e_{i_1}\wedge\cdots
\wedge e_{i_k}$ for $1\leq i_1<\cdots<i_k\leq n$ form a basis of
$\bigwedge^n(\Hi)$.
Given an operator $A\in\mathcal{L}(\Hi)$, define
\[
\Gamma^n(A)(u_1\otimes\cdots\otimes u_n):= Au_1\otimes\cdots
\otimes Au_n.
\]
Note that any element in $\bigwedge^n(\Hi)$ can be written as an
antisymmetrization of tensor products. Then it follows that $\Gamma
^n(A)$ restricts to an operator from $\bigwedge^n(\Hi)$ into
$\bigwedge
^n(\Hi)$. If $A\in~\mathcal{B}_1(\Hi)$, then $\tr\Gamma
^{(n)}(A)\leq
\Vert A\Vert _1^n/n!$, and
we can define
\[
\det(I+~A)= 1 + \sum_{k=1}^{\infty} \tr\bigl(\Gamma^{(k)}(A)\bigr).
\]
As one expects, $\det(I+~A)=\prod_j (1+\lambda_j)$ where $\lambda_j$
are the eigenvalues of $A$ counted with algebraic multiplicity (Theorem
XIII.106,~\cite{RS:book}).

\begin{lemma}[(Chapter 3 in~\cite{BS:book})]\label{fredholm_continuity_lemma}
$A\mapsto\det(I+A)$ is a continuous function on $\mathcal{B}_1(\Hi)$.
Explicitly,
\[
|\det(I+A)-\det(I+B)|\leq\Vert A-B\Vert _{1}\exp(\Vert A\Vert _1+\Vert B\Vert _1+1).
\]
If $A\in\mathcal{B}_1(\Hi)$ and $A=BC$ with $B,C\in\mathcal
{B}_2(\Hi
)$, then
\[
\Vert A\Vert _1\leq\Vert B\Vert _2\Vert C\Vert _2.\vadjust{\goodbreak}
\]
For $A\in\mathcal{B}_1(\Hi)$,
\[
|\det(I+A)|\leq e^{\Vert A\Vert _1}.
\]
If $A\in\mathcal{B}_2(\Hi)$ with kernel $A(x,y)$, then
\[
\Vert A\Vert _2 = \biggl(\int|A(x,y)|^2 \,dx \,dy\biggr)^{1/2}.
\]
\end{lemma}

\begin{lemma}\label{projection_pre_lemma}
If $K$ is an operator acting on a contour $\Sigma$, and $\chi$ is a
projection operator unto a subinterval of $\Sigma$, then
\[
\det(I+K\chi)_{L^2(\Sigma,\mu)}=\det(I+\chi K\chi)_{L^2(\Sigma
,\mu)}.
\]
\end{lemma}

In performing steepest descent analysis on Fredholm determinants, the
following proposition allows one to deform contours to descent curves.

\begin{lemma}[(Proposition 1 of~\cite{TW3})]\label{TWprop1}
Suppose $s\to\Gamma_s$ is a deformation of closed curves, and a kernel
$L(\eta,\eta')$ is analytic in a neighborhood of $\Gamma_s\times
\Gamma
_s\subset\C^2$ for each $s$. Then the Fredholm determinant of $L$
acting on $\Gamma_s$ is independent of $s$.
\end{lemma}



\subsubsection{\texorpdfstring{Proofs from Section \protect\ref{thm_proof_sec}}
{Proofs from Section 3.2}}\label{proofs_sec}
We now turn to the proofs of the previously stated lemmas and propositions.
\begin{pf*}{Proof of Lemma \protect\ref{deform_mu_to_C}}
We thank the referee for suggesting the following simple proof of this
result. The lemma follows from Cauchy's theorem once we show that for
fixed $\varepsilon$, the expression $\prod_{k=0}^{\infty} (1-\mu
\tau^k)\det
(I+\mu J^{\Gamma}_{\mu})$ is analytic in $\mu$ between $S_\e$ and
$\mathcal{C}_\e$ (note that we now include a subscript $\mu$ on $J$ to
emphasize the dependence of the kernel on $\mu$). However, this
expression was derived from and is equivalent to
\[
\frac{\det(I-\lambda K)}{\prod_{k=0}^{m-1}(1-\lambda\tau^k)},
\]
where $\lambda= \tau^{-m} \mu$ and $K$ is the operator (1) of \cite
{TW4}. The operator $K$ does not depend on $\lambda$ and is thus is an
entire function of $\lambda$. Therefore the only singularities are at
$\lambda=1,\ldots, \tau^{-m+1}$, which correspond to $\mu=\tau
,\ldots,
\tau^{m}$. None of these singularities are between the two contours;
thus the desired result follows.
\end{pf*}

\begin{pf*}{Proof of Lemma \protect\ref{mu_inequalities_lemma}}
We prove this with the scaling parameter $r=1$ as the general case
follows in a similar way.
Consider
\[
\log(g_{\e}(\tilde\mu))=\sum_{k=0}^{\infty} \log(1-\e
^{1/2}\tilde\mu
\tau_\e^k).
\]
We have $\sum_{k=0}^{\infty} \varepsilon^{1/2} \tau^k = \frac12(1
+ \e
^{1/2}c_\e)$ where $c_\e=O(1)$. So for $\tilde\mu\in R_1$ we have
\begin{eqnarray*}
\biggl|\log(g_{\e}(\tilde\mu))+\frac{\tilde\mu}2(1+ \e
^{1/2}c_\e)\biggr|
&=& \Biggl|\sum_{k=0}^{\infty}\log(1-\varepsilon^{1/2}\tilde\mu\tau^k)+
\varepsilon^{1/2}\tilde\mu\tau^k\Biggr|\\
&\leq& \sum_{k=0}^{\infty} |\log(1-\varepsilon^{1/2}\tilde\mu
\tau^k)+
\varepsilon^{1/2}\tilde\mu\tau^k|\\
&\leq& \sum_{k=0}^{\infty} |\varepsilon^{1/2}\tilde\mu
\tau
^k|^2 = \frac{\e|\tilde\mu|^2}{1-\tau^2}=
\frac{\e^{1/2} |\tilde\mu|^2}{4-4\e^{1/2}}\\
&\leq& c\e
^{1/2}|\tilde\mu|^2
\leq c' \e^{1/2}.
\end{eqnarray*}
The second inequality uses the fact that for $|z|\leq1/2$,
$|\log(1-z)+z|\leq|z|^2$. Since $\tilde\mu\in R_1$ it follows that
$|z|=\e^{1/2}|\tilde\mu|$ is bounded by $1/2$ for small enough~$\e$.
The constants here are finite and do not depend on any of the
parameters. This proves equation (\ref{g_e_ineq1}) and shows that the
convergence is uniform in $\tilde\mu$ on $R_1$.

We now turn to the second inequality, equation (\ref{g_e_ineq2}).
Consider the region
\[
D=\biggl\{z\dvtx \arg(z)\in\biggl[-{\frac{\pi}{10}},{\frac
{\pi
}{10}}\biggr]\biggr\}\cap\biggl\{z\dvtx \im(z)\in\biggl(-{\frac
{1}{10}},{
\frac{1}{10}}\biggr)\biggr\}\cap\{z\dvtx \re(z)\leq1\}.
\]
For all $z\in D$,
%
%
\begin{equation}\label{ineq2}
\re\bigl(\log(1-z)\bigr)\leq\re(-z-z^2/2).
\end{equation}
For $\tilde\mu\in R_2$, it is clear that $\e^{1/2}\tilde\mu\in D$.
Therefore, using (\ref{ineq2}),
\begin{eqnarray*}
\re(\log(g_{\e}(\tilde\mu))) &=& \sum_{k=0}^{\infty
}\re[\log
(1-\varepsilon^{1/2}\tilde\mu\tau^k)]\\
&\leq& \sum_{k=0}^{\infty} \bigl(-\re[\e^{1/2}\tilde\mu\tau^k]-\re
[(\e
^{1/2}\tilde\mu\tau^k)^2/2]\bigr)\\
&\leq& -\re(\tilde\mu/2)-\frac{1}{8}\e^{1/2}\re(\tilde
\mu^2).
\end{eqnarray*}
This proves equation (\ref{g_e_ineq2}). Note that from the definition
of $R_2$ we can calculate the argument of $\tilde\mu$, and we see that
$|\arg\tilde\mu|\leq\arctan(2\tan(\frac{\pi}{10}))<\frac{\pi}{4}$
and $|\tilde\mu|\geq r\geq1$. Therefore $\re(\tilde\mu^2)$ is positive
and bounded away from zero for all $\tilde\mu\in R_2$.
\end{pf*}

\begin{pf*}{Proof of Proposition \protect\ref{originally_cut_mu_lemma}}
This proof proceeds in a similar manner to the proof of Proposition
\ref
{reinclude_mu_lemma}; however, since in this case we have to deal with
$\e$ going to zero and changing contours, it is, by necessity, a little
more complicated. For this reason we encourage readers to first study
the simpler proof of Proposition~\ref{reinclude_mu_lemma}.\vadjust{\goodbreak}

In that proof we factor our operator into two pieces. Then, using the
decay of the exponential term, and the control over the size of the
$\csc$ term, we are able to show that the Hilbert--Schmidt norm of the
first factor is finite and that for the second factor it is bounded by
$|\tilde\mu|^{\alpha}$ for $\alpha<1$ (we show it for $\alpha=1/2$
though any $\alpha>0$ works, just with constant getting large as
$\alpha
\searrow0$). This gives an estimate on the trace norm of the operator,
which, by exponentiating, gives an upper bound $e^{c|\tilde\mu
|^{\alpha
}}$ on the size of the determinant. This upper bound is beat by the
exponential decay in $\tilde\mu$ of the prefactor term $p_{\e}$.

For the proof of Proposition~\ref{originally_cut_mu_lemma}, we do the
same sort of factorization of our operator into $AB$, where here,
\[
A(\zeta,\eta)=\frac{e^{c'\Psi(\zeta)}}{\zeta-\eta}
\]
with $0<c'<1$ fixed, and
\[
B(\eta,\zeta) = e^{-c'\Psi(\zeta)}e^{\Psi(\zeta)-\Psi(\eta)}\mu
f(\mu
,\zeta/\eta)\frac{g(\eta')}{g(\zeta)}\frac{1}{\eta}.
\]
We must be careful in keeping track of the contours on which these
operators act. It will be convenient for this proof to move the
contours for $\eta$ and $\zeta$ to contours $\Gamma_{\eta,l}^{0}$ and
$\Gamma_{\zeta,l}^{0}$ which are defined in the same manner as
$\Gamma
_{\eta,l}$ and $\Gamma_{\zeta,l}$ in Definition~\ref{kappacontours}.
The difference, however, is that these new contours start at $-1+\e
^{1/2}$ (resp., $-1+\e^{1/2}/2$) and go straight up for distance $l\e
^{1/2}$ before joining $\kappa(\theta)$ with $\alpha= - 1/2 +O(\e
^{1/2})$ [resp., $\alpha= 1/2 +O(\e^{1/2})$]. The purpose of this is to
avoid the necessity of creating a dimple in the contours, thus allowing
us to apply Lemma~\ref{final_estimate} in the form it was stated in
\cite{ACQ}. Changing the location of this vertical portion of our
contours does not affect the Taylor expansion we performed since we can
still center our rescaled variables at $\xi$, and all of those results
were valid in a compact region with respect to the rescaled variables.

Now using the estimates of Lemmas~\ref{kill_gamma_2_lemma} and \ref
{compact_eta_zeta_taylor_lemma}, we compute that $\Vert A\Vert _2<\infty$
(uniformly in $\e<\e_0$ and, trivially, also in $\tilde\mu$). Here we
calculate the Hilbert--Schmidt norm using Lemma \ref
{fredholm_continuity_lemma}. Intuitively this norm is uniformly bounded
as~$\e$ goes to zero because, while the denominator blows up as badly
as $\e^{-1/2}$, the numerator is roughly supported only on a region of
measure $\e^{1/2}$ (owing to the exponential decay of the exponential
when $\zeta$ differs from $\xi$ by more than order $\e^{1/2}$).

We wish to control $\Vert B\Vert _2$ now. Using the discussion before Lemma
\ref
{kill_gamma_2_lemma} we may rewrite $B$ as
\[
B(\eta,\zeta) = e^{-c\Psi(\zeta)}e^{\Psi(\zeta)-\Psi(\eta
)}\tilde\mu
f(\tilde\mu,\zeta/\eta) \biggl(\frac{\zeta}{\eta}\biggr)^{n_0}\frac
{g(\eta')}{g(\zeta)}\frac{1}{\eta}.
\]
Lemmas~\ref{kill_gamma_2_lemma} and~\ref{compact_eta_zeta_taylor_lemma}
show that
\[
\bigl|e^{-c\Psi(\zeta)}e^{\Psi(\zeta)-\Psi(\eta)}\bigr| \leq e^{-\e
^{-1}C(|\xi
-\eta|+|\xi-\zeta|)}\vadjust{\goodbreak}
\]
for some fixed constant $C>0$. On the other hand, Lemma \ref
{new_lemma_1} shows that
\[
\biggl|\biggl(\frac{\zeta}{\eta}\biggr)^{n_0}\frac{g(\eta')}{g(\zeta)}\biggr| \leq
e^{-\e^{-1}C'(|\xi-\eta|+|\xi-\zeta|)}
\]
for \textit{any} constant $C'$. In particular, we can take $C'<C$ which
shows that all of the term (besides the $f$ term) decay exponentially fast.

The final ingredient in proving our proposition is, therefore, control
of $|\tilde\mu f(\tilde\mu, z)|$ for $z=\zeta/\eta'$. We break it up
into two regions of $\eta',\zeta$: The first (1)~when $|\eta'-\zeta
|\leq c$ for a very small constant $c$ and the second (2) when $|\eta
'-\zeta|> c$. We will compute $\Vert B\Vert _2$ as the square root of
%
%
\begin{equation}\label{case1case2}
\int_{\eta,\zeta\in\mathrm{Case\ (1)}} |B(\eta,\zeta)|^2 \,d\eta
\,d\zeta
+ \int_{\eta,\zeta\in\mathrm{Case\ (2)}} |B(\eta,\zeta)|^2 \,d\eta
\,d\zeta.
\end{equation}
We will show that the first term can be bounded by $c''|\tilde\mu
|^{2\alpha}$ for any $\alpha<1$, while the second term can be bounded
by a large constant. As a result $\Vert B\Vert _2\leq c''|\tilde\mu|^{\alpha}$
which is exactly as desired since then $\Vert AB\Vert _1\leq e^{c''|\tilde\mu
|^{\alpha}}$; see Lemma~\ref{fredholm_continuity_lemma}.

Consider case (1) where $|\eta'-\zeta|\leq c$ for a constant $c$ which
is positive but small (depending on $T$). One may easily check from the
definition of the contours that $\e^{-1/2}(|\zeta/\eta|-1)$ is
contained in a compact subset of $(0,2)$. In fact, $\zeta/\eta'$ almost
exactly lies along the curve $|z|=1+\e^{1/2}$ and in particular (by
taking $\e_0$ and $c$ small enough) we can assume that $\zeta/\eta$
never leaves the region bounded by $|z|=1+(1\pm r)\e^{1/2}$ for any
fixed $r<1$. Let us call this region $R_{\e,r}$. Then we have:
\begin{lemma}\label{final_estimate}
Fix $\e_0$ and $r\in(0,1)$. Then for all $\e<\e_0$, $\tilde\mu\in
\tilde\mathcal{C}_{\e}$ and \mbox{$z\in R_{\e,r}$},
\[
|\tilde\mu f(\tilde\mu,z)| \leq c|\tilde\mu|^{\alpha}/|1-z|
\]
for some $\alpha\in(0,1)$, with $c=c(\alpha)$ independent of $z$,
$\tilde\mu$ and $\e$.
\end{lemma}

\begin{remark}
By changing the value of $\alpha$ in the definition of $\kappa(\theta)$
(which then goes into the definition of $\Gamma^{0}_{\eta,l}$ and
$\Gamma^{0}_{\zeta,l}$) and also focusing the region $R_{\e,r}$ around
$|z|=1+2\alpha\e^{1/2}$, we can take $\alpha$ arbitrarily small in the
above lemma at a cost of increasing the constant $c=c(\alpha)$ (the
same also applies for Proposition~\ref{reinclude_mu_lemma}). The
$|\tilde\mu|^{\alpha}$ comes from the fact that $(1+2\alpha\e
^{1/2})^{({1}/{2})\e^{-1/2}\log|\tilde\mu|} \approx|\tilde\mu
|^{\alpha}$. Another remark is that the proof below can be used to
provide an alternative proof of Lemma \ref
{muf_compact_sets_csc_limit_lemma} by studying the convergence of the
Riemann sum directly rather than by using functional equation
properties of $f$ and the analytic continuations.
\end{remark}

We complete the ongoing proof of Proposition \ref
{originally_cut_mu_lemma} and then return to the proof of the above lemma.

Case (1) is now done since we can estimate the first integral in
equation (\ref{case1case2}) using Lemma~\ref{final_estimate} and the
exponential decay\vadjust{\goodbreak} of the exponential term outside of $|\eta'-\zeta
|=O(\e
^{1/2})$. Therefore, just as with the $A$ operator, the $\e^{-1/2}$
blowup of $|\tilde\mu f(\tilde\mu,\zeta/\eta')|$ is countered by the
decay of the other terms, and we are just left with a large constant
time $|\tilde\mu|^{\alpha}$.

Turning to case (2) we need to show that the second integral in
equation (\ref{case1case2}) is bounded uniformly in $\e$ and $\tilde
\mu
\in\tilde C_{\e}$. This case corresponds to $|\eta'-\zeta|>c$ for some
fixed but small constant $c$. Since $\e^{-1/2}(|\zeta/\eta|-1)$ stays
bounded in a compact set, using an argument almost identical to the
proof of Lemma~\ref{mu_f_polynomial_bound_lemma} we can show that
$|\tilde\mu f(\tilde\mu,\zeta/\eta)|$ can be bounded by $C|\tilde
\mu
|^{C'}$ for positive yet finite constants $C$ and $C'$. The important
point here is that there is only a finite power of $|\tilde\mu|$. Since
$|\tilde\mu|<\e^{-1/2}$ this means that this term can blow up at most
polynomially in $\e^{-1/2}$. On the other hand we know that the
combination of the other terms decay exponentially fast like $e^{-\e
^{-1}c}$ for some small yet finite constant $c$. Hence the second
integral in equation (\ref{case1case2}) goes to zero.\looseness=-1

We now return to the proof of Lemma~\ref{final_estimate}, which will
complete the proof of Proposition~\ref{originally_cut_mu_lemma}.

\begin{pf*}{Proof of Lemma \protect\ref{final_estimate}}
We will prove the desired estimate for $z$ with $|z|=1+\e^{1/2}$. The
proof for general $z\in R_{\e,r}$ follows similarly.
Recall that
\[
\tilde\mu f(\tilde\mu,z) = \sum_{k=-\infty}^{\infty} \frac
{\tilde\mu
\tau^k}{1-\tilde\mu\tau^k} z^k.
\]

We will recenter doubly infinite sum at around the value
\[
k^*=\bigl\lfloor\tfrac{1}{2}\e^{-1/2} \log|\mu|\bigr\rfloor.
\]
This is motivated by the fact that for $\im\tilde\mu=\pm1$, and large
real part, the denominator is approximately minimized when $\tau
^{k}=1/|\tilde\mu|$, corresponding to $k\approx k^*$. This centering
results in
\[
\tilde\mu f(\tilde\mu,z) = \sum_{k=-\infty}^{\infty} \frac
{\tilde\mu
\tau^{k^*}\tau^k}{1-\tilde\mu\tau^{k^*} \tau^k} z^{k^*}z^k.
\]
By the definition of $k^*$,
\[
|z|^{k^*}= |\tilde\mu|^{1/2}\bigl(1+O(\e^{1/2})\bigr).
\]
Thus we find that
\[
|\tilde\mu f(\tilde\mu,z)| = |\tilde\mu|^{1/2}\Biggl|\sum_{k=-\infty
}^{\infty} \frac{\kapi\tau^k}{1-\kapi\tau^k} z^k\Biggr|,
\]
where
\[
\kapi=\tilde\mu\tau^{k^*}
\]
and is roughly on the unit circle except for a small dimple near 1. To
be more precise, due to the rounding in the definition of $k^*$ the
$\kapi$ is not exactly on the unit circle; however we have
\[
|1-\kapi|>\e^{1/2}, \qquad |\kapi|-1=O(\e^{1/2}).\vadjust{\goodbreak}
\]
The section of $\tilde\mathcal{C}_{\e}$ in which $\tilde\mu=\e
^{-1/2}-1+iy$ for $y\in(-1,1)$ corresponds to $\kapi$ lying along a
small dimple around $1$ (and still respects $|1-\kapi|>\e^{1/2}$). We
call the curve on which $\kapi$ lies $\Omega$.

We can bring the $|\tilde\mu|^{1/2}$ factor to the left and split the
summation into three parts, so that $|\tilde\mu|^{-1/2}|\tilde\mu
f(\tilde\mu,z)|$ equals
\[ 
 \Biggl|\sum_{k=-\infty}^{-\e^{-1/2}} \frac{\kapi\tau^k}{1-\kapi\tau
^k} z^k+ \sum_{k=-\e^{-1/2}}^{\e^{-1/2}} \frac{\kapi\tau
^k}{1-\kapi\tau
^k} z^k+ \sum_{k=\e^{-1/2}}^{\infty} \frac{\kapi\tau^k}{1-\kapi
\tau^k}
z^k\Biggr|.\vspace*{-1pt}
\]
We will control each of these term separately. The first and the third
are easiest. Consider
\[
\Biggl|(z-1)\sum_{k=-\infty}^{-\e^{-1/2}} \frac{\kapi\tau^k}{1-\kapi
\tau
^k} z^k\Biggr|.\vspace*{-1pt}
\]
We wish to show this is bounded by a constant which is independent of
$\tilde\mu$ and $\e$. Summing by parts the argument of the absolute
value can be written as\looseness=-1
\[
\frac{\kapi\tau^{-\e^{-1/2}+1}}{1-\kapi\tau^{-\e
^{-1/2}+1}}z^{-\e
^{-1/2}+1}+(1-\tau)\sum_{k=-\infty}^{-\e^{-1/2}}\frac{\kapi\tau
^k}{(1-\kapi\tau^k)(1-\kapi\tau^{k+1})}z^k.\vspace*{-1pt}
\]\looseness=0
We have $\tau^{-\e^{-1/2}+1} \approx e^2$ and $|z^{-\e
^{-1/2}+1}|\approx e^{-1}$ (where $e\sim2.718$). The denominator of
the first term is therefore bounded away from zero. Thus the absolute
value of this term is bounded by a constant. For the second term of
(\ref{three_term_eqn}) we can bring the absolute value inside the summation to get
\[
|1-\tau|\sum_{k=-\infty}^{-\e^{-1/2}}\biggl|\frac{\kapi\tau
^k}{(1-\kapi
\tau^k)(1-\kapi\tau^{k+1})}\biggr||z|^k.\vspace*{-1pt}
\]
The term $|\frac{\kapi\tau^k}{(1-\kapi\tau^k)(1-\kapi\tau
^{k+1})}|$ stays bounded above by a constant times the value at
$k=-\e^{-1/2}$. Therefore, replacing this by a constant, we can sum in
$|z|$ and we get $\frac{|z|^{-\e^{-1/2}}}{1-1/|z|}$. The numerator, as
noted before, is like $e^{-1}$ but the denominator is like $\e
^{1/2}/2$. This is canceled by the term $1-\tau=O(\e^{1/2})$ in front.
Thus the absolute value is bounded.

The argument for the third term of equation (\ref{three_term_eqn})
works in the same way, except rather than multiplying by $|1-z|$ and
showing the result is constant, we multiply by $|1-\tau z|$. This is,
however, sufficient since $|1-\tau z|$ and $|1-z|$ are effectively the
same for $z$ near 1 which is where our desired bound must be shown carefully.

We now turn to the middle term in equation (\ref{three_term_eqn}),
which is more difficult.
We will show that
\[
\Biggl|(1-z)\sum_{k=-\e^{-1/2}}^{\e^{-1/2}} \frac{\kapi\tau^k}{1-\kapi
\tau^k} z^k\Biggr|=O(\log|\tilde\mu|).\vadjust{\goodbreak}
\]
This is of smaller order than $|\tilde\mu|$ raised to any positive real
power and thus finishes the proof. For the sake of simplicity, we will
first show this with $z=1+\e^{1/2}$. The general argument for points
$z$ of the same radius and nonzero angle is very similar, as we will
observe at the end of the proof.
For the special choice of $z$, the prefactor $1-z=\e^{1/2}$.

The idea, as mentioned in the heuristic proof, is to show that this sum
is well approximated by a Riemann sum. In fact, the argument below can
be used to make that formal observation rigorous, and thus provides an
alternative method to the complex analytic approach we take in the
proof of Lemma~\ref{muf_compact_sets_csc_limit_lemma}. The sum we have
is given by
%
%
\begin{eqnarray}\label{three_term_eqn}\qquad
\e^{1/2} \sum_{k=-\e^{-1/2}}^{\e^{-1/2}}\frac{\kapi\tau
^k}{1-\kapi\tau
^k} z^k = \e^{1/2} \sum_{k=-\e^{-1/2}}^{\e^{-1/2}}\frac{\kapi
(1-\e
^{1/2}+O(\e))^k}{1-\kapi(1-2\e^{1/2}+O(\e))^k},
\end{eqnarray}
where we have used the fact that $\tau z = 1-\e^{1/2} + O(\e)$.
If $k=t\e^{-1/2}$ then the sum is close to a Riemann sum for
%
%
\begin{equation}\label{integral_equation_sigma}
 \int_{-1}^{1}\frac{\kapi e^{-t}}{1-\kapi e^{-2t}} \,dt.
\end{equation}
We use this formal relationship to prove that the sum in equation (\ref
{three_term_eqn}) is $O(\log|\tilde\mu|)$. We do this in a few steps.
The first step is to consider the difference between each term in our
sum and the analogous term in a Riemann sum for the integral. After
estimating the difference we show that this can be summed over $k$ and
gives us a finite error. The second step is to estimate the error of
this Riemann sum approximation to the actual integral. The final step
is to note that
\[
\int_{-1}^{1} \frac{\kapi e^{-t}}{1-\kapi e^{-2t}}\,dt \sim|\log
(1-\kapi
)|\sim\log|\tilde\mu|
\]
for $\kapi\in\Omega$ (in particular where $|1-\kapi|>\e^{1/2}$). Hence
it is easy to check that it is smaller than any power of $|\tilde\mu|$.

A single term in the Riemann sum for the integral looks like
$
\e^{1/2} {\kapi e^{-k\e^{1/2}}}/\break {(1-\kapi e^{-2k\e^{1/2}})}
$. Thus we are interested in estimating
%
%
\begin{equation}\label{estimating_eqn}
\e^{1/2}\biggl|\frac{\kapi(1-\e^{1/2}+O(\e))^k}{1-\kapi(1-2\e
^{1/2}+O(\e))^k} - \frac{\kapi e^{-k\e^{1/2}}}{1-\kapi e^{-2k\e^{1/2}}}
\biggr|.
\end{equation}
We claim that there exists $C<\infty$, independent of $\varepsilon$ and
$k$ satisfying \mbox{$k\e^{1/2}\leq1$}, such that the previous line is
bounded above by
%
%
\begin{equation}\label{giac}
 \frac{Ck^2\e^{3/2}}{(1-\kapi+\kapi2k\e^{1/2})}+\frac{Ck^3\e
^{2}}{(1-\kapi+\kapi2k\e^{1/2})^2}.
\end{equation}
To prove that (\ref{estimating_eqn})${}\le{}$(\ref{giac}) we expand the
powers of $k$ and the exponentials. For the numerator and denominator
of the first term inside of the\vadjust{\goodbreak} absolute value in (\ref
{estimating_eqn}), we have $\kapi(1-\e^{1/2}+O(\e))^k= \kapi-\kapi
k\e
^{1/2} + O(k^2\e)$ and
\begin{eqnarray*}
&&1-\kapi\bigl(1-2\e^{1/2}+O(\e)\bigr)^k \\
&&\qquad =  1-\kapi+\kapi2k\e^{1/2}
-\kapi2k^2\e+O(k\e)+O(k^3\e^{3/2})\\
&&\qquad =
(1-\kapi+\kapi2k\e^{1/2})\biggl(1 - \frac{\kapi2k^2\e+O(k\e)+O(k^3\e
^{3/2})}{1-\kapi+\kapi2k \e^{1/2}}\biggr).
\end{eqnarray*}
Using $1/(1-z)=1+z+O(z^2)$ for $|z|<1$, we see that
\begin{eqnarray*}
&&\frac{\kapi(1-\e^{1/2}+O(\e))^k}{1-\kapi(1-2\e^{1/2}+O(\e))^k}\\
&&\qquad=\frac{\kapi-\kapi k\e^{1/2} + O(k^2\e)}{1-\kapi+\kapi
2k\e
^{1/2}}\biggl(1+\frac{\kapi2k^2\e+O(k\e)+O(k^3\e^{3/2})}{1-\kapi
+\kapi2k \e^{1/2}}\biggr)\\
&&\qquad =
\bigl(\kapi-\kapi k \e^{1/2} + O(k^2\e)\bigr)\\
&&\qquad\quad{}\times \bigl(1-\kapi
+\kapi2k \e^{1/2} + \kapi2k^2\e+O(k\e)+O(k^3\e^{3/2})
\bigr)\\
&&\qquad\quad{}/{(1-\kapi+\kapi2k\e^{1/2})^2}.
\end{eqnarray*}
Likewise, the second term from equation (\ref{estimating_eqn}) can be
similarly estimated and shown to be
\begin{eqnarray*}
&&\frac{\kapi e^{-k\e^{1/2}}}{1-\kapi e^{-2k\e^{1/2}}}\\
&&\qquad= \frac{
(\kapi-\kapi k\e^{1/2} + O(k^2\e))(1-\kapi+\kapi2k \e^{1/2}
+ \kapi2k^2\e+O(k^3\e^{3/2}))}{(1-\kapi+\kapi2k\e^{1/2})^2}.
\end{eqnarray*}
Taking the difference of these two terms, and noting the cancellation
of a number of the terms in the numerator, gives (\ref{giac}).

To see that the error in (\ref{giac}) is bounded after the summation
over $k$ in the range $\{-\e^{-1/2},\ldots, \e^{-1/2}\}$, note that
this gives
\begin{eqnarray*}
&&\e^{1/2}\sum_{k=-\e^{-1/2}}^{\e^{1/2}}\frac{(2k\e
^{1/2})^2}{1-\kapi
+\kapi(2k\e^{1/2})}+\frac{(2k\e^{1/2})^3}{(1-\kapi+\kapi(2k\e^{1/2}))^2}
\\
&&\qquad\sim\int_{-1}^{1}\frac{(2t)^2}{1-\kapi+\kapi2t} +\frac
{(2t)^3}{(1-\kapi+\kapi2t)^2} \,dt.
\end{eqnarray*}
The Riemann sums and integrals
are easily shown to be convergent for our $\kapi$ which lies on
$\Omega
$, which is roughly the unit circle, and avoids the point 1 by distance
$\e^{1/2}$.

Having completed this first step, we now must show that the Riemann sum
for the integral in equation (\ref{integral_equation_sigma}) converges
to the integral. This uses the following estimate:
%
%
\begin{eqnarray}\label{riemann_approx_max}
\sum_{k=-\e^{-1/2}}^{\e^{-1/2}} \e^{1/2} \max_{(k-1/2)\e
^{1/2}\leq
t\leq(k+1/2)\e^{1/2}} \biggl| \frac{\kapi e^{-k\e^{1/2}}}{1-\kapi
e^{-2k\e^{1/2}}} - \frac{\kapi e^{-t}}{1-\kapi e^{-2t}}\biggr|\le C.\hspace*{-35pt}
\end{eqnarray}
To show this, observe that for $t\in\e^{1/2}[k-1/2,k+1/2]$, we can
expand the second fraction as
%
%
\begin{equation}\label{sec_frac}
\frac{\kapi e^{-k\e^{1/2}}(1+O(\e^{1/2}))}{1-\kapi e^{-2k\e
^{1/2}}(1-2l\e^{1/2}+O(\e))},
\end{equation}
where $l\in[-1/2,1/2]$.
Factoring the denominator as
%
%
\begin{equation}
(1-\kapi e^{-2k\e^{1/2}})\biggl(1+ \frac{\kapi e^{-2k\e^{1/2}}(2l\e
^{1/2}+O(\e))}{1-\kapi e^{-2k\e^{1/2}}}\biggr),
\end{equation}
we can use $1/(1+z)=1-z+O(z^2)$ (valid since $|1-\kapi e^{-2k\e
^{1/2}}|>\e^{1/2}$ and $|l|\leq1$) to rewrite (\ref{sec_frac}) as
\[
\frac{\kapi e^{-k\e^{1/2}}(1+O(\e^{1/2}))(1- {\kapi e^{-2k\e
^{1/2}}(2l\e^{1/2}+O(\e))}/{(1-\kapi e^{-2k\e^{1/2}})} )}{1-\kapi
e^{-2k\e^{1/2}}}.
\]
Canceling terms in this expression with the terms in the first part of
equation (\ref{riemann_approx_max}), we find that we are left with
terms bounded by
\[
\frac{O(\e^{1/2})}{1-\kapi e^{-2k\e^{1/2}}} + \frac{O(\e
^{1/2})}{(1-\kapi e^{-2k\e^{1/2}})^2}.
\]
These must be summed over $k$ and multiplied by the prefactor $\e
^{1/2}$. Summing over~$k$ we find that these are approximated by the integrals
\[
\e^{1/2}\int_{-1}^{1}\frac{1}{1-\kapi+\kapi2t}\,dt,\qquad \e^{1/2}\int
_{-1}^{1}\frac{1}{(1-\kapi+\kapi2t)^2}\,dt,
\]
where $|1-\kapi|>\e^{1/2}$. The first integral has a logarithmic
singularity at \mbox{$t=0$} which gives $|\log(1-\kapi)|$ which is clearly
bounded by a constant time $|\log\e^{1/2}|$ for \mbox{$\kapi\in\Omega$}. When
multiplied by $\e^{1/2}$ this term is clearly bounded in~$\e$.
Likewise, the second integral diverges like $|1/(1-\kapi)|$ which is
bounded by $\e^{-1/2}$, and again multiplying by the $\e^{1/2}$ factor
in front shows that this term is bounded. This proves the Riemann sum
approximation.

This estimate completes the proof of the desired bound when $z=1+\e
^{1/2}$. The general case of $|z|=1+\e^{1/2}$ is proved along a similar
line by letting $z= 1+\rho\e^{1/2}$ for $\rho$ on a suitably defined
contour such that $z$ lies on the circle of radius $1+\e^{1/2}$. The
prefactor is no longer $\e^{1/2}$ but rather now $\rho\e^{1/2}$ and all
estimates must take into account $\rho$. However, going through this
carefully one finds that the same sort of estimates as above hold, and
hence the theorem is proved in general.\vadjust{\goodbreak}
\end{pf*}

This lemma completes the proof of Proposition
\ref{originally_cut_mu_lemma}.
\end{pf*}

\begin{pf*}{Proof of Proposition \protect\ref{reinclude_mu_lemma}}
We will focus on the growth of the absolute value of the determinant.
Recall that if $K$ is trace class, then $|\det(I+K)|\leq\exp
{\Vert K\Vert _1}$. Furthermore, if $K$ can be factored into the product $K=AB$
where $A$ and $B$ are Hilbert--Schmidt, then $\Vert K\Vert _1\leq
\Vert A\Vert _2\Vert B\Vert _2$. We will demonstrate such a factorization and follow
this approach to control the size of the determinant.

Define $A\dvtx L^2(\tilde\Gamma_{\zeta})\rightarrow L^2(\tilde\Gamma
_{\eta
})$ and $B\dvtx L^2(\tilde\Gamma_{\eta})\rightarrow L^2(\tilde\Gamma
_{\zeta
})$ via the kernels $A(\tilde\zeta,\tilde\eta) = \frac{e^{-|\im
(\tilde
\zeta)|}}{\tilde\zeta-\tilde\eta}$ and
\[
B(\tilde\eta,\tilde\zeta) =  e^{|\im
(\tilde
\zeta)|}e^{-({T}/{3})(\tilde\zeta^3-\tilde\eta^3)+a\tilde z}
2^{1/3}\frac{\pi(-\tilde\mu)^{\tilde z}}{\sin(\pi\tilde z)} \frac
{\Gamma(2^{1/3}\tilde\zeta- 2^{1/3} XT^{-1/3})}{\Gamma
(2^{1/3}\tilde
\eta' - 2^{1/3} XT^{-1/3})},
\]
%
where $\tilde z = 2^{1/3}(\tilde\zeta-\tilde\eta)$. Notice that we have
put the factor $e^{-|\im(\tilde\zeta)|}$ into the $A$ kernel and
removed it from the $B$ contour. The point of this is to help control
the~$A$ kernel, without significantly impacting the norm of the $B$ kernel.

Consider first $\Vert A\Vert _2$ which is given by
\[
\Vert A\Vert _2^2 = \int_{\tilde\Gamma_{\zeta}}\int_{\tilde\Gamma_{\eta}}
\,d\tilde\zeta \,d\tilde\eta\frac{e^{-2|\im(\tilde\zeta)|}}{|\tilde
\zeta
-\tilde\eta|^2}.
\]
The integral in $\tilde\eta$ converges and is independent of $\tilde
\zeta$ (recall that $|\tilde\zeta-\tilde\eta|$ is bounded away from
zero) while the remaining integral in $\tilde\zeta$ is clearly
convergent; it is exponentially small as $\tilde\zeta$ goes away from
zero along $\tilde\Gamma_{\zeta}$. Thus $\Vert A\Vert _{2}<c$ with no
dependence on $\tilde\mu$ at all.

We now turn to computing $\Vert B\Vert _2$. First consider the cubic term
$\tilde\zeta^3$. The contour~$\tilde\Gamma_{\zeta}$ is
parametrized by
$-\frac{c_3}{2} + c_3 i r$ for $r\in(-\infty,\infty)$, that is, a
straight up and down line just to the left of the $y$ axis.\vspace*{1pt} By plugging
this parametrization in and cubing it, we see that $\re(\tilde\zeta^3)$
behaves like $|\im(\tilde\zeta)|^2$. This is crucial; even though our
contours are parallel and only differ horizontally by a small distance,
their relative locations lead to very different behavior for the real
part of their cube. For~$\tilde\eta$ on the right of the $y$ axis, the
real part still grows quadratically, however, with a negative sign.
This is important because this implies that $|e^{-({T}/{3})(\tilde
\zeta^3-\tilde\eta^3)}|$ behaves like the exponential of the real part
of the argument, which is to say, like
\[
e^{-({T}/{3})(|\im(\tilde\zeta)|^2+|\im(\tilde\eta)|^2)}.
\]
Turning to the $\tilde\mu$ term, observe that
\[
|(-\tilde\mu)^{-\tilde z}| = e^{\re[(\log|\tilde\mu|+i\arg
(-\tilde
\mu))(-\re(\tilde z)-i\im(\tilde z))]}
= e^{-\log|\tilde\mu| \re(\tilde z) +\arg(-\tilde\mu)\im(\tilde
z)}.
\]
The $\csc$ term behaves, for large $\im(\tilde z)$, like $e^{-\pi
|\im
(\tilde z)|}$. Finally, we must show that the Gamma functions have
sub-quadratic growth on vertical lines. This follows from Corollary
1.4.4 of~\cite{AAR} which states that for\vadjust{\goodbreak} $x=a+ib$ and $a_1\leq a\leq
a_2$, as $|b|\rightarrow\infty$
%
%
\begin{equation}
|\Gamma(a+ib)| = \sqrt{2\pi} |b|^{a-1/2} e^{-\pi|b|/2} \bigl(1+O(1/|b|)\bigr).
\end{equation}
Putting all these estimates together gives that for $\tilde\zeta$ and
$\tilde\eta$ far from the origin on their respective contours,
$|B(\tilde\eta,\tilde\zeta)|$ behaves like the following product of
exponentials:
\begin{eqnarray*}
&& e^{|\im(\tilde\zeta)|} e^{-({T}/{3})(|\im(\tilde\zeta)|^2+|\im
(\tilde
\eta)|^2)}\\
&&\qquad{}\times  e^{-\log|\tilde\mu|\re(\tilde z) + \arg(-\tilde\mu
)\im
(\tilde z) - \pi|\im(\tilde z)| - \pi2^{1/3}(|\im(\tilde\zeta
)|-|\im
(\tilde\eta')|)}.
\end{eqnarray*}
Now observe that, due to the location of the contours, $-\re(\tilde z)$
is constant and less than one (in fact equal to $1/2$ by our choice of
contours). Therefore we may factor out the term $e^{-\log|\tilde\mu
|\re
(\tilde z)} = |\tilde\mu|^\alpha$ for $\alpha=1/2<1$.

The Hilbert--Schmidt norm of what remains is clearly finite and
independent of $\tilde\mu$. This is just due to the strong exponential
decay from the quadratic terms $-\im(\zeta)^2$ and $-\im(\eta)^2$ in
the exponential. Therefore we find that $\Vert B\Vert _2\leq c|\tilde\mu
|^\alpha
$ for some constant $c$.

This shows that $\Vert K_a^{\csc,\Gamma}\Vert _1$ behaves like $|\tilde\mu
|^{\alpha}$ for $\alpha<1$. Using the bound that $|\det(I+K_a^{\csc
,\Gamma})|\leq e^{\Vert K_a^{\csc,\Gamma}\Vert }$, we find that $|\det
(I+K_a^{\csc,\Gamma})|\leq e^{|\tilde\mu|^{\alpha}}$. Comparing
this to
$e^{-\tilde\mu}$ we have our desired result. Note that the proof also
shows that $K_a^{\csc,\Gamma}$ is trace class.
\end{pf*}

\subsubsection{\texorpdfstring{Proofs from Section \protect\ref{J_to_K_sec}}
{Proofs from Section 3.3}}\label{JK_proofs_sec}
\mbox{}
\begin{pf*}{Proof of Lemma \protect\ref{kill_gamma_2_lemma}}
We may expand $\Psi(\eta)$ into powers of $\e$ with the expression for
$\eta$ in terms of $\kappa(\theta)$ from (\ref{kappa_eqn}) with
$\alpha
=-1/2$ (similarly $1/2$ for the $\zeta$ expansion). Doing this we find
%
%
\begin{eqnarray}\label{psi_real_eqn}
\re(\Psi(\eta)) &=&\e^{-1/2}\biggl( -{\frac14} \e
^{-1/2}T\alpha\cot^2\biggl(\frac{\theta}{2}\biggr) + {\frac18}T
[\alpha+\kappa(\theta)]^2 \cot^2\biggl(\frac{\theta}{2}\biggr) \biggr)
\nonumber
\\[-8pt]
\\[-8pt]
\nonumber
&&{}+ O(1).
\end{eqnarray}
We must show that everything in the parenthesis above is bounded below
by a positive constant times $|\eta-\xi|$ for all $\eta$ which start at
roughly angle $l\e^{1/2}$. Equivalently we can show that the terms in
the parenthesis behave bounded below by a positive constant times $|\pi
-\theta|$, where $\theta$ is the polar angle of $\eta$.

The second part of this expression is clearly positive regardless of
the value of~$\alpha$. What this suggests is that we must show (in
order to also be able to deal with $\alpha=1/2$ corresponding to the
$\zeta$ estimate) that for $\eta$ starting at angle $l\e^{1/2}$ and
going to zero, the first term dominates (if $l$ is large enough).

To see this we first note that since $\alpha=-1/2$, the first term is
clearly positive and dominates for $\theta$ bounded away from $\pi$.
This proves the inequality for any range of $\eta$ with $\theta$
bounded from $\pi$. Now note that for $\theta$ near $\pi$,
\begin{eqnarray*}
\cot^2\biggl(\frac{\theta}{2}\biggr) &\approx&{\frac{1}{4}}(\pi
-\theta
)^2,\qquad
\tan^2\biggl(\frac{\theta}{2}\biggr)\approx{4}(\pi-\theta)^{-2},\\
\log^2\biggl({\frac{2}{1-\cos(\theta)}}\biggr) &\approx&{\frac{1}{16} }(\pi-\theta)^4.
\end{eqnarray*}
We may expand the square in the second term in (\ref{psi_real_eqn}) and
use the above expressions to find that for some suitable constant $C>0$
(which depends on $X$ and $T$ only), we have
\[
\re(\Psi(\eta)) = \e^{-1/2}\bigl(-{\tfrac1{16} }\e^{-1/2}
T\alpha(\pi-\theta)^2 + C(\pi-\theta)^2\bigr) + O(1).
\]
Since, for some constant $c$, $c^{-1} |\xi-\eta| \leq(\pi-\theta
)\leq
c |\xi-\eta|$, the second part of the lemma follows from the above equation.
Now use the fact that $\pi-\theta\geq l\e^{1/2}$ to give
%
%
\begin{equation}\label{eqn_g}
\re(\Psi(\eta)) = \e^{-1/2}\biggl(-{\frac1{16} }lT\alpha
(\pi
-\theta) +{\frac{X^2}{8T}}(\pi-\theta)^2\biggr) + O(1).
\end{equation}
Since $\pi-\theta$ is bounded by $\pi$, we see that taking $l$ large
enough, the first term always dominates for the entire range of $\theta
\in[0,\pi-l\e^{1/2}]$. Therefore since $\alpha=-1/2$, we find that we
have have the desired lower bound in $\e^{-1/2}$ and $|\pi-\theta|$ for
the first part of the lemma.

Turn now to the bound for $\re(\Psi(\zeta))$. In the case of the
$\eta$
contour we took $\alpha=-1/2$; however, since we now are dealing with
the $\zeta$ contour we must take $\alpha=1/2$. This change in the sign
of $\alpha$, and the argument above shows that equation~(\ref{eqn_g})
implies the desired bound for $\re(\Psi(\zeta))$, for $l$ large enough.
\end{pf*}

Before proving Lemma~\ref{mu_f_polynomial_bound_lemma} we record the
following key lemma on the meromorphic extension of $\mu f(\mu,z)$.
Recall that $\mu f(\mu,z)$ has poles at $\mu= \tau^j$, $j\in\Z$.

\begin{lemma}\label{f_functional_eqn_lemma}
For $\mu\neq\tau^j$, $j\in\Z$, $\mu f(\mu,z)$ is analytic in $z$ for
$1<|z|<\tau^{-1}$ and extends analytically to all $z\neq0$ or $\tau^k$
for $k\in\Z$. This extension is given by first writing $\mu f(\mu,z) =
g_+(z)+g_-(z)$ where
\[
g_+(z)=\sum_{k=0}^{\infty} \frac{\mu\tau^kz^k}{1-\tau^k\mu},\qquad
g_-(z)=\sum_{k=1}^{\infty} \frac{\mu\tau^{-k}z^{-k}}{1-\tau
^{-k}\mu},
\]
and where $g_+$ is now defined for $|z|<\tau^{-1}$, and $g_-$ is
defined for $|z|>1$. These functions satisfy the following two
functional equations which imply the analytic continuation:
\[
g_+(z)=\frac{\mu}{1-\tau z}+\mu g_+(\tau z), \qquad g_-(z)=\frac
{1}{1-z}+\frac{1}{\mu}g_-(z/\tau).
\]
By repeating this functional equation, we find that
\begin{eqnarray*}
g_+(z)&=&\sum_{k=1}^{N}\frac{\mu^k}{1-\tau^k z}+\mu^N g_+(\tau^N
z),\\
g_-(z)&=&\sum_{k=0}^{N-1}\frac{\mu^{-k}}{1-\tau^{-k}z}+\mu
^{-N}g_-(z\tau^{-N}).
\end{eqnarray*}
\end{lemma}

\begin{pf}
We prove the $g_+$ functional equation, since the $g_-$ one follows
similarly. Observe that
\begin{eqnarray*}
 g_+(z) & =& \sum_{k=0}^{\infty} \mu(\tau z)^k \biggl( 1+\frac
{1}{1-\mu\tau^k} -1\biggr) \\
&= &\frac{\mu}{1-\tau z}+\sum_{k=0}^{\infty}
\frac{\mu^2\tau^k}{1-\mu\tau^k} (\tau z)^k = \frac{\mu}{1-\tau
z} + \mu
g_+(\tau z),
\end{eqnarray*}
which is the desired relation.
\end{pf}

\begin{pf*}{Proof of Lemma \protect\ref{mu_f_polynomial_bound_lemma}}
Recall that $\tilde\mu$ lies on a compact subset of $\tilde\mathcal{C}$ and hence that $|1-\tilde\mu\tau^k|$ stays bounded from below as
$k$ varies. Also observe that due to our choices of contours for $\eta
'$ and $\zeta$, $|\zeta/\eta'|$ stays bounded in $(1,\tau^{-1})$. Write
$z=\zeta/\eta'$. Split $\tilde\mu f(\tilde\mu,z)$ as $g_+(z)+g_-(z)$
(see Lemma~\ref{f_functional_eqn_lemma} above), and we see that
$g_+(z)$ is bounded by a constant time $1/(1-\tau z)$, and likewise
$g_-(z)$ is bounded by a constant time $1/(1-z)$. Writing this in terms
of $\zeta$ and $\eta'$ again we have our desired upper bound.
\end{pf*}

\begin{pf*}{Proof of Lemma \protect\ref{new_lemma_1}}
Observe that $n_0\approx\frac{1}{4}\e^{-1/2}\log\e$, and hence due to
the choice of the contours for $\eta$ and $\zeta$, $\re(\log(\zeta
/\eta
))=\mathcal{O}(\e^{1/2})$. This implies that this first term diverges
in $\e$ like $\log\e$ which is clearly beaten by $\e^{-1/2}$ (which is
a clear lower bound for the right-hand side for the choice of contours).

Now consider the infinite sum of logarithms. The closer that $1+\tau
^n\zeta$ gets to zero, the worse the sum, so let us assume that the
denominator is smaller than the numerator. Then due to the restriction
on where $\eta$ and $\zeta$ lie on their respective contours, we are
assured of having $|(1+\tau^n \eta)/(1+\tau^n\zeta)|$ bounded above by
a constant times $\e^{-1/2}\tau^n |\zeta-\eta|$. This constant goes to
zero as $l_0$ increases. Summing in~$n$ then gives an upper bound of
$c|\zeta-\eta|\e^{-1}$, and the triangle inequality completes the proof.
\end{pf*}

\begin{pf*}{Proof of Lemma \protect\ref{gratio_limit}}
The discussion from equations (\ref{qgamma_eqn}) to (\ref{gfinaleqn})
is rigorous, and, for $\tilde\eta'$ and $\tilde\zeta$ in the fixed
compact regions of their contours, we have uniform control over the
error in $\e$ and uniform convergence of the $q$-Gamma functions to
standard Gamma functions. This suffices to prove the lemma.\vadjust{\goodbreak}
\end{pf*}

\begin{pf*}{Proof of Lemma \protect\ref{compact_eta_zeta_taylor_lemma}}
Recall that we have defined
\[
m=\frac{1}{2}\biggl[\e^{-1/2}\biggl(-a+\frac{X^2}{2T}\biggr)+\frac{1}{2}t+x\biggr].
\]
The argument now amounts to a Taylor series expansion with control over
the remainder term. Let us start by recording the first four
derivatives of~$\Lambda(\zeta)$.
\begin{eqnarray*}
\Lambda(\zeta) &=& -x\log(1-\zeta)+\frac{t\zeta}{1-\zeta} +
m\log
\zeta,\\
\Lambda'(\zeta) &=& \frac{x}{1-\zeta}+\frac{t}{(1-\zeta)^2}+\frac
{m}{\zeta},\\
\Lambda''(\zeta) &=& \frac{x}{(1-\zeta)^2}+\frac{2t}{(1-\zeta
)^3}-\frac{m}{\zeta^2},\\
\Lambda'''(\zeta) &=& \frac{2x}{(1-\zeta)^3}+\frac{6t}{(1-\zeta
)^4}+\frac{2m}{\zeta^3},\\
\Lambda''''(\zeta) &=& \frac{6x}{(1-\zeta)^4}+\frac{24 t}{(1-\zeta
)^5} -\frac{6m}{\zeta^4}.
\end{eqnarray*}
We Taylor expand $\Psi(\zeta)=\Lambda(\zeta)-\Lambda(\xi)$ around
$\xi$
and then expand in $\e$ as $\e$ goes to zero and find that
\begin{eqnarray*}
\Lambda'(\xi) &=& \frac{a}{2}\e^{-1/2} + O(1),\\
\Lambda''(\xi) &=& O(\e^{-1/2}),\\
\Lambda'''(\xi) &=& \frac{-T}{8} \e^{-3/2} +O(\e^{-1}),\\
\Lambda''''(\xi) &=& O(\e^{-3/2}).
\end{eqnarray*}
A Taylor series remainder estimate shows then that
\begin{eqnarray*}
&&\biggl|\Psi(\zeta) - \biggl[\Lambda'(\xi)(\zeta-\xi
)+{\frac1{2!}}\Lambda''(\xi)(\zeta-\xi)^2 +
{
\frac1{3!}}\Lambda'''(\xi)(\zeta-\xi)^3\biggr]\biggr|\\
&&\qquad \leq\sup_{t\in
B(\xi,|\zeta-\xi|)} {\frac1{4!}}|\Lambda''''(t)|
|\zeta-\xi|^4,
\end{eqnarray*}
where $B(\xi,|\zeta-\xi|)$ denotes the ball around $\xi$ of radius
$|\zeta-\xi|$. Now considering the scaling we have that $\zeta-\xi=
c_3^{-1}\e^{1/2}\tilde\zeta$, so that when we plug this in along with
the estimates on derivatives of $\Lambda$ at $\xi$, we find that the
equation above becomes
\[
\biggl|\Psi(\zeta) - \biggl[2^{1/3}a\tilde\zeta-\frac{T}{3}\tilde\zeta
^3 \biggr]\biggr|=O(\e^{1/2}).
\]
The above estimate therefore proves the desired first claimed result.\vadjust{\goodbreak}

The second result follows readily from $|e^z-e^w|\leq|z-w|\max\{
|e^z|,|e^w|\}$ and the first result, as well as the boundedness of the
limiting integrand.~%
\end{pf*}

\begin{pf*}{Proof of Lemma \protect\ref{muf_compact_sets_csc_limit_lemma}}
Expanding in $\e$ we have that
\[
z=\frac{\xi+c_{3}^{-1}\e^{1/2}\tilde\zeta}{\xi+c_{3}^{-1}\e
^{1/2}\tilde\eta'} = 1-\e^{1/2}\tilde z + O(\e),
\]
where the error is uniform for our range of $\tilde\eta'$ and $\tilde
\zeta$ and where
$
\tilde z = c_{3}^{-1}(\tilde\zeta-\tilde\eta')$.
We now appeal to the functional equation for $f$, explained in Lemma
\ref{f_functional_eqn_lemma}.\vspace*{1pt} We wish to study $\e^{1/2}g_{+}(z)$ and
$\e^{1/2}g_{-}(z)$ as $\e\searrow0$ and show that they converge
uniformly to suitable integrals. First consider the $g_{+}$ case.
Let us, for the moment, assume that $|\tilde\mu|<1$. We know that
$|\tau z|<1$, thus for any $N\geq0$, we have
\[
\e^{1/2} g_{+}(z) = \e^{1/2}\sum_{k=1}^{N} \frac{\tilde\mu
^k}{1-\tau^k
z} + \e^{1/2}\tilde\mu^N g_{+}(\tau^N z).
\]
Since, by assumption, $|\tilde\mu|<1$, the first sum is the partial sum
of a convergent series. Each term may be expanded in $\e$. Noting that
\[
1-\tau^k z = 1-\bigl(1-2\e^{1/2}+O(\e)\bigr)\bigl(1-\e^{1/2}\tilde z +O(\e)\bigr) =
(2k+\tilde z)\e^{1/2} +kO(\e),
\]
we find that
\[
\e^{1/2}\frac{\tilde\mu^k}{1-\tau^k z} = \frac{\tilde\mu
^k}{2k+\tilde
z} + k O(\e^{1/2}).
\]
The last part of the expression for $g_{+}$ is bounded in $\e$, and
thus we end up with the following asymptotics:
\[
\e^{1/2} g_{+}(z) = \sum_{k=1}^{N} \frac{\tilde\mu^k}{2k+\tilde z} +
N^2 O(\e^{1/2}) + \tilde\mu^N O(1).
\]
It is possible to choose $N(\e)$ which goes to infinity, such that $N^2
O(\e^{1/2}) = o(1)$. Then for any fixed compact set contained in $\C
\setminus\{-2,-4,-6,\ldots\}$ we have uniform convergence of this
sequence of analytic functions to some function, which is necessarily
analytic and equals
\[
\sum_{k=1}^{\infty} \frac{\tilde\mu^k}{2k+\tilde z}.
\]
This expansion is valid for $|\tilde\mu|<1$ and for all $\tilde z\in
\C
\setminus\{-2,-4,-6,\ldots\}$.

Likewise for $\e^{1/2}g_{-}(z)$, for $|\tilde\mu|>1$ and for $\tilde
z\in\C\setminus\{-2,-4,-6,\ldots\}$, we have uniform convergence to
the analytic function
\[
\sum_{k=-\infty}^{0} \frac{\tilde\mu^k}{2k+\tilde z}.\vadjust{\goodbreak}
\]

We now introduce the Hurwitz--Lerch transcendental function and relate
some basic properties of it which can be found in~\cite{SC:2001s}.
\[
\Phi(a,s,w) = \sum_{k=0}^{\infty} \frac{a^k}{(w+k)^s}
\]
for $w>0$ real and either $|a|<1$ and $s\in\C$ or $|a|=1$ and $\re
(s)>1$. For $\re(s)>0$ it is possible to analytically extend this
function using the integral formula
\[
\Phi(a,s,w) = \frac{1}{\Gamma(s)} \int_0^{\infty} \frac
{e^{-(w-1)t}}{e^t-a} t^{s-1} \,dt,
\]
where additionally $a\in\C\setminus[1,\infty)$ and $\re(w)>0$.

We can express our series in terms of this function as
\[
\sum_{k=1}^{\infty} \frac{\tilde\mu^k}{2k+\tilde z} = \frac
{1}{2}\tilde
\mu\Phi(\tilde\mu, 1, 1+\tilde z/2),\qquad
\sum_{k=-\infty}^{0} \frac{\tilde\mu^k}{2k-\tilde z} = -\frac
{1}{2}\Phi
(\tilde\mu^{-1}, 1, -\tilde z/2).
\]
These two functions can be analytically continued using the integral
formula onto the same region where $\re(1+\tilde z/2)>0$ and $\re
(-\tilde z/2)>0$, that is, where $\re(\tilde z/2)\in(-1,0)$.
Additionally the analytic continuation is valid for all $\tilde\mu$ not
along $\R^+$.

We wish now to use Vitali's convergence theorem to conclude that
$\tilde
\mu f(\tilde\mu, z)$ converges uniformly for general $\tilde\mu$ to the
sum of these two analytic continuations. In order to do that we need a
priori boundedness of $\e^{1/2}g_+$ and $\e^{1/2}g_-$ for compact
regions of $\tilde\mu$ away from $\R^+$. This can be shown directly as
follows. By assumption on $\tilde\mu$ we have that $|1-\tau^k \tilde
\mu
|>c^{-1}$ for some positive constant $c$. Consider $\e^{1/2}g_+$ first.
\[
|\e^{1/2}g_+(z)| \leq\e^{1/2}\tilde\mu\sum_{k=0}^{\infty} \frac
{|\tau z|^k}{|1-\tau^k \tilde\mu|} \leq c\e^{1/2} \frac{1}{1-|\tau z|}.
\]
We know that $|\tau z|$ is bounded to order $\e^{1/2}$ away from $1$,
and therefore this shows that $|\e^{1/2}g_+(z)|$ has an upper bound
uniform in $\tilde\mu$. We can do a similar computation for $\e
^{1/2}g_-(z)$ and find the same result, this time using that $|z|$ is
bounded to order $\e^{1/2}$ away from $1$.

As a result of this a priori boundedness uniform in $\tilde\mu$, we
have that for compact sets of $\tilde\mu$ away from $\R^+$, uniformly
in $\e$, $\e^{1/2}g_+$ and $\e^{1/2}g_-$ are uniformly bounded as
$\e$
goes to zero. Therefore Vitali's convergence theorem implies that they
converge uniformly to their analytic continuation.

Now observe that
\[
\frac{1}{2}\tilde\mu\Phi(\tilde\mu,1, 1+\tilde z/2) = \frac
{1}{2}\int_0^{\infty} \frac{\tilde\mu e^{-\tilde z t/2}}{e^t-\tilde
\mu
}\,dt,
\]
and
\[
-\frac{1}{2}\Phi(\tilde\mu^{-1}, 1,-\tilde z/2)= -\frac{1}{2}\int
_0^{\infty} \frac{e^{-(-\tilde z/2 -1)t}}{e^{t} -1/\tilde{\mu}}\,dt=
\frac
{1}{2}\int_{-\infty}^{0} \frac{\tilde\mu e^{-\tilde z
t/2}}{e^{t}-\tilde\mu}\,dt.
\]
Therefore we can combine these as a single integral
\[
\frac{1}{2}\int_{-\infty}^{\infty} \frac{\tilde\mu e^{-\tilde z
t/2}}{e^t-\tilde\mu}\,dt = \frac{1}{2}\int_0^{\infty} \frac{\tilde
\mu
s^{-\tilde z/2}}{s-\tilde\mu} \frac{ds}{s}.
\]
The first of the above equations proves the lemma, and for an
alternative expression we use the second of the integrals (which
followed from the change of variables $e^t=s$), and thus, on the region
$\re(\tilde z/2)\in(-1,0)$ this integral converges and equals
\[
{\tfrac{1}{2}}\pi(-\tilde\mu)^{-\tilde z} \csc(\pi
\tilde z/2)
\]
which is analytic for $\tilde\mu\in\C\setminus[0,\infty)$ and for all
$\tilde z\in\C\setminus2\Z$. Therefore it is the analytic
continuation of our asymptotic series.
\end{pf*}

\section{KPZ equation limit of WASEP}\label{BG_sec}
\mbox{}
\begin{pf*}{Proof of Theorem \protect\ref{BG_thm}}
First let us describe in simple terms the dynamics in $T$ of $Z_\e
(T,X)$ defined in (\ref{scaledhgt}). To ease the presentation, we will
now drop all subscripts $\e$. There is a deterministic part, and there
are random jumps. The jumps are at rates
\[
r_-(x)=\varepsilon^{-2}q\bigl(1-\eta(x)\bigr)\eta(x+1)= \tfrac14\eps
^{-2}q\bigl(1-\hat\eta
(x)\bigr)\bigl(1+\hat\eta(x+1)\bigr)
\]
to $e^{-2\lambda}Z$ and
\[
r_+(x)=\varepsilon^{-2}p\eta(x)\bigl(1-\eta(x+1)\bigr)=\tfrac14\eps
^{-2}p\bigl(1+\hat\eta
(x)\bigr)\bigl(1-\hat\eta(x+1)\bigr)
\]
to $e^{2\lambda}Z$, independently at each site $X\in\eps\mathbb Z$. We
write it as follows:
\[
dZ = \Omega Z \,dT + ( e^{-2\lambda}-1) Z \,dM_-
+ ( e^{2\lambda}-1) Z \,dM_+,
\]
where
\[
\Omega= \eps^{-2}\nu+ ( e^{-2\lambda}-1) r_-
+ ( e^{2\lambda}-1)r_+
\]
and $dM_\pm(T,X) = dP_\pm(T,X) -r_\pm(X) \,dT$ where
$P_-(T,X), P_+(T,X)$, $X\in\eps\mathbb{Z}$ are independent Poisson
processes running at
rates $r_-(T,X), r_+(T,X)$. Let
\[
D = 2\sqrt{pq} = 1-\tfrac12 \eps+ \mathcal{O} (\eps^2)
\]
and $\Delta= \Delta_\eps$ be the $\eps\mathbb Z$ Laplacian, $\Delta
f(X) = \eps^{-2}(f(X+\eps) -2f(X) + f(X-\eps))$. We also have
\[
{\tfrac12} D\Delta Z = {\tfrac12} \eps^{-2}
D\bigl( e^{-\lambda\hat\eta(x+1)} -2 + e^{\lambda\hat\eta(x)} \bigr) Z.
\]
The parameters have been carefully chosen so that
\[
\Omega= {\tfrac12} \eps^{-2}D\bigl( e^{-\lambda\hat\eta(x+1)}
-2 + e^{\lambda\hat\eta(x)} \bigr).
\]
%
We can do this because the four cases, corresponding to the four
possibilities for $\hat\eta(x),\hat\eta(x+1)$: $11$, $(-1)(-1)$,
$1(-1)$, $(-1)1$, give four equations in three unknowns:
\begin{eqnarray*}
11 & &\quad {{\tfrac12}} \eps^{-2} D(
e^{-\lambda} -2 + e^{\lambda} )= \eps^{-2}\nu;
\\
(-1)(-1) & &\quad {{\tfrac12}} \eps^{-2} D ( e^{\lambda
} -2 + e^{-\lambda} )= \eps^{-2}\nu;
\\
1(-1) & &\quad {{\tfrac12}} \eps^{-2} D ( e^{\lambda}
-2 + e^{\lambda} )= \eps^{-2}\nu+ ( e^{2\lambda}-1)\e^{-2}p;
\\
(-1)1 & &\quad {{\tfrac12}} \eps^{-2} D( e^{-\lambda}
-2 + e^{-\lambda} )= \eps^{-2}\nu+ ( e^{-2\lambda}-1)\e^{-2}q.
\end{eqnarray*}
But luckily, the first two equations are the same, so it is actually
three equations in three unknowns. This is the G\"artner transformation
\cite{G}. The
solution is $\lambda= \frac12 \log(q/p)$, $D = 2\sqrt{pq}$, $\nu
= p+q - 2\sqrt{pq}$.

Hence~\cite{G,BG},
%
%
\begin{equation}\label{sde7}
dZ_\eps= {{\tfrac12}} D_\eps\Delta_\eps Z_\eps \,dT +
Z_\eps
\,dM_\eps,
\end{equation}
where
\[
dM_\eps(X)= ( e^{-2\lambda_\eps}-1) \,dM_-(X)
+ ( e^{2\lambda_\eps}-1) \,dM_+(X)
\]
are martingales in $T$ with
\[
d\langle M_\eps(X),M_\eps(Y)\rangle= \eps^{-1}\mathbf{1}(X=Y)
b_\eps
\bigl(\tau_{-[\eps^{-1}X]}\eta\bigr) \,dT,
\]
where $\tau_x\eta(y) = \eta(y-x)$ and
\[
b_\eps(\eta)
=1- \hat\eta(1) \hat\eta(0)+ \hat{b}_\eps(\eta),
\]
where
\begin{eqnarray*}
\hat{b}_\eps(\eta) & = & \eps^{-1}\bigl\{ \bigl[ p\bigl(
(e^{-2\lambda_\eps}-1)^2-4\eps\bigr) + q\bigl( (e^{2\lambda_\eps}-1)^2-4\eps\bigr)\bigr]
\\
&&\hspace*{21pt}{} + [q(e^{-2\lambda_\eps}-1)^2-p(e^{2\lambda_\eps}-1)^2 ]\bigl(\hat
\eta
(1)- \hat\eta(0)\bigr) \\
&&\hspace*{22pt}{} - [q(e^{-2\lambda_\eps}-1)^2+p(e^{2\lambda
_\eps
}-1)^2 -\eps]\hat\eta(1) \hat\eta(0)\bigr\}
.
\end{eqnarray*}
Clearly $b_\eps\ge0$. It is easy to check that there is a $C<\infty$
such that
%
%
\begin{equation}
\label{ppp}0\le\hat{b}_\eps\le C\eps^{1/2}
\end{equation}
and, for sufficiently small $\eps>0$,
\[
0\le b_\eps\le3.
\]
We have the following bound on the initial data. For each $p=1,2,\ldots
$ there exists $C_p<\infty$ such that for all $X\in\mathbb{R}$,
%
%
\begin{equation}\label{init}
E[ Z^{p}_\ep(0,X)]\le e^{C_p X}.
\end{equation}
For any $\delta>0$ let $\mathscr{P}^\delta_\e$ denote the
distribution of
$Z_\ep(T,X)$, $T\in[\delta,\infty)$, as measure on $D[[\delta
,\infty),
C(\mathbb{R})]$ where $D$ means the right continuous paths with left
limits, with the topology of uniform convergence on compact sets.
In~\cite{BG}, Section 4, it is shown that if (\ref{init}) holds, then,
for any $\delta>0$, $\mathscr{P}^\delta_\e$, $\e>0$, are tight. The
limit points are consistent as $\delta\searrow0$, and from the
integral version of (\ref{sde7}),
\begin{eqnarray*}
{Z}_\eps(T,X) & = & \eps\sum_{Y\in\eps\mathbb{Z}} p_\eps(T,X-Y)
Z_\eps(0,Y)
\\
&&{} + \int_0^T \eps\sum_{Y\in\eps\mathbb{Z}} p_\eps
(T-S,X-Y) Z_\eps(S,Y)\,dM_\eps(S,Y),
\end{eqnarray*}
where $p_\eps(T,X) $ are the\vspace*{-1pt} transition probabilities for the
continuous time random
walk with generator $ {{\frac12}}D \Delta_\eps$,
normalized so that
$
p_\eps(T,X) \to p(T,X) = \frac{ e^{-X^2/ 2T} }{\sqrt{2\pi T}}
$, and we have
\[
\lim_{T\rightarrow0} \lim_{\e\to0} E\biggl[ \biggl( Z_\eps(T,X)-\eps\sum
_{Y\in\eps\mathbb{Z}} p_\eps(T,X-Y) Z_\eps(0,Y)\biggr)^2 \biggr]= 0
\]
so that the initial data (\ref{half_brownian}) hold under the limit
$\mathscr{P}$. Finally, we need to identify the limit of
the martingale terms. Recall the key estimate in~\cite{BG} which,
translated to our context, says that for any $0<\delta<T_0<\infty$ and
$\rho>0$,
there are $C_1,C_2>0$ such that for all $\delta\le S< T\le T_0$ and
$\e>0$,
%
%
\begin{eqnarray}\label{gradbd}
&&E \bigl[\bigl| E\bigl[ \bigl(Z_\e(T, X+\e)-Z_\e(T, X)\bigr) \bigl(Z_\e(T, X)-Z_\e(T, X-\e)\bigr) |
\mathscr
{F}(S)\bigr]\bigr| \bigr]
\nonumber
\\[-8pt]
\\[-8pt]
\nonumber
&&\qquad\le\e^{1/2 -\rho} |T-S|^{-1/2} e^{ a |X| }.
\end{eqnarray}
Again, this is only proved using (\ref{init}), and extends without
change to the present context. Let us briefly recall why such a thing
is true. It is well known in the theory of stochastic partial
differential equations that the solutions of a stochastic heat equation
will be H\"older $1/2 -\rho$ in space, for any $\rho>0$. This is proved
using only the integral version of the equation and
$L^p$ bounds on the initial data. Hence it extends in a standard way to
a discretization such as (\ref{sde7}) of such an equation,
as long as we have (\ref{init}), with constants independent of $\e$.

Let $\varphi$ be a smooth test function on $\mathbb{R}$. We hope to
show that under $\mathscr{P}$,
\[
N_T(\varphi):= \int_{\mathbb{R}} \varphi(X) Z(T,X) \,dX -
{
\frac12} \int_0^T \int_{\mathbb{R}}\varphi''(X) Z(S,X) \,dX \,dS
\]
and
%
%
\begin{equation}\label{Lambdamart}
\Lambda_T(\varphi):=N_T(\varphi)^2 - \frac{1}{2} \int_0^T \int
_{\mathbb{R}}\varphi^2 (X) Z^2(S,X) \,dX \,dS
\end{equation}
are local martingales. Because we also have $E[Z^2(T,X)] \le e^{C|X|}$
for all $T>0$, we have uniqueness of the corresponding martingale\vadjust{\goodbreak}
problem, following Section 5 of~\cite{BG} and Theorem 2.2 of~\cite{BC}.
That $N_T$ is a local martingale under $\mathscr{P}$ follows from
(\ref
{sde7}) which says that microscopically,
\[
N_{T,\e}(\varphi):= \int_{\mathbb{R}} \varphi(X) Z_\e
(T,X) \,dX - \frac{1}{2} \int_0^T \int_{\mathbb{R}}D_{\e}\Delta_\e
\varphi
(X) Z_\e(S,X) \,dX \,dS
\]
is a martingale under $\mathscr{P}_\e$. The key point is to identify
the quadratic variation, that is, the martingale $\Lambda_T(\varphi)$.
Microscopically we have that
\[
\Lambda_{T,\e}(\varphi):=N_{T,\e}(\varphi)^2
- {\frac12} \int_0^T \e\sum_{\e\mathbb{Z}}\varphi^2 (X)
b_\e(S,X)Z_\e^2(S,X) \,dX \,dS.
\]
Following the argument in Section 4 of~\cite{BG},
shows that (\ref{gradbd}) suffices to prove that
\[
{\frac12} \int_0^T \e\sum_{\e\mathbb{Z}}\varphi^2
(X)\hat
\eta_\e(S,X+\e)\hat\eta_\e(S,X)Z_\e^2(S,X) \,dX \,dS \to0
\]
in $\mathscr{P}_\e$ probability. Together with (\ref{ppp}) this shows
that $\Lambda_T(\varphi)$ defined in (\ref{Lambdamart}) is
a martingale under $\mathscr{P}$, which completes the proof.
\end{pf*}

\section{Manipulations and asymptotics of the edge crossover
distributions}\label{manipSec}
We now consider various asymptotics and properties of the edge
crossover distributions.

\subsection{\texorpdfstring{Large $T$ asymptotics of the edge crossover distributions (proof of Corollary~\protect\ref{cor1})}
{Large T asymptotics of the edge crossover distributions (proof of Corollary 5)}}\label{BBP_proof_sec}
The proof of Corollary~\ref{cor1} proceeds similarly to that of
Corollary 3 of~\cite{ACQ}.
The first step is to cut the $\tilde\mu$ contour off outside of a
compact region around the origin. Proposition 18 of~\cite{ACQ} (with
the modifications explained in Section~\ref{thm_proof_sec}) shows that
for a fixed $T$, the tail of the $\tilde\mu$ integrand is exponentially
decaying in $\tilde\mu$, and a quick inspection of the proof shows that
increasing $T$ just serves to speed up this decay. This shows that we
can cut off the infinite tails of the $\tilde\mathcal{C}$ contour at
cost which goes to zero as the cut occurs further out.

From now on we may assume that $\tilde\mu$ lies on a compact region
along $\tilde\mathcal{C}$. If we plug in the scalings for space as
$2^{1/3}T^{2/3}X$ and fluctuations as $2^{-1/3}T^{1/3}s$ and make the
change of variables that $\tilde\zeta=T^{-1/3}\zeta,\tilde\eta
=T^{-1/3}\eta,$ and $\tilde\eta'=T^{-1/3}\eta'$, then we find that the
integrand in the kernel for $K^{\mathrm{edge}}_{a}$ can be written as
%
%
\begin{eqnarray}\label{scaled_kernel_eqn}
\qquad&&\exp\biggl\{-\frac{1}{3}(\zeta^3-\eta'^3)+\bigl(s+X^2)(\zeta-\eta')\biggr\} \frac
{\pi
2^{1/3}T^{-1/3}(-\tilde\mu)^{-2^{1/3}T^{-1/3}(\zeta-\eta')}}{\sin
(\pi
2^{1/3}T^{-1/3}(\zeta-\eta'))}
\nonumber
\\[-8pt]
\\[-8pt]
\nonumber
&&\qquad{}\times\frac{\Gamma(2^{1/3}T^{-1/3}(\zeta-
X))}{\Gamma(2^{1/3}T^{-1/3}(\eta' - X))} \frac{d\zeta}{\zeta-\eta}.
\end{eqnarray}
The change of variables scales the contours $\tilde\Gamma_{\zeta}$ and
$\tilde\Gamma_{\eta}$ by a factor of $T^{1/3}$. These contours,
however, can be deformed back to their original form by a combination
of Cauchy's theorem and Proposition 1 of~\cite{TW3},which says that as
long as we do not pass any poles in the kernel, we can deform the
contours on which an operator acts without changing the value of the
Fredholm determinant.

Let $\Gamma_{\zeta}$ and $\Gamma_{\eta}$ be these rescaled contours.
The only requirement on these contours is that they look like those
given in Figure~\ref{new_gamma_contours} and that they both go to the
right of the right most pole of the Gamma functions which occurs at $X$.

We now claim that with the scalings described above $\det(I-K^{\mathrm
{edge}}_{a})_{L^2(\tilde\Gamma_{\eta})}$ converges, uniformly in
$\tilde
\mu$, to $\det(I-K_{a})_{L^2(\Gamma_{\eta})}$ where $K_{a}$ has kernel
$K_{a}(\eta,\eta')$ given by
%
%
\begin{eqnarray}\label{Ka_eqn}
K_{a}(\eta,\eta') &=& \int_{\Gamma_{\zeta}}\frac{\exp\{-
({1}/{3})(\zeta
^3-\eta'^3)+s(\zeta-\eta')+X^2(\zeta-\eta')\}}{(\zeta-\eta
)(\zeta-\eta
')}
\nonumber
\\[-8pt]
\\[-8pt]
\nonumber
&&\hspace*{4pt}\quad{}\times\frac{\eta'-X}{\zeta-X}\,d\zeta.
\end{eqnarray}
The claim follows exactly as in the proof of Corollary 3 of~\cite{ACQ}
and relies on the fact that in the scalings present in (\ref
{scaled_kernel_eqn}), the first fraction converges for compact sets of~$\zeta$
and $\eta'$ to $1/(\zeta-\eta')$, while the second fraction
converges to $(\eta'-X)/(\zeta-X)$. Convergence for compact sets of
$\zeta$ and $\eta'$ is enough since the exponentials provide sufficient
decay for the necessary trace class convergence of operators.

Now we use the method of~\cite{TW3} to factor (\ref{Ka_eqn}) into a
product of three operators $\operatorname{ABC}$ and then reorder as $\operatorname{BCA}$ without
changing the value of the determinant. Observe that given our choice of
contours, $\re(\zeta-\eta')<0$ and hence
\[
\frac{\exp\{s(\zeta-\eta')\}}{\zeta-\eta'} = \int_{s}^{\infty}
\exp\{
x(\zeta-\eta')\}\,dx.
\]
Inserting this into (\ref{Ka_eqn}) we find that $K_a=\operatorname{ABC}$ where
$A\dvtx L^2(s,\infty)\rightarrow L^2(\Gamma_{\eta})$, $B\dvtx L^2(\Gamma
_{\zeta
})\rightarrow L^2(s,\infty)$ and $C\dvtx L^2(\Gamma_{\eta})\rightarrow
L^2(\Gamma_{\zeta})$ and are given by their kernels
%
%
\begin{eqnarray}
A(\eta,x)&=& \exp\biggl\{\frac{1}{3}\eta^3-(x+X^2)\eta\biggr\}(\eta-X),\\
B(x,\zeta)&=& \exp\biggl\{\frac{1}{3}\zeta^3-(x+X^2)\zeta\biggr\},\\
C(\zeta,\eta)&=& \frac{1}{(\zeta-\eta)(\zeta-X)}.
\end{eqnarray}
Reordering does not change the value of the determinant, and we are
left with an operator $\operatorname{BCA}(x,y)$ acting on $L^2(s,\infty)$ with kernel
\begin{eqnarray*}
\operatorname{BCA}(x,y) &=& -\int_{\Gamma_{\zeta}}\int_{\Gamma_{\eta}} \frac{\exp
\{
-\zeta^3/3+\eta^3/3+y\zeta-x\eta+X^2(\zeta-\eta)\}}{\zeta-\eta
}\\
&&\phantom{-\int_{\Gamma_{\zeta}}\int_{\Gamma_{\eta}}}{}\times\frac
{\eta-X}{\zeta-X}\,d\eta \,d\zeta.
\end{eqnarray*}
Shifting the $x$ and $y$ contours by $X^2$ and using
$
\frac{\eta-X}{\zeta-X}= 1+ \frac{\eta-\zeta}{\zeta-X},
$
we have an operator $\operatorname{BCA}$ acting on $L_2(s-X^2,\infty)$ with kernel
\[
\operatorname{BCA}(x,y) = -\int_{\Gamma_{\zeta}}\int_{\Gamma_{\eta}} \frac{\exp
\{
-\zeta^3/3+\eta^3/3+y\zeta-x\eta\}}{\zeta-\eta}\biggl(1+ \frac{\eta
-\zeta
}{\zeta-X}\biggr)\,d\eta \,d\zeta.
\]
Expanding the multiplication, the first term corresponds to the
$\mathrm{Airy_2}$ kernel,
\[
K_{A_2}(x,y)= \int_0^\infty\Ai(t+x)\Ai(t+y)\,dt.
\]
The second term is
\[
\int_{\Gamma_{\zeta}}\int_{\Gamma_{\eta}} \frac{\exp\{-\zeta
^3/3+\eta
^3/3+y\zeta-x\eta\}} {\zeta-X}\,d\eta \,d\zeta.
\]
We can factor this into the $\eta$ and $\zeta$ integrals separately.
The $\eta$ integral gives $\Ai(x)$. The $\zeta$ integral can be
evaluated as follows. First recall that due to the dimple in the
definition of $\tilde\Gamma_{\zeta}$, $\zeta-X$ is on a contour which
lies to the right of the origin. Make the change variables to let
$Z=\zeta-X$, which gives for the $\zeta$ integral,
\[
e^{-X^3/3+Xy}\int dZ e^{-Z^3/3-bZ^2+cZ}\frac{1}{Z},
\]
where $b=X$ and $c=-X^2+y$.
We now appeal to Lemma 31(B) of Appendix (A) of~\cite{BFP} which
states that [recall we have absorbed factors of $(2\pi i)^{-1}$ into
our $dZ$]
%
%
\begin{equation}\label{BFP_Appendix_A}
\int dZ e^{-Z^3/3-bZ^2+cZ}\frac{1}{Z} = -e^{-({2}/{3})b^3-bc}\int
_0^{\infty} \,dt \Ai(b^2+c+t)e^{-bt}.
\end{equation}
The integral in the above equation, however, is over a contour which is
to the left of the origin, where as our integral is to the right. This
is easily fixed by deforming through the pole at $Z=0$ which gives
\[
e^{-X^3/3+Xy}\biggl(1+\int dZ e^{-Z^3/3-bZ^2+cZ}\frac{1}{Z}\biggr),
\]
where the $Z$ integral is now to the left of the origin. Applying (\ref
{BFP_Appendix_A}) we are left with
\begin{eqnarray*}
&&e^{-X^3/3+Xy}\biggl(1- e^{X^3/3-Xy}\int_0^\infty \,dt \Ai(y+t)e^{-Xt}
\biggr) \\
&&\qquad= e^{-X^3/3+Xy}-\int_0^\infty \,dt \Ai(y+t)e^{-Xt}.
\end{eqnarray*}
Thus the final kernel is
\[
K_{\mathcal{A}_{2}}(x,y) + \Ai(x)\biggl(e^{-X^3/3+Xy}-\int_0^\infty \,dt
\Ai(y+t)e^{-Xt}\biggr),
\]
which one can recognize from Definition~\ref{cor_and_conj_definition}
of Section~\ref{definition_sec} as
\[
K_{\mathcal{A}_{2\to\mathrm{BM}}}(X,x;X,y).
\]

\subsection{Alternative forms for the edge crossover
distributions}\label{kernel_manip_sec}
In this section we develop an alternative formula for the edge
crossover distribution. Our starting point is equation (\ref
{kcscgamma}) for $K_{a}^{\csc,\Gamma}$. This can be transformed into
$K_{s}^{\mathrm{edge}}$\vspace*{1pt} by taking $a=2^{-1/3}T^{1/3}s$ and
$X=2^{1/3}T^{2/3}X'$. We will stick to the original form, however.
Recalling the equation, we have a kernel
\begin{eqnarray*}
&&\int_{\tilde\Gamma_{\zeta}} \exp\biggl\{-\frac{T}{3}(\tilde\zeta
^3-\tilde
\eta'^3)+2^{1/3}a(\tilde\zeta-\tilde\eta')\biggr\}2^{1/3}\biggl(\int
_{-\infty}^{\infty} \frac{\tilde\mu e^{-2^{1/3}t(\tilde\zeta-
\tilde
\eta')}}{ e^t - \tilde\mu}\,dt \biggr)\\
&&\qquad{}\times \frac{\Gamma(2^{1/3}\tilde\zeta-
{X}/{T})}{\Gamma(2^{1/3}\tilde\eta' - {X}/{T})}\frac
{d\tilde
\zeta}{\tilde\zeta-\tilde\eta}.
\end{eqnarray*}
For $\re(z)<0$ we have
\[
\int_a^{\infty} e^{xz} \,dx= -\frac{e^{az}}{z},
\]
which, noting that $\re(\tilde\zeta-\tilde\eta)<0$, we may apply
to the
above kernel to get
\begin{eqnarray*}
&&\hspace*{-5pt}-2^{2/3}\int_{\tilde\Gamma_{\zeta}}\int_{-\infty}^{\infty}\int
_{a}^{\infty} \!\!\exp\biggl\{-\frac{T}{3}(\tilde\zeta^3-\tilde\eta
'^3)-2^{1/3}a\tilde\eta'\biggr\} \frac{\tilde\mu e^{-2^{1/3}t(\tilde
\zeta- \tilde\eta')}}{ e^t - \tilde\mu} e^{2^{1/3}[(a-x)\tilde
\eta
+x\tilde\zeta] }\\
&&\hspace*{-5pt}\phantom{-2^{2/3}\int_{\tilde\Gamma_{\zeta}}\int_{-\infty}^{\infty}\int
_{a}^{\infty} }{}\times  \frac{\Gamma(2^{1/3}\tilde\zeta-
{X}/{T})}{\Gamma
(2^{1/3}\tilde\eta' - {X}/{T})}\,dx \,dt \,d\tilde\zeta.
\end{eqnarray*}
This kernel can be factored as a product $\operatorname{ABC}$, where
%
%
\begin{eqnarray}
A\dvtx L^2(a,\infty)&\rightarrow& L^2(\tilde\Gamma_{\eta}),\qquad B\dvtx L^2(\tilde
\Gamma_{\zeta})\rightarrow L^2(a,\infty),
\nonumber
\\[-8pt]
\\[-8pt]
\nonumber
 C\dvtx L^2(\tilde\Gamma_{\eta
})&\rightarrow& L^2(\tilde\Gamma_{\zeta}),
\end{eqnarray}
and the operators are given by their kernels
\begin{eqnarray*}
 A(\tilde\eta,x) &=& e^{2^{1/3}(a-x)\tilde\eta},\qquad  B(x,\tilde\zeta
) = e^{2^{1/3}x \tilde\zeta}, \\
C(\tilde\zeta,\tilde\eta) &=& -2^{2/3}\int_{-\infty}^{\infty}
\exp
\biggl\{-\frac{T}{3}(\tilde\zeta^3-\tilde\eta^3)-2^{1/3}a\tilde\eta\biggr\}
\frac{\tilde\mu e^{-2^{1/3}t(\tilde\zeta- \tilde\eta')}}{ e^t -
\tilde\mu}\\
&&\hspace*{27pt}\qquad{}\times\frac{\Gamma(2^{1/3}\tilde\zeta- {X}/{T})}{\Gamma
(2^{1/3}\tilde\eta- {X}/{T})}\,dt.
\end{eqnarray*}
Since $\det(I-\operatorname{ABC}) = \det(I-\operatorname{BCA})$, we consider $\operatorname{BCA}$ acting on
$L^2(a,\infty)$ with kernel
%
\begin{eqnarray*}
&&\int_{-\infty}^{\infty} \frac{\tilde\mu \,dt}{e^{-t} - \tilde\mu}
\biggl\{
2^{2/3}\int_{\Gamma_{\tilde\zeta}}\int_{\Gamma_{\tilde\eta}}
\exp\biggl\{-\frac{T}{3}(\tilde\zeta^3-\tilde\eta^3)+2^{1/3}(x+t)\tilde
\zeta-2^{1/3}(y+t)\tilde\eta\biggr\}\\
&&\hspace*{204pt}\qquad{}\times\frac{\Gamma(2^{1/3}\tilde\zeta-
{X}/{T})}{\Gamma(2^{1/3}\tilde\eta- {X}/{T})}\,d\tilde\eta
\,d\tilde\zeta\biggr\}.
\end{eqnarray*}
Recall the two integral formulas
%
%
\begin{eqnarray}
\Gamma(z) &=& \int_0^\infty s_1^{z-1}e^{-s_1}\,ds_1,\\
\frac1{\Gamma(z)} &=& -\frac{1}{2\pi i}\int_C(-s_2)^{-z} e^{-s_2}\,ds_2,
\end{eqnarray}
where $C$ is counterclockwise from $\infty$ to $\infty$ going around
$\R
_+$ (the branch of the logarithm is cut along $\R_+$). Both equations
holds for $\re(z)>0$ and can be analytically extended using the
functional equation for the Gamma function. We can rewrite
\begin{eqnarray*}
&&\int_{\Gamma_{\tilde\zeta}}\int_{\Gamma_{\tilde\eta}}
\exp\biggl\{-\frac{T}{3}(\tilde\zeta^3-\tilde\eta^3)+2^{1/3}(x+t)\tilde
\zeta-2^{1/3}(y+t)\tilde\eta\biggr\}\\
&&\phantom{\int_{\Gamma_{\tilde\zeta}}\int_{\Gamma_{\tilde\eta}}}{}\times\frac{\Gamma(2^{1/3}\tilde\zeta-
{X}/{T})}{\Gamma(2^{1/3}\tilde\eta- {X}/{T})}\,d\tilde\eta
\,d\tilde\zeta
\\
&&\qquad= -\frac{1}{2\pi i}\int_C\int_0^\infty\int_{\Gamma_{\tilde\zeta
}}\int
_{\Gamma_{\tilde\eta}}
\exp\biggl\{-\frac{T}{3}(\tilde\zeta^3-\tilde\eta^3)+2^{1/3}(x+t +\log
s_1)\tilde\zeta\\
&&\hspace*{175pt}\qquad{}-2^{1/3}\bigl(y+t+\log(-s_2)\bigr)\tilde\eta\biggr\}
\\
&&\hspace*{95pt}\quad\qquad{}\times d\tilde\eta \,d\tilde\zeta s_1^{-1- {X}/{T} }(-s_2)^{-
{X}/{T} }e^{-s_1-s_2}\,ds_1\,ds_2.
\end{eqnarray*}
Using the formula for the Airy function given by
\[
\Ai(r) = \int_{\tilde\Gamma_{\zeta}} \exp\biggl\{-\frac{1}{3}z^3 +rz\biggr\} \,dz,
\]
we find that our kernel equals
\begin{eqnarray*}
&&-\frac{1}{2\pi i}2^{2/3}T^{-2/3}\int_C\int_0^\infty\int_{-\infty
}^{\infty}\frac{\tilde\mu}{\tilde\mu- e^{-t}} \Ai
\bigl(T^{-1/3}2^{1/3}(x+t+\log s_1)\bigr)\\
&&\hspace*{108pt}\qquad{}\times \Ai\bigl(T^{-1/3}2^{1/3}\bigl(y+t+\log
(-s_2)\bigr)\bigr)\,dt
\\
&&\hspace*{108pt}\qquad{}\times s_1^{-1- {X}/{T} }(-s_2)^{- {X}/{T} }e^{-s_1-s_2}\,ds_1\,ds_2.
\end{eqnarray*}
Note that the formula is only making sense as an integral for $X>0$ and
has to be extended by analytic continuation to other values.

We can also write it in compact form as follows.
Define the Gamma transformed Airy function
%
%
\begin{equation}\label{ag}
{\operatorname{Ai}}^{\Gamma}(a,b,c) = \int_{\tilde\Gamma_{\zeta}} \exp
\biggl\{-\frac
{1}{3}z^3 +az\biggr\} \Gamma( bz+c) \,dz
\end{equation}
and the inverse Gamma transformed Airy function
%
%
\begin{equation}\label{aig}
{\operatorname{Ai}}_{\Gamma}(a,b,c) = \int_{\tilde\Gamma_{\eta}} \exp\biggl\{
\frac
{1}{3}z^3 -az\biggr\} \frac{1}{\Gamma( bz+c)} \,dz.
\end{equation}
These are automatically well defined, as long the contours are such
that for $a,b,c$ they avoid the poles of the Gamma function. Plugging
these in we get
%
%
\begin{eqnarray}
&&\int_{\Gamma_{\tilde\zeta}}\int_{\Gamma_{\tilde\eta}}
\exp\biggl\{-\frac{T}{3}(\tilde\zeta^3-\tilde\eta^3)+2^{1/3}(x+t)\tilde
\zeta-2^{1/3}(y+t)\tilde\eta\biggr\}
\nonumber
\\[-8pt]
\\[-8pt]
\nonumber
&&\hspace*{9pt}\qquad{}\times\frac{\Gamma(2^{1/3}\tilde\zeta-
{X}/{T})}{\Gamma(2^{1/3}\tilde\eta- {X}/{T})}\,d\tilde\eta
\,d\tilde\zeta\\
&&\qquad=T^{-2/3} {\operatorname{Ai}}^{\Gamma}\biggl(\kappa_T^{-1}(x+t),\kappa
_T^{-1},- \frac
{X}{T}\biggr) {\operatorname{Ai}}_{\Gamma}\biggl(\kappa_T^{-1}(y+t),\kappa_T^{-1},-
\frac{X}{T}\biggr),
\end{eqnarray}
where $\kappa_T=2^{-1/3}T^{1/3}$. Thus we find
%
%
\begin{eqnarray}\label{manipulated}
K_{X,T,\tilde\mu}(x,y)&=&\kappa_T^{-2}\int_{-\infty}^{\infty} \frac
{\tilde
\mu \,dt}{e^{-t} - \tilde\mu} {\operatorname{Ai}}^{\Gamma}\biggl(\kappa
_T^{-1}(x+t),\kappa
_T^{-1},- \frac{X}{T}\biggr)
\nonumber
\\[-8pt]
\\[-8pt]
\nonumber
&&{}\times{\operatorname{Ai}}_{\Gamma}\biggl(\kappa_T^{-1}(y+t)
,\kappa
_T^{-1},- \frac{X}{T}\biggr).
\end{eqnarray}

\subsection{Airy Gamma asymptotics}

In this section we obtain some asymptotics of the Gamma transformed
Airy functions (\ref{ag}) and
(\ref{aig}) which will prove useful in Section~\ref{upper_tail_sec}. We
start by observing some bounds on Gamma functions. Recall the
functional equation
%
%
\begin{equation}\label{functeqn}
\Gamma(z+1) = z\Gamma(z)
\end{equation}
and the bound (\cite{AS}, equation (6.1.26))
%
%
\begin{equation}\label{absbdd}
|\Gamma(x+iy)| \leq|\Gamma(x)|.\vadjust{\goodbreak}
\end{equation}
The following two asymptotic bounds are standard; see~\cite{AAR}, equations (1.4.3)
and~(1.4.4), respectively. For $\delta>0$ and $|\arg z|\leq\pi-\delta$,
%
%
\begin{equation}\label{stirling}
\Gamma(z) = \sqrt{2\pi} z^{z-1/2}e^{-z} \bigl(1+o(1)\bigr)\qquad \mbox{as }
|z|\to\infty.
\end{equation}

When $z=x+iy$ and $x_1\leq x\leq x_2$ and $|y|\to\infty$, then
%
%
\begin{equation}\label{imgdecay}
|\Gamma(z)| = \sqrt{2\pi} |y|^{x-1/2} e^{-({\pi}/{2}) |y|}
\bigl(1+O(1/|y|)\bigr),
\end{equation}
where the constant implied by $O(1/|y|)$ depends only on $x_1$ and $x_2$.

\begin{lemma}\label{arcboundlemma}
There exists a constant $C>0$ such that for all $r>0$ and all
$z=e^{i\theta}$ for $\theta\in~[-3\pi/4,3\pi/4]$, we have
%
%
\begin{equation}
\Gamma(rz) \leq C\Gamma(r).
\end{equation}
\end{lemma}
\begin{pf}
Consider separately the three cases: (1) when $r\leq1/2$, (2) when
$r\in(1/2,2)$ and (3) and when $r\geq2$. For case (1), when $r\leq
1/2$ we may apply the functional equation (\ref{functeqn}) once, giving
$\Gamma(z) = \Gamma(z+1)/z$. Since $|z|\leq1/2$, it is clear that
$|z+1|\geq1/2$. Therefore, using (\ref{absbdd}) we find that
%
%
\begin{equation}
|\Gamma(z)| \leq\frac{\Gamma(x)}{|z|}\leq\frac{C'}{|z|},
\end{equation}
where $x$ is real and bounded in $[1/2,3/2]$, and hence the bound
$\Gamma(x)\leq C'$ for $C'=\Gamma(3/2)$. Since $\Gamma(r) \approx c/r$
for $r\leq1/2$ it follows that we can express the bound $C'/|z|$ in
terms of $C\Gamma(r)$ for an appropriate constant $C$, as desired.

Turning to case (2), when $r\in(1/2,2)$ we may apply the functional
equation $k$ times so that $1/2\leq\re(z+k) \leq r$. Here $k$ is at
least 1, but at most 3. This shows that
%
%
\begin{equation}\label{kfold}
\Gamma(z) = \frac{\Gamma(z+k)}{z\cdots(z+k-1)}.
\end{equation}
Since $\arg(z)\in(-3\pi/4,3\pi/4)$ it follows that $|z+k|\geq c$
for a
fixed positive constant. Thus, taking absolute values in the above
equation and using this bound, we find
%
%
\begin{equation}
|\Gamma(z)| \leq C|\Gamma(z+k)| \leq C\Gamma(r),
\end{equation}
where $C$ is a fixed constant (bounded by $1/c^3$).

Turning to case (3), when $r\geq2$, we may apply the functional
equation $k$ times so that $r-1\leq\re(z+k)\leq r$. Here $k\geq3$,
and since $\arg(z)\in(-3\pi/4,3\pi/4)$, it follows that $|z+k| \geq
1$. Thus taking absolute values in (\ref{kfold}), we have
%
%
\begin{equation}
|\Gamma(z)| \leq|\Gamma(z+k)| \leq\Gamma(r),
\end{equation}
as desired.
\end{pf}

Here is another, rather weak inequality we will use.
\begin{lemma}\label{1overr}
There exists a constant $C>0$ such that for all $r>0$, we have
%
%
\begin{equation}
|\Gamma(re^{3i\pi/4 })| \leq C/r.
\end{equation}
\end{lemma}

\begin{pf}
For $r\leq2$ this follows immediately from Lemma~\ref{arcboundlemma}.
For $r\geq2$, set $z=re^{3i\pi/4 }$. We may apply the functional
equation (\ref{functeqn}) $k$ times so that $-1/2 \leq\re(z+k)\leq
1/2$. Here $k\geq3$, and since $\arg(z)=3\pi/4$, it follows that
$|z+k| \geq1$. Thus taking absolute values in (\ref{kfold}) we have
%
%
\begin{equation}
|\Gamma(z)| \leq|\Gamma(z+k)|.
\end{equation}
We may now apply Lemma~\ref{imgdecay} with $x_1=-1/2$ and $x_2=1/2$.
This clearly implies the desired decay (actually much stronger decay
than needed).
\end{pf}

\begin{lemma}\label{inversegammabound}
For all constants $c>\pi/2$, there exist $C>0$ such that for all $z$
with $\re(z)\geq0$,
%
%
\begin{equation}\label{gammaimaginarybound}
|{1}/{\Gamma(z)}| \leq C e^{c |z|}.
\end{equation}
\end{lemma}

\begin{pf}
This is established in two parts. For $z$ such that $|\im(z)|<1$, this
follows from functional equation (\ref{functeqn}) and the boundedness
of $1/\Gamma(z)$ for $0\leq\re(z)\leq1$ and $|\im(z)|<1$. Similarly,
for $z$ such that $|\im(z)|\geq1$ we may first consider such $z$ which
also satisfy $0\leq\re(z)\leq1$. By (\ref{imgdecay}), these $z$
satisfy (\ref{gammaimaginarybound}). This bound can be then be extended
to all $z$ with $|\im(z)|\geq1$ and $\re(z)\geq0$ by the functional
equation.
\end{pf}

Recall the definitions (\ref{ag}) and (\ref{aig}) of the Gamma
transformed and inverse Gamma transformed Airy functions.
In the following lemma constants may change line to line.
\begin{lemma}\label{lemma49}
Fix $b>0$. Then:
\begin{enumerate}[(1)]
\item[(1)] There exists a constant $C$ such that for $a\geq0$,
%
%
\begin{equation}\label{lemeqn1}
|{\operatorname{Ai}}^{\Gamma}(a,b,0)| \leq C\bigl((1+|a|)^{-1}\Gamma
\bigl(b(1+|a|)^{-1}\bigr) +b^{-1}\bigr).
\end{equation}
\item[(2)] For any $\e>0$ there exists a constant $C$ such that for all
$a\leq0$,
%
%
\begin{eqnarray}\label{lemeqn2}
&&|{\operatorname{Ai}}^{\Gamma}(a,b,0)|
\nonumber
\\[-8pt]
\\[-8pt]
\nonumber
&&\qquad\leq C\biggl((1+|a|)^{-1}\Gamma
\bigl(b(1+|a|)^{-1}\bigr) + \frac{1}{b(1+|a|)^{3/2}} + (1+|a|)^{\e}b^{-1-\e}
\biggr).
\end{eqnarray}
\item[(3)] There exists a constant $C$ such that for all $a\geq0$,
%
%
\begin{equation}\label{lemeqn3}
|{\operatorname{Ai}}_{\Gamma}(a,b,0)| \leq Ce^{-({2}/{3})a^{3/2}} (1+|a|)^{-1/4}.
\end{equation}
\item[(4)] There exists a constant $C$ such that for all $a\leq0$ and all
$c>\pi/2$,
%
%
\begin{equation}\label{lemeqn4}
|{\operatorname{Ai}}_{\Gamma}(a,b,0)| \leq C e^{c b |a|^{1/2}}.\vadjust{\goodbreak}
\end{equation}
\end{enumerate}
\end{lemma}
\begin{pf}
The lemma is intended to give decay bounds as $|a|$ grows. We can split
consideration up into $|a|\geq1$ and $|a|\leq1$. For $|a|\leq1$
direct inspection of the integrals reveals the above claimed bounds.
What follows, therefore, deals with the $|a|\geq1$ bounds. We will
establish these bounds in terms of their dependence on $|a|$. However,
in order to state the results of the lemma for all $a$, we then make
the modification of replacing $|a|$ by $1+|a|$ which, up to constants,
does not affect the validity of the bounds.

To prove equation (\ref{lemeqn1}) change variables $z\mapsto
a^{1/2}\tilde{z}$ and the deform the image of the contour $\tilde
\Gamma
_{\zeta}$ so that it is given by three pieces: a ray coming from
$e^{-3i\pi/4 }\infty$ to a positively oriented arc of the circle
centered at the origin of radius $a^{-3/2}$ (from angle $-3\pi/4$ to
$3\pi/4$) and finally a ray going toward $e^{3i\pi/4}\infty$. The
integral now is
\[
{\operatorname{Ai}}^{\Gamma}(a,b,0) = a^{1/2}\int\exp\biggl\{-a^{3/2}\biggl(\frac
{1}{3}\tilde
{z}^3 +\tilde{z}\biggr)\biggr\} \Gamma(ba^{1/2}\tilde{z}) \,d\tilde{z}.
\]

Let us first consider the integral along the circular arc. Along this
arc, $\re(-a^{3/2}(\frac{1}{3}\tilde{z}^3 +\tilde{z}))$ is of order 1.
For the Gamma function, we may apply Lemma~\ref{arcboundlemma} to see
that $|\Gamma(ba^{1/2}\tilde{z})|\leq C\Gamma(ba^{-1})$. Since the arc
has length of order $a^{-3/2}$, we find that (recalling the $a^{1/2}$
prefactor above) the contribution of the circular arc is like
$Ca^{-1}\Gamma(ba^{-1})$.


We must consider the two rays, though by symmetry it suffices to
consider just one. The argument of the exponential can be bounded in
real part by $-a^{3/2}\tilde{z}$ and the Gamma can be
bounded (Lemma~\ref{1overr}) by $\frac{C}{ba^{1/2}|\tilde{z}|}$. Thus
the integral along the ray may be bounded above by the following real integral:
\[
a^{1/2}\int_{a^{-3/2}}^{\infty} e^{-ca^{3/2}r} \frac{C
\,dr}{ba^{1/2}r} =
\frac{C'}{b}.
\]
Thus the contributions of the integrals along the rays are like $C/b$,
and hence the bound is established.

To prove equation (\ref{lemeqn2}) replace $a$ by $-\tilde{a}$ and
change variables $z\mapsto\tilde{a}^{1/2}\tilde{z}$, giving
\[
{\operatorname{Ai}}^{\Gamma}(a,b,0) = \tilde{a}^{1/2}\int\exp\biggl\{-\tilde
{a}^{3/2}\biggl(\frac{1}{3}\tilde{z}^3 +\tilde{z}\biggr)\biggr\} \Gamma(b\tilde
{a}^{1/2}\tilde{z}) \,d\tilde{z}.
\]
For the contour in the resulting integral choose the following: a ray
coming from $e^{-3i\pi/4}\infty$ to $-i$, a line segment from $-i$ to
$-\tilde{a}^{-3/2}i$, a positively oriented arc of the circle centered
at the origin of radius $\tilde{a}^{-3/2}$ going from $-\tilde
{a}^{-3/2}i$ to $\tilde{a}^{-3/2}i$, a~line segment from $\tilde
{a}^{-3/2}i$ to $i$ and a ray from $i$ to $e^{3i\pi/4}\infty$.

As before, it is easy to show that the integral on the arc is bounded
by $C\tilde{a}^{-1}\Gamma(b\tilde{a}^{-1})$ for some constant $C>0$.
The integral on the rays can be easily bounded as well. Just as in the
proof of Lemma~\ref{1overr} we may establish that along the ray
%
%
\begin{equation}
\Gamma(b\tilde{a}^{1/2}\tilde{z}) \leq\frac{C}{b\tilde
{a}^{1/2}|\tilde
{z}|}\leq\frac{C}{b\tilde{a}^{1/2}},\vadjust{\goodbreak}
\end{equation}
where the last inequality is from the fact that $|\tilde{z}\geq1$
along the contour. Thus the integral along the ray is bounded by
%
%
\begin{equation}
\frac{C}{b} \int\exp\biggl\{-\tilde{a}^{3/2}\biggl(\frac{1}{3}\tilde{z}^3
+\tilde
{z}\biggr)\biggr\} \,d\tilde{z}.
\end{equation}
This, in turn, can be bounded by the real integral
%
%
\begin{equation}
\frac{C}{b} \int_0^{\infty} \exp\{-c\tilde{a}^{3/2}r\} \,dr \leq
\frac
{C'}{b \tilde{a}^{3/2}}
\end{equation}
for some constant $C'>0$. Thus the integral along the rays are bounded
by $C/(b\tilde{a}^{3/2})$.

It remains to control the integral along the line segments. By symmetry
it suffices to consider just the segment from $\tilde{a}^{-3/2}i$ to
$i$. By the triangle inequality we can consider the norm of the
integrand, and along this contour the exponential term is uniformly of
norm 1. Thus the integral along the line segment is bounded by
%
%
\begin{equation}\label{fulllineinteg}
\tilde{a}^{1/2}\int_{\tilde{a}^{-3/2}i}^{i} |\Gamma(b\tilde
{a}^{1/2}\tilde{z})| \,d\tilde{z} \leq b^{-1}\int_{b\tilde
{a}^{-1}i}^{\infty i} |\Gamma(z)| \,dz.
\end{equation}
The integral
%
%
\begin{equation}\label{riinteg}
\int_{ri}^{\infty i} |\Gamma(z)| \,dz = O(r^{-\e})
\end{equation}
for any $\e$. To check this fact consider the two cases $r\leq1$ and
$r\geq1$. If $r\leq1$, then split the integral (\ref{riinteg}) into
two parts---from $ri$ to $i$ and $i$ to $\infty i$. On the first part,
observe that $|\Gamma(z)|\leq C/|z|$ [as follows from Lemma~\ref
{arcboundlemma} and the small $r$ asymptotics of $\Gamma(r)$]. This
bounding integral with $1/|z|$ can be evaluated to equal $\log(1/r)$
which is bounded by $Cr^{-\e}$ for any $\e$ (here $C$ depends on $\e$
though). The remaining integral from $i$ to $\infty i$ is easily
bounded by a constant by using the asymptotics of (\ref{imgdecay}).
Thus we find the desired bound claimed in (\ref{riinteg}). When $r\geq
1$, the asymptotics of (\ref{imgdecay}) easily yields the desired bound.

Using the estimate of (\ref{riinteg}) we may bound (\ref
{fulllineinteg}) and conclude the desired bound of $C|a|^{\e}b^{-1-\e}$.

To prove equation (\ref{lemeqn3}), change variables $z\mapsto
a^{1/2}\tilde{z}$ and choose the integration contour to be the
following: a ray from $e^{-i\pi/3}\infty$ to $1$ and then a ray from
$1$ to $e^{i\pi/3}\infty$. Along these contours, the reciprocal of the
Gamma function is bounded by a constant [this can be seen by applying
the Stirling's formula asymptotics~(\ref{stirling})], and hence can be
removed, leaving us with the standard asymptotic analysis for the Airy
function; thus follows the formula.

To prove equation (\ref{lemeqn4}), replace $a$ by $-\tilde{a}$ and
change variables $z\mapsto\tilde{a}^{1/2}\tilde{z}$, giving
\[
{\operatorname{Ai}}_{\Gamma}(a,b,0) = \tilde{a}^{1/2}\int\exp\biggl\{\tilde
{a}^{3/2}\biggl(\frac{1}{3}\tilde{z}^3 +\tilde{z}\biggr)\biggr\} \frac{1}{\Gamma
(b\tilde
{a}^{1/2}\tilde{z})} \,d\tilde{z}.
\]
We may choose the integration contour to be the following: a ray from
$e^{-i\pi/4} \infty$ to $-i$, a line segment from $-i$ to $i$ and a ray
from $i$ to $e^{i\pi/4}\infty$. Let us consider the integral along the
line segment. We may bound the reciprocal of the Gamma function on the
imaginary axis by Lemma~\ref{inversegammabound}. From this bound and
the fact that along this line segment, $\re(\frac{1}{3}\tilde{z}^3
+\tilde{z})=0$, one easily bounds the integral by $Ce^{c b\tilde
{a}^{1/2}}$ (note that the $\tilde{a}^{1/2}$ prefactor can be absorbed
into the exponential term through the constant $c$). It remains to
establish a similar upper bound on the integral over the two rays. By
symmetry we can consider only the ray going from $i$ to $e^{i\pi/4}$.
The bound of Lemma~\ref{inversegammabound} is valid as long as $\re
(w)\geq0$, and thus we may substitute it into the integrand. This
yields a bound of (we have absorbed the $\tilde{a}^{1/2}$ as above)
%
%
\begin{equation}
C \int\exp\biggl\{\tilde{a}^{3/2}\biggl(\frac{1}{3}\tilde{z}^3 +\tilde{z}\biggr)\biggr\} e^{c
b \tilde{\alpha}^{1/2}\tilde{z}} \,d\tilde{z},
\end{equation}
with the integral over the ray in question. Due to the rapid decay of
$\re(\frac{1}{3}\tilde{z}^3 +\tilde{z})$ along the ray, it is now easy
to show that this is (just as with the integral along the line segment)
bounded by $Ce^{c b\tilde{a}^{1/2}}$, where still we are assuming
$c>\pi
/2$. Combining these bounds yields the claimed formula (\ref{lemeqn4}).
\end{pf}
%
\subsection{Upper tail of $F_{T,0}^{\mathrm{edge}}$}\label{upper_tail_sec}

Start with the formula
%
%
\begin{equation}\label{kedge}
1- F_{T,0}^{\mathrm{edge}} ( s) =- \int_{\tilde\mathcal{C}}e^{-\tilde\mu
}\frac{d\tilde\mu}{\tilde\mu}[
\det(I-\tilde{K}_{T,\tilde\mu} )- \det I],
\end{equation}
where the determinants are evaluated on $L^2(s,\infty)$, and
\[
\tilde{K}_{T,\tilde\mu}(x,y)=\int_{-\infty}^{\infty} \frac
{\tilde\mu
\,dt}{e^{-\kappa_T t} - \tilde\mu} {\operatorname{Ai}}^{\Gamma}(x+t
,\kappa
_T^{-1},0) {\operatorname{Ai}}_{\Gamma}(y+t,\kappa_T^{-1},0)
\]
is obtained by rescaling $K_{X,T,\tilde\mu}$, given in (\ref{manipulated}).
First of all note that
%
%
\begin{equation}
\det(I-\tilde{K}_{T,\tilde\mu} )= \det(I- A)\qquad \mbox{where }
A=U^{-1} \tilde{K}_{T,\tilde\mu} U
.
\end{equation}
We will use this with
%
%
\begin{equation}
Uf(x) = (x^4+1)^{-1/2} f(x).
\end{equation}
We also make use of the fact that if $A=A_1A_2$ for $A_1$ and $A_2$
Hilbert--Schmidt, then
$A$ is trace-class with
%
%
\begin{equation}\label{trbd}\qquad
|\det(I+A)-\det I |\le\|A\|_1 e^{ \|A\|_1 +1} \le\|A_1\|_2 \|A_2\|_2
e^{\|A_1\|_2 \|A_2\|_2 +1}.
\end{equation}

We factor $A=A_1A_2$ where $A_1\dvtx L^2(\R)\to L^2(s,\infty)$,
$A_2\dvtx L^2(s,\infty)\to L^2(\R)$ are defined by their integral kernels as
%
%
\begin{eqnarray}\label{aone}
A_1(x,t)&=& {\operatorname{Ai}}^{\Gamma}(x+t,\kappa
_T^{-1},0)(x^4+1)^{-1/2}(t^4+1)^{-1/2}
, \\\label{atwo}
A_2(t,y)&=&\frac{\tilde\mu \,dt}{e^{-\kappa_T t} - \tilde\mu}
{\operatorname{Ai}}_{\Gamma}(y+t,\kappa_T^{-1},0)(x^4+1)^{1/2}(t^4+1)^{1/2}.
\end{eqnarray}
We estimate the Hilbert--Schmidt norms with the aid of Lemma \ref
{lemma49}. In what follows, constants are denoted by upper and lower
case c and can change value from line to line. For
$A_1$ we use the bound
%
%
\begin{equation}
|{\operatorname{Ai}}^{\Gamma}(a,b,0)| \leq C |b|^{-3/2}\sqrt{ |a|^2+1},
\end{equation}
which holds for all $a,b$ from Lemma~\ref{lemma49}. We thus have
%
%
\begin{eqnarray}\label{aonebd}
\qquad\| A_1\|^2_{2} &\le& CT\int_{-\infty}^\infty \,dt \int_{-\infty}^\infty \,dx
(|x+t|^2+1)(x^4+1)^{-1}(t^4+1)^{-1} \,dx\,dt
\nonumber
\\[-8pt]
\\[-8pt]
\nonumber
 &\le& CT.
\end{eqnarray}
For $A_2$ we start with the bound
%
%
\begin{equation}
\biggl|\frac{\tilde\mu}{e^{-\kappa_T t} - \tilde\mu} \biggr|\le C|\tilde
{\mu} |(e^{\kappa_T t}\wedge1),
\end{equation}
as follows from the fact that the contour on which $\tilde{\mu}$ lies
is bounded from the positive real axis by a uniform distance. Then we
use (\ref{lemeqn3}) and (\ref{lemeqn4}) to get
\begin{eqnarray*}
\| A_2\|^2_{2} & \le& C |\tilde\mu|^2 \biggl(\int_{s}^\infty \,dx \int
_{-x}^\infty \,dt (e^{2\kappa_T t}\wedge1 ) e^{-4/3 |x+t|^{3/2} }
(x^4+1)(t^4+1) \\
& &\hspace*{32pt} {}+C\int_{s}^\infty \,dx \int^{-x}_{-\infty} \,dt
(e^{2\kappa_T t}\wedge1 )e^{-4\kappa_T^{-1}|x+t|^{1/2}}
(x^4+1)(t^4+1)\biggr) \\
&:=& C|\tilde\mu|^2(I_1 +I_2).
\end{eqnarray*}

Look at $I_2$ first. We are only interested in $s\gg1$, so $-x\le-s
\le0$, and therefore $e^{2\kappa_T t}<1$ so we
have
\begin{eqnarray*}
I_2 &=& \int_{s}^\infty(x^4+1) \,dx \int^{-x}_{-\infty} \,dt
e^{2\kappa_Tt-4\kappa_T^{-1}\sqrt{-x-t}} (t^4+1)\\
&\le&\int_{s}^\infty(x^4+1) \,dx e^{-2\kappa_T x}\bigl( (x^4+1) + C\kappa
_T^{-20}\bigr)e^{(1/2)\kappa_T^{-3}}.
\end{eqnarray*}
Assuming $T\ge T_0>0$, we have
\[
I_2 \le C s^8 e^{-2\kappa_Ts}.
\]
Now for $I_1$, we write it as $I_1=I_3 + I_4$ where $I_3$ denotes the
$t$ integration from $-x$ to $-x/2$, and $I_4$ denotes the $t$
integration from $-x/2$ to $\infty$.
Now
\[
I_3 \le\int_{s}^\infty \,dx \int_{-x}^{-x/2} \,dt e^{2\kappa_T
t}(x^4+1)(t^4+1)\ \le c s^8 e^{-\kappa_T s}
\]
and
\[
I_4 \le\int_{s}^\infty \,dx \int_{-x/2}^{\infty} \,dt e^{-4/3
|x+t|^{3/2} } (x^4+1)(t^4+1)\le c s^8 e^{-c s^{3/2}}.
\]
Note that the $c$ in the last exponent can be computed to be about
$4\cdot3^{-1}\cdot2^{-3/2}$, but it is not optimal,
so we do not pursue it. We obtain
%
%
\begin{equation}\label{atwobd}
\| A_2\|^2_{2} \le C|\tilde\mu|^2 s^8(e^{-2\kappa_T s}+e^{-c s^{3/2}}).
\end{equation}

In summary, combining (\ref{trbd}), (\ref{aonebd}), (\ref{atwobd}) we
obtain a bound
%
%
\begin{equation}\label{}\qquad
|\det(I-\tilde{K}_{T,\tilde\mu} )-\det I|\le C|\tilde\mu| s^4
T^{1/2}(e^{-cT^{1/3} s}+e^{-c s^{3/2}})e^{(1/2)|\tilde\mu|},
\end{equation}
as long as
$C s^4 T^{1/2}(e^{-cT^{1/3} s}+e^{-c s^{3/2}})\le1/2$.
Note that $\bar{c}=| \int_{\tilde\mathcal{C}}e^{-\tilde\mu+
(1/2)|\tilde\mu|}\times |\tilde\mu|\frac{d\tilde\mu}{\tilde\mu}|\le10$.
From (\ref{kedge}), we obtain with a new constant $\tilde{C}=10C$,
%
%
\begin{equation}
|1- F_{T,0}^{\mathrm{edge}} ( s) |\le\tilde{C}s^4 T^{1/2}(e^{-cT^{1/3}
s}+e^{-c s^{3/2}}),
\end{equation}
whenever the right-hand side is less than or equal to $5$. But since
the left-hand side is less than or equal to $1$, the inequality holds
for any $s$.


\subsection{Proof of small $T$ edge crossover distribution
behavior}\label{small_t}
\mbox{}
\begin{pf*}{Proof of Proposition \protect\ref{small_T_prop}}
We have
\begin{eqnarray*}
&&E[ I_n^2(T,X)] \\
&&\qquad=\int_{\Delta'_n(T)} \int_{\mathbb R^{n+1}}\prod_{i=0}^n
p^2(T_{i+1}-T_{i}, X_{i+1}-X_{i}) e^{2X_0}\mathbf{1}_{X_0\ge0}
\,dX_0\prod_{i=1}^{n}\,dT_i \,dX_i.
\end{eqnarray*}
Rescaling $T_i=T\tilde{T}_i$ and $X_i= \sqrt{T}\tilde{X}_i$, we see that
\[
E[ I_n^2(T,X)] \sim C_n T^{{n}/{2}}
\]
with $C_n\sim C\sqrt{n!}$. This shows that as $T\searrow0$,
\[
E\bigl[\bigl(Z(T,X) - I_0(T,X)-I_1(T,X)\bigr)^2\bigr] \le CT.
\]
Note that $I_0(T,X)$ is just the solution of the heat equation with
initial data $Z(0,X)$. Let $\mathscr{F}(0)$ be the $\sigma$-field
generated by $Z(0,X)= \exp\{- B(X)\}\mathbf{1}_{X\ge0}$, $X\ge0$.
Given $\mathscr{F}(0)$,
$I_1(T,X)$ is Gaussian, with mean zero and covariance
%
\begin{eqnarray*}
&&\operatorname{Cov} (I_1(T,X), I_1(T,Y)|\mathscr{F}(0))\\
&&\qquad =
\int_0^T\int_{-\infty}^\infty\biggl( \int_0^\infty
p(T-T_1,X-X_1)p(T_1,X_1-X_0) e^{-B(X_0)}\,dX_0\biggr)\\
&&\qquad\quad{}\times\biggl( \int_0^\infty
p(T-T_1,Y-X_1)p(T_1,X_1-X'_0) e^{-B(X'_0)}\,dX'_0\biggr)\,dX_1\,dT_1
\\
&&\qquad = \int_0^\infty\int_0^\infty T^{-1/2}\Psi(
T^{-1/2}X,T^{-1/2}Y,T^{-1/2}X_0, T^{-1/2}X_0') \\
&&\hspace*{49pt}\qquad{}\times
e^{-B(X_0)-B(X_0')}\,dX_0\,dX_0'.
\end{eqnarray*}
Hence
if we let
\[
\mathcal{Z}^{\mathrm{init}}(T,X)= T^{-1/4}I_1(T,T^{1/2}X),
\]
it has the desired properties.
\end{pf*}

\section*{Acknowledgments}
We would like to thank Percy Deift, Craig Tracy, Harold Widom and
Herbert Spohn for their interest in, and support for this project. IC
wishes to thank G\'{e}rard Ben Arous, Patrik Ferrari and Antonio
Auffinger for discussions on TASEP and KPZ. We thank our first referee
for helpful input and a simplified (and corrected) version of the proof
of Lemma~\ref{deform_mu_to_C}. We also thank William Stanton for
noticing a few errors in an earlier version of this article.

%
%

%


\printaddresses


\begin{thebibliography}{46}

\bibitem{AS}
%
\begin{bbook}[mr]
\bauthor{\bsnm{Abramowitz},~\bfnm{Milton}\binits{M.}} \AND
\bauthor{\bsnm{Stegun},~\bfnm{Irene~A.}\binits{I.~A.}}
(\byear{1964}).
\btitle{Handbook of Mathematical Functions with Formulas, Graphs, and
Mathematical Tables}.
\bseries{National Bureau of Standards Applied Mathematics Series}
\bvolume{55}.
\bpublisher{U.S. Government Printing Office}, \baddress{Washington, DC}.
\bid{mr={0167642}}
\bptnote{check year}%
\bptok{imsref}%
\end{bbook}
%
\endbibitem

\bibitem{AKQ}
%
\begin{barticle}[auto:STB|2012/03/12|15:33:09]
\bauthor{\bsnm{Alberts},~\bfnm{T.}\binits{T.}},
\bauthor{\bsnm{Khanin},~\bfnm{K.}\binits{K.}} \AND
\bauthor{\bsnm{Quastel},~\bfnm{J.}\binits{J.}}
(\byear{2010}).
\btitle{The intermediate disorder regime for directed polymers in dimension
$1+1$}.
\bjournal{Phys. Rev. Lett.}
\bvolume{105}
\bpages{090603}.
\bptok{imsref}%
\end{barticle}
%
\endbibitem

\bibitem{ACQ}
%
\begin{barticle}[mr]
\bauthor{\bsnm{Amir},~\bfnm{Gideon}\binits{G.}},
\bauthor{\bsnm{Corwin},~\bfnm{Ivan}\binits{I.}} \AND
\bauthor{\bsnm{Quastel},~\bfnm{Jeremy}\binits{J.}}
(\byear{2011}).
\btitle{Probability distribution of the free energy of the continuum directed
random polymer in {$1+1$} dimensions}.
\bjournal{Comm. Pure Appl. Math.}
\bvolume{64}
\bpages{466--537}.
\bid{doi={10.1002/cpa.20347}, issn={0010-3640}, mr={2796514}}
\bptok{imsref}%
\end{barticle}
%
\endbibitem

\bibitem{AAR}
%
\begin{bbook}[mr]
\bauthor{\bsnm{Andrews},~\bfnm{George~E.}\binits{G.~E.}},
\bauthor{\bsnm{Askey},~\bfnm{Richard}\binits{R.}} \AND
\bauthor{\bsnm{Roy},~\bfnm{Ranjan}\binits{R.}}
(\byear{1999}).
\btitle{Special Functions}.
\bseries{Encyclopedia of Mathematics and Its Applications}
\bvolume{71}.
\bpublisher{Cambridge Univ. Press}, \baddress{Cambridge}.
\bid{mr={1688958}}
\bptnote{check year}%
\bptok{imsref}%
\end{bbook}
%
\endbibitem

\bibitem{BBP}
%
\begin{barticle}[mr]
\bauthor{\bsnm{Baik},~\bfnm{Jinho}\binits{J.}},
\bauthor{\bsnm{Ben~Arous},~\bfnm{G{\'e}rard}\binits{G.}} \AND
\bauthor{\bsnm{P{\'e}ch{\'e}},~\bfnm{Sandrine}\binits{S.}}
(\byear{2005}).
\btitle{Phase transition of the largest eigenvalue for nonnull complex sample
covariance matrices}.
\bjournal{Ann. Probab.}
\bvolume{33}
\bpages{1643--1697}.
\bid{doi={10.1214/009117905000000233}, issn={0091-1798}, mr={2165575}}
\bptnote{check year}%
\bptok{imsref}%
\end{barticle}
%
\endbibitem

\bibitem{BFP}
%
\begin{barticle}[mr]
\bauthor{\bsnm{Baik},~\bfnm{Jinho}\binits{J.}},
\bauthor{\bsnm{Ferrari},~\bfnm{Patrik~L.}\binits{P.~L.}} \AND
\bauthor{\bsnm{P{\'e}ch{\'e}},~\bfnm{Sandrine}\binits{S.}}
(\byear{2010}).
\btitle{Limit process of stationary {TASEP} near the characteristic line}.
\bjournal{Comm. Pure Appl. Math.}
\bvolume{63}
\bpages{1017--1070}.
\bid{doi={10.1002/cpa.20316}, issn={0010-3640}, mr={2642384}}
\bptok{imsref}%
\end{barticle}
%
\endbibitem

\bibitem{BaikRains}
%
\begin{barticle}[mr]
\bauthor{\bsnm{Baik},~\bfnm{Jinho}\binits{J.}} \AND
\bauthor{\bsnm{Rains},~\bfnm{Eric~M.}\binits{E.~M.}}
(\byear{2000}).
\btitle{Limiting distributions for a polynuclear growth model with external
sources}.
\bjournal{J. Stat. Phys.}
\bvolume{100}
\bpages{523--541}.
\bid{doi={10.1023/A:1018615306992}, issn={0022-4715}, mr={1788477}}
\bptok{imsref}%
\end{barticle}
%
\endbibitem

\bibitem{BQS}
%
\begin{barticle}[mr]
\bauthor{\bsnm{Bal{\'a}zs},~\bfnm{M.}\binits{M.}},
\bauthor{\bsnm{Quastel},~\bfnm{J.}\binits{J.}} \AND
\bauthor{\bsnm{Sepp{\"a}l{\"a}inen},~\bfnm{T.}\binits{T.}}
(\byear{2011}).
\btitle{Fluctuation exponent of the {KPZ}/stochastic {B}urgers equation}.
\bjournal{J. Amer. Math. Soc.}
\bvolume{24}
\bpages{683--708}.
\bid{doi={10.1090/S0894-0347-2011-00692-9}, issn={0894-0347}, mr={2784327}}
\bptok{imsref}%
\end{barticle}
%
\endbibitem

\bibitem{BAC}
%
\begin{barticle}[mr]
\bauthor{\bsnm{Ben~Arous},~\bfnm{G{\'e}rard}\binits{G.}} \AND
\bauthor{\bsnm{Corwin},~\bfnm{Ivan}\binits{I.}}
(\byear{2011}).
\btitle{Current fluctuations for {TASEP}: A proof of the {P}r\"ahofer--{S}pohn
conjecture}.
\bjournal{Ann. Probab.}
\bvolume{39}
\bpages{104--138}.
\bid{doi={10.1214/10-AOP550}, issn={0091-1798}, mr={2778798}}
\bptok{imsref}%
\end{barticle}
%
\endbibitem

\bibitem{BC}
%
\begin{barticle}[mr]
\bauthor{\bsnm{Bertini},~\bfnm{Lorenzo}\binits{L.}} \AND
\bauthor{\bsnm{Cancrini},~\bfnm{Nicoletta}\binits{N.}}
(\byear{1995}).
\btitle{The stochastic heat equation: {F}eynman--{K}ac formula and
intermittence}.
\bjournal{J. Stat. Phys.}
\bvolume{78}
\bpages{1377--1401}.
\bid{doi={10.1007/BF02180136}, issn={0022-4715}, mr={1316109}}
\bptok{imsref}%
\end{barticle}
%
\endbibitem


\bibitem{BG}
%
\begin{barticle}[mr]
\bauthor{\bsnm{Bertini},~\bfnm{Lorenzo}\binits{L.}} \AND
\bauthor{\bsnm{Giacomin},~\bfnm{Giambattista}\binits{G.}}
(\byear{1997}).
\btitle{Stochastic {B}urgers and {KPZ} equations from particle systems}.
\bjournal{Comm. Math. Phys.}
\bvolume{183}
\bpages{571--607}.
\bid{doi={10.1007/s002200050044}, issn={0010-3616}, mr={1462228}}
\bptok{imsref}%
\end{barticle}
%
\endbibitem

\bibitem{BFS09}
%
\begin{barticle}[mr]
\bauthor{\bsnm{Borodin},~\bfnm{Alexei}\binits{A.}},
\bauthor{\bsnm{Ferrari},~\bfnm{Patrik~L.}\binits{P.~L.}} \AND
\bauthor{\bsnm{Sasamoto},~\bfnm{Tomohiro}\binits{T.}}
(\byear{2009}).
\btitle{Two speed {TASEP}}.
\bjournal{J.~Stat. Phys.}
\bvolume{137}
\bpages{936--977}.
\bid{doi={10.1007/s10955-009-9837-7}, issn={0022-4715}, mr={2570757}}
\bptok{imsref}%
\end{barticle}
%
\endbibitem


\bibitem{TC}
%
\begin{barticle}[mr]
\bauthor{\bsnm{Chan},~\bfnm{Terence}\binits{T.}}
(\byear{2000}).
\btitle{Scaling limits of {W}ick ordered {KPZ} equation}.
\bjournal{Comm. Math. Phys.}
\bvolume{209}
\bpages{671--690}.
\bid{issn={0010-3616}, mr={1743612}}
\bptok{imsref}%
\end{barticle}
%
\endbibitem

\bibitem{CFP}
%
\begin{barticle}[mr]
\bauthor{\bsnm{Corwin},~\bfnm{Ivan}\binits{I.}},
\bauthor{\bsnm{Ferrari},~\bfnm{Patrik~L.}\binits{P.~L.}} \AND
\bauthor{\bsnm{P{\'e}ch{\'e}},~\bfnm{Sandrine}\binits{S.}}
(\byear{2010}).
\btitle{Limit processes for {TASEP} with shocks and rarefaction fans}.
\bjournal{J. Stat. Phys.}
\bvolume{140}
\bpages{232--267}.
\bid{doi={10.1007/s10955-010-9995-7}, issn={0022-4715}, mr={2659279}}
\bptok{imsref}%
\end{barticle}
%
\endbibitem

\bibitem{CFP10b}
%
\begin{barticle}[auto:STB|2012/03/12|15:33:09]
\bauthor{\bsnm{Corwin},~\bfnm{I.}\binits{I.}},
\bauthor{\bsnm{Ferrari},~\bfnm{P.~L.}\binits{P.~L.}} \AND
\bauthor{\bsnm{P{\'e}ch{\'e}},~\bfnm{S.}\binits{S.}}
(\byear{2012}).
\btitle{Universality of slow decorrelation in KPZ growth}.
\bjournal{Ann. Inst. Henri Poincar\'{e} Probab. Stat.}
\bvolume{48}
\bpages{134--150}.
\bid{mr={2919201}}
\bptok{imsref}%
\end{barticle}
%
\endbibitem



\bibitem{Erdelyi}
%
\begin{bbook}[mr]
\bauthor{\bsnm{Erd{\'e}lyi},~\bfnm{A.}\binits{A.}}
(\byear{1956}).
\btitle{Asymptotic Expansions}.
\bpublisher{Dover}, \baddress{New York}.
\bid{mr={0078494}}
\bptok{imsref}%
\end{bbook}
%
\endbibitem

\bibitem{FS}
%
\begin{barticle}[mr]
\bauthor{\bsnm{Ferrari},~\bfnm{Patrik~L.}\binits{P.~L.}} \AND
\bauthor{\bsnm{Spohn},~\bfnm{Herbert}\binits{H.}}
(\byear{2006}).
\btitle{Scaling limit for the space--time covariance of the stationary totally
asymmetric simple exclusion process}.
\bjournal{Comm. Math. Phys.}
\bvolume{265}
\bpages{1--44}.
\bid{doi={10.1007/s00220-006-1549-0}, issn={0010-3616}, mr={2217295}}
\bptok{imsref}%
\end{barticle}
%
\endbibitem

\bibitem{G}
%
\begin{barticle}[mr]
\bauthor{\bsnm{G{\"a}rtner},~\bfnm{J{\"u}rgen}\binits{J.}}
(\byear{1988}).
\btitle{Convergence towards {B}urgers' equation and propagation of
chaos for
weakly asymmetric exclusion processes}.
\bjournal{Stochastic Process. Appl.}
\bvolume{27}
\bpages{233--260}.
\bid{doi={10.1016/0304-4149(87)90040-8}, issn={0304-4149}, mr={0931030}}
\bptok{imsref}%
\end{barticle}
%
\endbibitem

\bibitem{GJ}
%
\begin{bmisc}[auto:STB|2012/03/12|15:33:09]
\bauthor{\bsnm{Gol{\c{c}}alves},~\bfnm{P.}\binits{P.}} \AND
\bauthor{\bsnm{Jara},~\bfnm{M.}\binits{M.}}
\bhowpublished{Universality of KPZ equation. Available at
arXiv:\arxivurl{1003.4478}}.
\bptok{imsref}%
\end{bmisc}
%
\endbibitem

\bibitem{SI04}
%
\begin{barticle}[mr]
\bauthor{\bsnm{Imamura},~\bfnm{T.}\binits{T.}} \AND
\bauthor{\bsnm{Sasamoto},~\bfnm{T.}\binits{T.}}
(\byear{2004}).
\btitle{Fluctuations of the one-dimensional polynuclear growth model with
external sources}.
\bjournal{Nuclear Phys. B}
\bvolume{699}
\bpages{503--544}.
\bid{doi={10.1016/j.nuclphysb.2004.07.030}, issn={0550-3213}, mr={2098552}}
\bptok{imsref}%
\end{barticle}
%
\endbibitem

\bibitem{J}
%
\begin{barticle}[mr]
\bauthor{\bsnm{Johansson},~\bfnm{Kurt}\binits{K.}}
(\byear{2000}).
\btitle{Shape fluctuations and random matrices}.
\bjournal{Comm. Math. Phys.}
\bvolume{209}
\bpages{437--476}.
\bid{doi={10.1007/s002200050027}, issn={0010-3616}, mr={1737991}}
\bptok{imsref}%
\end{barticle}
%
\endbibitem

\bibitem{JDTASEP}
%
\begin{barticle}[mr]
\bauthor{\bsnm{Johansson},~\bfnm{Kurt}\binits{K.}}
(\byear{2003}).
\btitle{Discrete polynuclear growth and determinantal processes}.
\bjournal{Comm. Math. Phys.}
\bvolume{242}
\bpages{277--329}.
\bid{issn={0010-3616}, mr={2018275}}
\bptok{imsref}%
\end{barticle}
%
\endbibitem

\bibitem{KPZ}
%
\begin{barticle}[auto:STB|2012/03/12|15:33:09]
\bauthor{\bsnm{Kardar},~\bfnm{K.}\binits{K.}},
\bauthor{\bsnm{Parisi},~\bfnm{G.}\binits{G.}} \AND
\bauthor{\bsnm{Zhang},~\bfnm{Y.~Z.}\binits{Y.~Z.}}
(\byear{1986}).
\btitle{Dynamic scaling of growing interfaces}.
\bjournal{Phys. Rev. Lett.}
\bvolume{56}
\bpages{889--892}.
\bptok{imsref}%
\end{barticle}
%
\endbibitem

\bibitem{KK}
%
\begin{barticle}[auto:STB|2012/03/12|15:33:09]
\bauthor{\bsnm{Kolokolov},~\bfnm{I.~V.}\binits{I.~V.}} \AND
\bauthor{\bsnm{Korshunov},~\bfnm{S.~E.}\binits{S.~E.}}
(\byear{2008}).
\btitle{Universal and non-universal tails of distribution functions in the
directed polymer and KPZ problems}.
\bjournal{Phys. Rev. B}
\bvolume{78}
\bpages{024206}.
\bptok{imsref}%
\end{barticle}
%
\endbibitem

\bibitem{Liggett}
%
\begin{bbook}[mr]
\bauthor{\bsnm{Liggett},~\bfnm{Thomas~M.}\binits{T.~M.}}
(\byear{2005}).
\btitle{Interacting Particle Systems}.
\bpublisher{Springer}, \baddress{Berlin}.
\bid{mr={2108619}}
\bptok{imsref}%
\end{bbook}
%
\endbibitem

\bibitem{PS}
%
\begin{barticle}[mr]
\bauthor{\bsnm{Pr{\"a}hofer},~\bfnm{Michael}\binits{M.}} \AND
\bauthor{\bsnm{Spohn},~\bfnm{Herbert}\binits{H.}}
(\byear{2002}).
\btitle{Scale invariance of the {PNG} droplet and the {A}iry process}.
\bjournal{J. Stat. Phys.}
\bvolume{108}
\bpages{1071--1106}.
\bid{doi={10.1023/A:1019791415147}, issn={0022-4715}, mr={1933446}}
\bptok{imsref}%
\end{barticle}
%
\endbibitem

\bibitem{PS02}
%
\begin{bincollection}[mr]
\bauthor{\bsnm{Pr{\"a}hofer},~\bfnm{Michael}\binits{M.}} \AND
\bauthor{\bsnm{Spohn},~\bfnm{Herbert}\binits{H.}}
(\byear{2002}).
\btitle{Current fluctuations for the totally asymmetric simple exclusion
process}.
In \bbooktitle{In and Out of Equilibrium ({M}ambucaba, 2000)}.
\bseries{Progress in Probability}
\bvolume{51}
\bpages{185--204}.
\bpublisher{Birkh\"auser}, \baddress{Boston, MA}.
\bid{mr={1901953}}
\bptok{imsref}%
\end{bincollection}
%
\endbibitem

\bibitem{RS:book}
%
\begin{bbook}[mr]
\bauthor{\bsnm{Reed},~\bfnm{Michael}\binits{M.}} \AND
\bauthor{\bsnm{Simon},~\bfnm{Barry}\binits{B.}}
(\byear{1978}).
\btitle{Methods of Modern Mathematical Physics. {IV}. {A}nalysis of Operators}.
\bpublisher{Academic Press},
\baddress{New York}.
\bid{mr={0493421}}
\bptok{imsref}%
\end{bbook}
%
\endbibitem

\bibitem{Rez91}
%
\begin{barticle}[mr]
\bauthor{\bsnm{Rezakhanlou},~\bfnm{Fraydoun}\binits{F.}}
(\byear{1991}).
\btitle{Hydrodynamic limit for attractive particle systems on
{$\mathbf{Z}\sp d$}}.
\bjournal{Comm. Math. Phys.}
\bvolume{140}
\bpages{417--448}.
\bid{issn={0010-3616}, mr={1130693}}
\bptok{imsref}%
\end{barticle}
%
\endbibitem

\bibitem{SaSp1}
%
\begin{barticle}[mr]
\bauthor{\bsnm{Sasamoto},~\bfnm{Tomohiro}\binits{T.}} \AND
\bauthor{\bsnm{Spohn},~\bfnm{Herbert}\binits{H.}}
(\byear{2010}).
\btitle{Exact height distributions for the {KPZ} equation with narrow wedge
initial condition}.
\bjournal{Nuclear Phys. B}
\bvolume{834}
\bpages{523--542}.
\bid{doi={10.1016/j.nuclphysb.2010.03.026}, issn={0550-3213}, mr={2628936}}
\bptok{imsref}%
\end{barticle}
%
\endbibitem

\bibitem{SaSp2}
%
\begin{barticle}[auto:STB|2012/03/12|15:33:09]
\bauthor{\bsnm{Sasamoto},~\bfnm{T.}\binits{T.}} \AND
\bauthor{\bsnm{Spohn},~\bfnm{H.}\binits{H.}}
(\byear{2010}).
\btitle{One-dimensional KPZ equation: An exact solution and its universality}.
\bjournal{Phys. Rev. Lett.}
\bvolume{104}
\bpages{23}.
\bptok{imsref}%
\end{barticle}
%
\endbibitem

\bibitem{SaSp3}
%
\begin{barticle}[mr]
\bauthor{\bsnm{Sasamoto},~\bfnm{Tomohiro}\binits{T.}} \AND
\bauthor{\bsnm{Spohn},~\bfnm{Herbert}\binits{H.}}
(\byear{2010}).
\btitle{The crossover regime for the weakly asymmetric simple exclusion
process}.
\bjournal{J. Stat. Phys.}
\bvolume{140}
\bpages{209--231}.
\bid{doi={10.1007/s10955-010-9990-z}, issn={0022-4715}, mr={2659278}}
\bptnote{check year}%
\bptok{imsref}%
\end{barticle}
%
\endbibitem

\bibitem{TS:1998h}
%
\begin{barticle}[mr]
\bauthor{\bsnm{Sepp{\"a}l{\"a}inen},~\bfnm{T.}\binits{T.}}
(\byear{1998}).
\btitle{Hydrodynamic scaling, convex duality and asymptotic shapes of growth
models}.
\bjournal{Markov Process. Related Fields}
\bvolume{4}
\bpages{1--26}.
\bid{issn={1024-2953}, mr={1625007}}
\bptok{imsref}%
\end{barticle}
%
\endbibitem

\bibitem{BS:book}
%
\begin{bbook}[mr]
\bauthor{\bsnm{Simon},~\bfnm{Barry}\binits{B.}}
(\byear{2005}).
\btitle{Trace Ideals and Their Applications},
\bedition{2nd} ed.
\bseries{Mathematical Surveys and Monographs}
\bvolume{120}.
\bpublisher{Amer. Math. Soc.}, \baddress{Providence, RI}.
\bid{mr={2154153}}
\bptok{imsref}%
\end{bbook}
%
\endbibitem

\bibitem{SC:2001s}
%
\begin{bbook}[mr]
\bauthor{\bsnm{Srivastava},~\bfnm{H.~M.}\binits{H.~M.}} \AND
\bauthor{\bsnm{Choi},~\bfnm{Junesang}\binits{J.}}
(\byear{2001}).
\btitle{Series Associated with the Zeta and Related Functions}.
\bpublisher{Kluwer Academic}, \baddress{Dordrecht}.
\bid{mr={1849375}}
\bptok{imsref}%
\end{bbook}
%
\endbibitem

\bibitem{TW0}
%
\begin{barticle}[mr]
\bauthor{\bsnm{Tracy},~\bfnm{Craig~A.}\binits{C.~A.}} \AND
\bauthor{\bsnm{Widom},~\bfnm{Harold}\binits{H.}}
(\byear{1994}).
\btitle{Level-spacing distributions and the {A}iry kernel}.
\bjournal{Comm. Math. Phys.}
\bvolume{159}
\bpages{151--174}.
\bid{issn={0010-3616}, mr={1257246}}
\bptok{imsref}%
\end{barticle}
%
\endbibitem

%

\bibitem{TW3}
%
\begin{barticle}[mr]
\bauthor{\bsnm{Tracy},~\bfnm{Craig~A.}\binits{C.~A.}} \AND
\bauthor{\bsnm{Widom},~\bfnm{Harold}\binits{H.}}
(\byear{2009}).
\btitle{Asymptotics in {ASEP} with step initial condition}.
\bjournal{Comm. Math. Phys.}
\bvolume{290}
\bpages{129--154}.
\bid{doi={10.1007/s00220-009-0761-0}, issn={0010-3616}, mr={2520510}}
\bptok{imsref}%
\end{barticle}
%
\endbibitem

\bibitem{TW4}
%
\begin{barticle}[mr]
\bauthor{\bsnm{Tracy},~\bfnm{Craig~A.}\binits{C.~A.}} \AND
\bauthor{\bsnm{Widom},~\bfnm{Harold}\binits{H.}}
(\byear{2009}).
\btitle{On {ASEP} with step {B}ernoulli initial condition}.
\bjournal{J. Stat. Phys.}
\bvolume{137}
\bpages{825--838}.
\bid{doi={10.1007/s10955-009-9867-1}, issn={0022-4715}, mr={2570751}}
\bptok{imsref}%
\end{barticle}
%
\endbibitem

\bibitem{TW5}
%
\begin{barticle}[mr]
\bauthor{\bsnm{Tracy},~\bfnm{Craig~A.}\binits{C.~A.}} \AND
\bauthor{\bsnm{Widom},~\bfnm{Harold}\binits{H.}}
(\byear{2010}).
\btitle{Formulas for {ASEP} with two-sided {B}ernoulli initial condition}.
\bjournal{J. Stat. Phys.}
\bvolume{140}
\bpages{619--634}.
\bid{doi={10.1007/s10955-010-0013-x}, issn={0022-4715}, mr={2670733}}
\bptok{imsref}%
\end{barticle}
%
\endbibitem


\bibitem{W}
%
\begin{bincollection}[mr]
\bauthor{\bsnm{Walsh},~\bfnm{John~B.}\binits{J.~B.}}
(\byear{1986}).
\btitle{An introduction to stochastic partial differential equations}.
In \bbooktitle{\'{E}cole D'\'et\'e de Probabilit\'es de {S}aint-{F}lour,
{XIV}---1984}.
\bseries{Lecture Notes in Math.}
\bvolume{1180}
\bpages{265--439}.
\bpublisher{Springer}, \baddress{Berlin}.
\bid{mr={0876085}}
\bptok{imsref}%
\end{bincollection}
%
\endbibitem

\end{thebibliography}
\end{document}